\documentclass{amsart}
\usepackage{amssymb,epsfig}
\usepackage{ifpdf}
\usepackage{color}
\usepackage{bm}
\usepackage{bbold}
\usepackage{mathrsfs}

\newtheorem{theorem}{Theorem}[section]
\newtheorem{lemma}[theorem]{Lemma}
\newtheorem{proposition}[theorem]{Proposition}

\newtheorem{algorithm}[theorem]{Algorithm}

\theoremstyle{definition}
\newtheorem{definition}[theorem]{Definition}
\newtheorem{condition}[theorem]{Condition}
\newtheorem{modcondition}{Condition}

\theoremstyle{remark}
\newtheorem{remark}[theorem]{Remark}
\newtheorem{remarks}[theorem]{Remarks}

\newtheorem{notation}[theorem]{Notation}
\newtheorem{example}[theorem]{Example}
\numberwithin{equation}{section}

\newcommand{\eps}{\varepsilon}

\newcommand{\R}{\mathbb R}
\newcommand{\N}{\mathbb N}

\newcommand{\VV}{\mathscr{V}}

\newcommand{\UU}{\mathscr{U}}
\newcommand{\PP}{\mathscr{T}}
\newcommand{\HH}{\mathscr{H}}
\newcommand{\KK}{\mathscr{K}}

\newcommand{\dom}{\mathrm{dom}}
\newcommand{\cF}{\mathcal F}
\newcommand{\cA}{\mathcal A}
\newcommand{\cP}{\mathcal P}
\newcommand{\cO}{\mathcal O}
\newcommand{\cS}{\mathcal S}
\newcommand{\cL}{\mathcal L}
\newcommand{\Lis}{\cL\mathrm{is}}

\newcommand{\tria}{\mathcal T}

\DeclareMathOperator{\ran}{ran}
\DeclareMathOperator{\meas}{meas}
\DeclareMathOperator*{\argmin}{argmin}

\DeclareMathOperator{\supp}{supp}
\DeclareMathOperator{\dist}{dist}
\DeclareMathOperator{\Span}{span}
\DeclareMathOperator{\divv}{div}
\DeclareMathOperator{\curl}{curl}

\DeclareMathOperator{\diam}{diam}
\newcommand{\grad}{\nabla}

\usepackage{scalerel}

\title[Adaptive wavelet method for first order system least squares]{An optimal adaptive wavelet method for first order system least squares}
\thanks{}

\date{\today}

\author{Nikolaos Rekatsinas and Rob Stevenson}
\address{
Korteweg-de Vries Institute for Mathematics,
University of Amsterdam,
P.O. Box 94248,
1090 GE Amsterdam, The Netherlands}
\email{n.rekatsinas@uva.nl, r.p.stevenson@uva.nl}



\begin{document}
\begin{abstract} In this paper, it is shown that {\em any} well-posed 2nd order PDE can be reformulated as a well-posed first order least squares system.
This system will be solved by an adaptive wavelet solver in optimal computational complexity.
The applications that are considered are second order elliptic PDEs with general inhomogeneous boundary conditions, and the stationary Navier-Stokes equations.
\end{abstract}

\maketitle
\section{Introduction}
In this paper, a wavelet method is constructed for the optimal adaptive solution of {\em stationary} PDEs.
We develop a general procedure to write {\em any} well-posed 2nd order PDE as a well-posed first order least squares system.
The (natural)  least squares formulations contain dual norms, that, however, impose no difficulties for a wavelet solver.
The advantages of the first order least squares system formulation are two-fold.

Firstly, regardless of the original problem, the least squares problem is symmetric and positive definite, which opens the possibility to develop an optimal adaptive solver.
The obvious use of the least-squares functional as an a posteriori error estimator, however, is not known to yield a convergent method (see, however, \cite{37.3} for an alternative for Poisson's problem).
As we will see, the use of the (approximate) residual in wavelet coordinates as an a posteriori error estimator does give rise to an optimal adaptive solver.

Secondly, as we will discuss in more detail in the following subsections, the optimal application of a wavelet solver to a first order system reformulation allows for a simpler, and quantitatively more efficient 
approximate residual evaluation than with the standard formulation of second order. Moreover, it applies equally well to semi-linear equations, as e.g. the stationary Navier-Stokes equations, and it applies to wavelets that have only {\em one} vanishing moment.

The approach to apply the wavelet solver to a well-posed first order least squares system 
reformulation
also applies  to {\em time-dependent} PDEs in simultaneous space-time variational formulations,
as parabolic problems or instationary Navier-Stokes equations. 
With those problems, the wavelet bases consist of tensor products of temporal and spatial wavelets. 
Consequently, they require a different procedure for the approximate evaluation of the residual in wavelet coordinates, which will be the topic of a forthcoming work.

\subsection{Adaptive wavelet schemes, and the approximate residual evaluation}
Adaptive wavelet schemes can solve well-posed linear and nonlinear operator equations at the best possible rate allowed by the basis in linear complexity (\cite{45.2,45.25,316,45.26,298,249.92,249.94}). 
Schemes with those properties will be called  {\em optimal}.
The schemes can be applied to PDEs, which we study in this work, as well as to integral equations (\cite{58.1}).

There are two kinds of adaptive wavelet schemes. 
One approach is to apply some convergent iterative method to the infinite system in wavelet coordinates, with decreasing tolerances for the inexact evaluations of residuals (\cite{45.25,45.26}).
These schemes rely on the application of coarsening to achieve optimal rates.

The other approach is to solve a sequence of Galerkin approximations from spans of nested sets of wavelets. The (approximate) residual in wavelet coordinates of the current approximation is used as an a posteriori error estimator to make an optimal selection of the wavelets to be added to form the next set (\cite{45.2}).
With this scheme, that is studied in the current work, the application of coarsening can be avoided (\cite{316,75.36}), and it turns out to be quantitatively more efficient.
This approach is restricted to PDOs whose Fr\'{e}chet derivatives are symmetric and positive definite (compact perturbations can be added though, see \cite{75.365}).
\medskip

A key computational ingredient of both schemes is the approximate evaluation of the residual in wavelet coordinates.
Let us discuss this for a linear operator equation $A u =f$, with, for some separable Hilbert spaces $\HH$ and $\KK$, for convenience over $\R$, $f \in \KK'$  and $A \in \Lis(\HH,\KK')$ (i.e., $A \in \cL(\HH,\KK')$ and $A^{-1} \in \cL(\KK',\HH)$).

Equipping $\HH$ and $\KK$ with {\em Riesz bases} $\Psi^\HH$, $\Psi^\KK$, formally viewed as column vectors, $A u =f$ can be equivalently written as a bi-infinite system of coupled scalar equations ${\bf A} {\bf u}={\bf f}$, where ${\bf f}=f(\Psi^\KK)$ is the infinite `load vector', ${\bf A}=(A \Psi^\HH)(\Psi^\KK)$ is the infinite `stiffness' or system matrix, and $u={\bf u}^\top \Psi^\HH$.

Here we made use of following notations:
\begin{notation}
For countable collections of functions $\Sigma$ and $\Upsilon$, we write $g(\Sigma)=[g(\sigma)]_{\sigma \in \Sigma}$,
$M(\Sigma)(\Upsilon)=[M(\sigma)(\upsilon)]_{\upsilon \in \Upsilon, \sigma \in \Sigma}$, and $\langle \Upsilon, \Sigma \rangle = [\langle \upsilon, \sigma \rangle]_{\upsilon \in \Upsilon, \sigma \in \Sigma}$, assuming $g$, $M$, and $\langle\,,\,\rangle$ are such that the expressions at the right-hand sides are well-defined.

The space of square summable vectors of reals indexed over a countable index set $\vee$ will be denoted as $\ell_2(\vee)$ or simply as $\ell_2$. The norm on this space will be simply denoted as $\|\cdot\|$.
\end{notation}

As a consequence of $A \in \Lis(\HH,\KK')$, we have that ${\bf A} \in \Lis(\ell_2,\ell_2)$. For the moment, let us additionally assume that ${\bf A}$ is symmetric and positive definite, as when $\KK=\HH$, $(Au)(v)= (Av)(u)$ and $(Au)(u) \gtrsim \|u\|_\HH^2$ ($u,v \in \HH$). If this is not the case, then the following can be applied to the normal equations ${\bf A}^\top {\bf A} {\bf u}={\bf A}^\top {\bf f}$.

For the finitely supported approximations ${\bf \tilde{u}}$ to ${\bf u}$ that are generated  inside the adaptive wavelet scheme, the residual ${\bf r}={\bf f} -{\bf A} {\bf \tilde{u}}$ has to be approximated within a sufficiently small relative tolerance.
The resulting scheme has been shown to converge with the best possible rate:
Whenever ${\bf u}$ {\em can} be approximated at rate $s$, i.e. ${\bf u} \in \cA^s$, meaning that for any $N \in \N$ there {\em exists} a vector of length $N$ that approximates ${\bf u}$ within tolerance ${\mathcal O}(N^{-s})$, the approximations produced by the scheme converge with this rate $s$.
Moreover, the scheme has {\em linear computational complexity} under the {\em cost condition} that
\begin{equation} \label{condition}
\parbox{11.4cm}{\em the approximate residual evaluation 
within an (absolute) tolerance $\eps \gtrsim \|{\bf r}\|$ requires not more than 
${\mathcal O}(\eps^{-1/s}+\# \supp {\bf \tilde{u}})$ operations.}
\end{equation}
The lower bound on $\eps$ reflects the fact that inside the adaptive scheme, it is never needed to approximate a residual more accurately than within a sufficiently small, but fixed relative tolerance. 
The validity of \eqref{condition} will require additional properties of $\Psi^\HH$ and $\Psi^\KK$ in addition to being Riesz bases. For that reason we consider wavelet bases.


The {\em standard way} to approximate the residual within tolerance $\eps$ is to approximate both ${\bf f}$ and ${\bf A} {\bf \tilde{u}}$ {\em separately} within tolerance $\eps/2$.
Under reasonable assumptions, ${\bf f}$ can be approximated within tolerance $\eps/2$ by a vector of length ${\mathcal O}(\eps^{-1/s})$.


For the approximation of ${\bf A} {\bf \tilde{u}}$, it is used that, thanks to the properties of the wavelets as having vanishing moments, 
each column of ${\bf A}$, although generally infinitely supported, can be well approximated by finitely supported vectors.
In the approximate matrix-vector multiplication routine introduced in \cite{45.2}, known as the {\bf APPLY}-routine, the accuracy with which a column is approximated is judiciously chosen {\em depending} on the size of the corresponding entry in the input vector ${\bf \tilde{u}}$.
It has been shown to realise a tolerance $\eps/2$ at cost ${\mathcal O}(\eps^{-1/s}|{\bf \tilde{u}}|_{\cA^{s}}^{1/s}+\# \supp {\bf \tilde{u}})$, for any $s$ in some range $(0,s^*]$. 
For wavelets that have {\em sufficiently many} vanishing moments, this range was shown to include the range of $s \in (0,s_{\max}]$ for which, in view of the order of the wavelets, ${\bf u} \in \cA^s$ can possibly be expected (cf. \cite{249.81}). 
Using that for the approximations ${\bf \tilde{u}}$ to ${\bf u}$ that are generated  inside the adaptive wavelet scheme, it holds that $|{\bf \tilde{u}}|_{\cA^{s}} \lesssim |{\bf u}|_{\cA^{s}}$,  in those cases the cost condition is satisfied, and so the adaptive wavelet scheme is optimal.


The {\bf APPLY}-routine, however, is quite difficult to implement. Note, in particular, that its outcome depends {\em nonlinearly} on the input vector ${\bf \tilde{u}}$.
Furthermore, in experiments, the routine turns out to be quantitatively expensive.
Finally, although it has been generalized to certain classes of semi-linear PDEs, in those cases it has not been shown that $s^* \geq s_{\max}$, meaning that for nonlinear problems the issue of optimality is actually open.


\subsection{An alternative for the {\bf APPLY} routine}
A main goal of this paper is to develop a quantitatively efficient alternative for the {\bf APPLY}-routine, that, moreover, gives rise to provable optimal adaptive wavelet schemes for classes of {\em semi-linear} PDEs, and applies to wavelets with only {\em one vanishing moment}.
As an introduction, we consider the model problem of Poisson's equation in one space dimension
$
\left\{
\begin{array}{@{}rcl@{}l@{}}
-u''&\!\!=\!\!&f& \text{ on } (0,1),\\ 
u&\!\!=\!\!&0& \text{ at }\{0,1\},
\end{array}
\right.$
that, in standard variational form, reads as finding $u \in \HH:=H^1_0(0,1)$ such that
$$
\langle u', v' \rangle_{L_2(0,1)}=  \langle f, v \rangle_{L_2(0,1)} ,\quad (v \in \KK:=H^1_0(0,1)),
$$
where, by identifying $L_2(0,1)'$ with $L_2(0,1)$ and using that $H_0^1(0,1) \hookrightarrow L_2(0,1)$ is dense, $\langle \,,\, \rangle_{L_2(0,1)}$ is also used to denote the duality pairing on $H^{-1}(0,1)\times H^1_0(0,1)$.
We consider piecewise polynomial, locally supported wavelet Riesz bases $\Psi^\HH$ and $\Psi^\KK$ for $H^1_0(0,1)$.
Let us exclusively consider {\em admissible} approximations ${\bf \tilde{u}}$ to ${\bf u}$ in the sense that their finite supports form {\em trees}, meaning that if $\lambda \in \supp {\bf \tilde{u}}$, then 
then there exists a $\mu  \in \supp {\bf \tilde{u}}$, whose level is one less than the level of $\lambda$, and $\meas(\supp \psi^\HH_\lambda \cap \supp \psi^\HH_\mu)>0$.
 It is known that the approximation classes $\cA^s$ become only `slightly' smaller by this restriction to tree approximation compared to unconstrained approximation (cf. \cite{45.21}).
 What is more, the restriction to tree approximation seems mandatory anyway to construct an optimal algorithm for nonlinear operators.
The benefit of tree approximation is that $\tilde{u}:={\bf \tilde{u}}^\top \Psi^\HH$ has an alternative, `single-scale' representation 
as a piecewise polynomial w.r.t. a partition ${\mathcal T}_1$ of $(0,1)$ with $\# {\mathcal T}_1 \lesssim \# {\supp {\bf \tilde{u}}}$.

For the moment, let us make the {\em additional assumption} that $\Psi^\HH$ is selected inside $H^2(0,1)$. Then, for an admissible ${\bf \tilde{u}}$, with its support denoted as $\Lambda^\HH$, integration-by-parts shows that
$$
{\bf r}:={\bf f}-{\bf A} {\bf \tilde{u}}=\langle \Psi^\KK,f+\tilde{u}'' \rangle_{L_2(0,1)},
$$
where $\tilde{u}''$ is piecewise polynomial w.r.t. ${\mathcal T}_1$.
If ${\bf u} \in \cA^s$, then for any $\eps>0$ there exists a piecewise polynomial $f_\eps$ w.r.t.\ a partition $\tria_2$ of (0,1) into ${\mathcal O}(\eps^{-1/s})$ subintervals such that $\|f-f_\eps\|_{H^{-1}(0,1)} \leq \eps$. \footnote{Indeed, for an admissible 
${\bf \bar{u}}$ with $\|{\bf u}-{\bf \bar{u}}\| \leq \eps$ and $\# \supp {\bf \bar{u}} \lesssim \eps^{-1/s}$, take $f_\eps=-\bar{u}''$ and use $\|f+\bar{u}''\|_{H^{-1}(\Omega)} \eqsim \|u-\bar{u}\|_{H^1(0,1)}\eqsim\|{\bf u}-{\bf \bar{u}}\|$.} The term $f-f_\eps$ is commonly referred to as {\em data oscillation}.

The function $f_\eps+\tilde{u}''$ is  piecewise polynomial w.r.t.\ the smallest common refinement $\tria$ of $\tria_1$ and $\tria_2$.
Thanks to this piecewise smoothness of $f_\eps+\tilde{u}''$ w.r.t. $\tria$, and the property of $\psi_\lambda^\KK$ having {\em one} vanishing moment,  $|\langle \psi_\lambda^\KK , f_\eps+\tilde{u}''\rangle_{L_2(0,1)}|$ is decreasing as function of the minimal difference of the {\em level} of $\psi_\lambda^\KK$ and that of any subinterval in $\tria$ that has non-empty intersection with $\supp \psi_\lambda^\KK$. Here with the level of $\psi_\lambda^\KK$ or that of an interval $\omega$, we mean an $\ell \in \N_0$ such that $2^{-\ell} \eqsim \diam(\supp \psi_\lambda^\KK)$ or $2^{-\ell} \eqsim \diam \omega$, respectively.
In particular, given a constant $\varsigma>0$, there exists a constant $k$, such that by dropping all $\lambda$ for which the aforementioned minimal level difference exceeds $k$, the remaining indices form a tree $\Lambda^\KK$ with $\# \Lambda^\KK \lesssim \# \tria \lesssim \eps^{-1/s}+\# \Lambda^\HH$ (dependent on $k$), and (see Prop.~\ref{prop1})
\begin{align*}
\|{\bf f}_\eps-{\bf A} {\bf \tilde{u}}-({\bf f}_\eps-{\bf A} {\bf \tilde{u}})|_{\Lambda^\KK}\|  \leq \varsigma \|f_\eps+\tilde{u}''\|_{H^{-1}(0,1)}
&\leq \varsigma \|f+\tilde{u}''\|_{H^{-1}(0,1)}+\varsigma \eps\\
&\lesssim \varsigma \|u-\tilde{u}\|_{H^1(0,1)}+\varsigma \eps,
\end{align*}
and so, using $\|{\bf r}\|\eqsim \|u-\tilde{u}\|_{H^1(0,1)}$ and $\|{\bf f}-{\bf f}_\eps\|\eqsim \|f-f_\eps\|_{H^{-1}(0,1)}$, 
$$
\|{\bf r}-{\bf r}|_{\Lambda^\KK}\| \lesssim  \varsigma \|{\bf r}\| +\eps.
$$
Note that for $\varsigma$ being sufficiently small, and so $k$ sufficiently large, by taking $\eps$ suitably the approximate residual will meet any accuracy that is required in the cost condition \eqref{condition}.

By selecting `single scale' collections $\Phi^\HH$ and $\Phi^\KK$ with $\Span \Phi^\HH \supseteq \Span \Psi^\HH|_{\Lambda^\HH}$ and $\Span \Phi^\KK \supseteq \Span \Psi^\KK|_{\Lambda^\KK}$, 
and $\# \Phi^\HH \lesssim \# \Lambda^\HH$ and $\# \Phi^\KK \lesssim \# \Lambda^\KK$, this approximate residual ${\bf r}|_{\Lambda^\KK}$ can be computed in ${\mathcal O}(\Lambda^\KK)$ operations as follows:
First express $\tilde{u}$ in terms of  $\Phi^\HH$ by applying a multi-to-single scale transformation to ${\bf \tilde{u}}$, then apply to this representation the sparse stiffness matrix $\langle (\Phi^\KK)',(\Phi^\HH)' \rangle_{L_2(0,1)}$, subtract $\langle \Phi^\KK,f \rangle_{L_2(0,1)}$, and finally apply the transpose of the multi-to-single scale transformation involving $\Psi^\KK|_{\Lambda^\KK}$ and $\Phi^\KK$.
This approximate residual evaluation thus satisfies the cost condition for optimality, it is relatively easy to implement, and it is observed to be quantitatively much more efficient. 

It furthermore generalizes to semi-linear operators, in any case for nonlinear terms that are multivariate polynomials in $u$ and derivatives of $u$.
Indeed, as an example, suppose that instead of $-u''=f$  the equation reads as $-u'' +u^3=f$. Then the residual is given by $\langle \Psi^\KK,f+\tilde{u}''-\tilde{u}^3 \rangle_{L_2(0,1)}$.
Since $f_\eps+\tilde{u}''-\tilde{u}^3$ is a piecewise polynomial w.r.t.\ $\tria$, the same arguments shows that $\langle \Psi^\KK,f+\tilde{u}''-\tilde{u}^3 \rangle_{L_2(0,1)}\big|_{\Lambda^\KK}$ is a valid approximate residual.
\medskip

The essential idea behind our approximate residual evaluation is that, after the replacement of $f$ by $f_\eps$, the different terms that constitute the residual are expressed in a {\em common} dictionary, before the residual, {\em as a whole}, is integrated against $\Psi^\KK$.
In our simple one-dimensional example this was possible by selecting $\Psi^\HH \subset H^2(0,1)$, so that the operator could be applied to the wavelets in strong, or more precisely, mild sense, meaning that the result of the application lands in $L_2(0,1)$.
It required piecewise smooth, globally $C^1$-wavelets. Although the same approach applies in more dimensions, there, except on product domains, the construction of $C^1$-wavelet bases is cumbersome.
For that reason, our approach will be to write a PDE of second order as a {\em system of PDEs of first order}.
It will turn out that there are several possibilities to do so.

\subsection{A common first order system least squares formulation}
To introduce ideas, let us again consider the model problem of Poisson's equation in one dimension. By introducing the additional unknown $\theta=u'$, for given $f \in L_2(0,1)$ this PDE can be written as the first order system
of  finding $(u,\theta) \in H^1_0(0,1) \times H^1(0,1)$ such that 
$$
\vec{H}(u,\theta):=(-\theta'-f,-u'+\theta)=\vec{0} \quad \text{in } L_2(0,1)\times L_2(0,1).
$$
The corresponding homogeneous operator\footnote{For general non-affine $\vec{H}$, $\vec{H}^{\rm h}$ should be read as the Fr\'{e}chet derivative $D\vec{H}(u,\theta)$.} $\vec{H}^{\rm h}:=(v,\eta)\mapsto (-\eta',-v'+\eta)$ is in $\Lis(H^1_0(0,1) \times H^1(0,1), L_2(0,1) \times L_2(0,1))$ (cf. \cite[(proof of) Thm.~3.1]{249.96}). To arrive at a symmetric and positive definite system, we consider the least squares problem of solving
$$
\argmin_{(u,\theta) \in H^1_0(0,1) \times H^1(0,1)} \|\vec{H}(u,\theta)\|_{L_2(0,1)\times L_2(0,1)}^2.
$$
Its solution solves the Euler-Lagrange equations 
$$
\langle \vec{H}^h(v,\eta),\vec{H}(u,\theta)\rangle_{L_2(0,1) \times L_2(0,1)}=0 \quad((v,\eta) \in H^1_0(0,1) \times H^1(0,1)).
$$
which in this setting are known as the {\em normal equations}.

To these 
 equations we apply the adaptive wavelet scheme, so with `$\HH$'$=$`$\KK$'$=H^1_0(0,1) \times H^1(0,1)$,
$($`$A$'$(u,\theta))(v,\eta):=\langle \vec{H}^h(v,\eta),\vec{H}^h(u,\theta)\rangle_{L_2(0,1) \times L_2(0,1)}$ and right-hand side
`$f$'$(v,\eta):=\langle f, -\eta' \rangle_{L_2(0,1)}$.
From $\vec{H}^{\rm h}$ being a homeomorphism with its range, i.e.,
$$
\|\vec{H}^{\rm h}(v,\eta)\|_{L_2(0,1) \times L_2(0,1)} \eqsim \|(v,\eta)\|_{H^1_0(0,1) \times H^1(0,1)},
$$
being a consequence of $\vec{H}^{\rm h}$ being even boundedly invertible between the full spaces, it follows 
that the bilinear form is bounded, symmetric, and coercive.
After equipping $H_0^1(0,1)$ and $H^1(0,1)$ with wavelet Riesz bases $\Psi^{H^1_0}$ and $\Psi^{H^1}$, for admissible ${\bf \tilde{u}}$ and $\bm{\tilde{\theta}}$, with $\tilde{u}:={\bf \tilde{u}}^\top \Psi^{H^1_0}$ and $\tilde{\theta}:=\bm{\tilde{\theta}}^\top \Psi^{H^1}$ the residual reads as
\begin{equation} \label{90}
{\bf r}=\left[
\begin{array}{c}
{\bf r}_1 \\ {\bf r}_2
\end{array}
\right]
=
\left[
\begin{array}{c}
\langle (\Psi^{H^1_0})',\tilde{u}'-\tilde{\theta}\rangle_{L_2(0,1)} \\
\langle (\Psi^{H^1})',\tilde{\theta}'+f\rangle_{L_2(0,1)} +\langle \Psi^{H^1},\tilde{\theta}-\tilde{u}'\rangle_{L_2(0,1)}
\end{array}
\right].
\end{equation}

The construction of an approximate residual follows the same lines as described before for  the standard variational formulation.\footnote{Actually, in the current setting its analysis is more straightforward, because the residuals are measured in $L_2(0,1)$ instead of in $H^{-1}(0,1)$.}
The functions $\tilde{u}',\,\tilde{\theta},\,\tilde{\theta}'$ are piecewise polynomials w.r.t.\ a partition $\tria_1$ of $(0,1)$ into ${\mathcal O}(\#\supp {\bf \tilde{u}}+\#\supp \bm{\tilde{\theta}})$ subintervals.
If $({\bf u},\bm{\theta}) \in \cA^s$, then there exists a piecewise polynomial $f_\eps$ w.r.t.\ a partition $\tria_2$ of $(0,1)$ into ${\mathcal O}(\eps^{-1/s})$ subintervals such that $\|f-f_\eps\|_{L_2(0,1)} \leq \eps$.
Thanks to the piecewise smoothness of $\tilde{u}'-\tilde{\theta}$ and $\tilde{\theta}'+f_\eps$, there exist trees $\Lambda^{H^1_0}$ and $\Lambda^{H^1}$, with $\# \Lambda^{H^1_0}+ \# \Lambda^{H^1} \lesssim \#\tria_1+\#\tria_2$ (dependent on $\varsigma$), such that
$$
\Big\|{\bf r}-
\left[
\begin{array}{c}
{\bf r}_1|_{\Lambda^{H^1_0}} \\ {\bf r}_2|_{\Lambda^{H^1}}
\end{array}
\right]
\Big\| \lesssim \varsigma (\|\tilde{u}'-\tilde{\theta}\|_{L_2(0,1)}+\|\tilde{\theta}'+f_\eps\|_{L_2(0,1)}) +\eps
\lesssim \varsigma \|{\bf r}\|+\eps.
$$
Since the approximate residual can be evaluated in ${\mathcal O}(\# \Lambda^{H^1_0}  \cup \Lambda^{H^1})$ operations, we conclude that it satisfies the cost condition \eqref{condition} for optimality of the adaptive wavelet scheme.
\medskip

\begin{remark} Recall that, as always with least squares formulations, the same results are valid when lower order, possibly {\em non-symmetric} terms are added to the second order PDE, as long as the standard variational formulation remains well-posed. Furthermore, as we will discuss, least squares formulations allows to handle {\em inhomogeneous boundary conditions}.
Finally, as we will see, the approach of reformulating a 2nd order PDE as a first order least squares problem, and then optimally solving the normal equations applies to {\em any} well-posed PDE, not necessarily being elliptic.
\end{remark}

In \cite{38.42} we applied the adaptive wavelet scheme to a least squares formulation of the current, common type.
{\em Disadvantages} of this formulation are that (i) it requires that $f \in L_2(0,1)$, instead of $f \in H^{-1}(0,1)$ as allowed in the standard variational formulation.
Related to that, and more importantly, for a semi-linear equation $-u''+N(u)=f$, (ii) it is needed that $N$ maps $H^1_0(0,1)$ into $L_2(0,1)$, instead of into $H^{-1}(0,1)$.
Finally, with the generalization of this least squares formulation to more than one space dimensions, (iii) the space $H^1(0,1)$ for $\theta$ reads as $H(\divv;\Omega)$.
In \cite{38.42}, for two-dimensional connected polygonal domains $\Omega$, we managed to construct a wavelet Riesz basis for $H(\divv;\Omega)$.
This construction, however, relied on the fact that, in two dimensions, any divergence-free function is the curl of an $H^1$-function.
To the best of our knowledge, wavelet Riesz bases for $H(\divv;\Omega)$ for non-product domains in three and more dimensions have not been constructed.

In the next subsection, we describe a prototype of a least-squares formulation with which these disadvantages (i)--(iii) are avoided.

\subsection{A seemingly unpractical least squares formulation} \label{Sseemingly}
The first order system least squares formulation that will be studied in this paper reads, for the model problem, as follows:
Again we introduce $\theta=u'$, but now consider the first order system of finding $(u,\theta) \in H^1_0(0,1) \times L_2(0,1)$ such that 
$$
\vec{H}(u,\theta):=(D'\theta+f,-u'+\theta)=\vec{0} \quad \text{in } H^{-1}(0,1)\times L_2(0,1).
$$
where $(D'\theta)(v):=-\langle \theta,v'\rangle_{L_2(0,1)}$, i.e., $D'\theta$ is the distributional derivative of $\theta$. 

In `primal' mixed form, this system reads as
$$
\langle \theta,v'\rangle_{L_2(0,1)}+\langle \theta-u',\eta \rangle_{L_2(0,1)}=\langle f,v \rangle_{L_2(0,1)} \quad ((v,\eta) \in H^1_0(0,1)\times L_2(0,1)).
$$
The 
corresponding homogeneous operator $\vec{H}^{\rm h}$ is in $\Lis(H^1_0(0,1) \times L_2(0,1), H^{-1}(0,1) \times L_2(0,1))$, and the least squares problem reads as
 solving
\begin{equation} \label{36}
\argmin_{(u,\theta) \in H^1_0(0,1) \times L_2(0,1)} \|\vec{H}(u,\theta)\|_{H^{-1}(0,1)\times L_2(0,1)}^2,
\end{equation}
with normal equations reading as
\begin{equation} \label{37}
\begin{split}
0 & =\langle \vec{H}^h(v,\eta),\vec{H}(u,\theta)\rangle_{H^{-1}(0,1) \times L_2(0,1)}\\
& = \langle D'\eta,D'\theta  -f \rangle_{H^{-1}(0,1)}+\langle -v'+\theta,-u'+\theta\rangle_{L_2(0,1)}
 \end{split}
\end{equation}
($(v,\eta) \in H^1_0(0,1) \times L_2(0,1)$).

In the terminology from \cite{23.5}, our current least squares problem is identified as being unpractical because of the appearance of the dual norm.
To deal with this, as  in \cite{56.3}, we select {\em some} wavelet Riesz basis $\Psi^{\hat{H}^1_0}$ for $H^1_0(0,1)$, and replace the norm on $H^{-1}(0,1)$ in \eqref{36} by the {\em equivalent} norm defined by $\|g(\Psi^{\hat{H}^1_0})\|$ for $g \in H^{-1}(\Omega)$.
Correspondingly, in \eqref{37} we replace the inner product $\langle g,h\rangle_{H^{-1}(0,1)}$ by $g(\Psi^{\hat{H}^1_0})^\top h(\Psi^{\hat{H}^1_0})$, so that 
the resulting normal equations read as finding $(u,\theta) \in H^1_0(0,1) \times L_2(0,1)$ such that
$$
\langle \eta, (\Psi^{\hat{H}^1_0})'\rangle_{L_2(0,1)}\big\{\langle (\Psi^{\hat{H}^1_0})',\theta \rangle_{L_2(0,1)}-\langle \Psi^{\hat{H}^1_0},f\rangle_{L_2(0,1)}\big\}+\langle -v'+\eta,-u'+\theta \rangle_{L_2(0,1)}=0
$$
for all $(v,\eta) \in H^1_0(0,1) \times L_2(0,1)$.

To apply the adaptive wavelet scheme to these normal equations, we equip $H^1_0(0,1)$ and $L_2(0,1)$ with wavelet Riesz bases $\Psi^{H^1_0}$ and $\Psi^{L_2}$, respectively. When these bases have order $p+1$ and $p$, the best possible convergence rate $s_{\max}$ will be equal to $p$ ($p/n$ on an $n$-dimensional domain).
Note that the order of the basis $\Psi^{\hat{H}^1_0}$ is irrelevant.

For approximations $(\tilde{u},\tilde{\theta})=({\bf \tilde{u}}^\top \Psi^{H^1},\bm{\tilde{\theta}}^\top \Psi^{L_2})$ for admissible ${\bf \tilde{u}}$ and $\bm{\tilde{\theta}}$, the residual ${\bf r}$  of $({\bf \tilde{u}},\bm{\tilde{\eta}})$ reads as
$$
\left[
\begin{array}{c}
\langle (\Psi^{H^1_0})',\tilde{u}'-\tilde{\theta} \rangle_{L_2(0,1)}\\
\langle \Psi^{L_2}, (\Psi^{\hat{H}^1_0})'\rangle_{L_2(0,1)} \big\{\langle (\Psi^{\hat{H}^1_0})',\tilde{\theta} \rangle_{L_2(0,1)}- \langle \Psi^{\hat{H}^1_0},f\rangle_{L_2(0,1)} \big\}+\langle \Psi^{L_2},\tilde{\theta}-\tilde{u}'\rangle_{L_2(0,1)}
\end{array}
\right]
$$
that, under the {\em additional condition} that $\Psi^{L_2} \subset H^1(0,1)$, and thus $\tilde{\theta} \in H^1(0,1)$, is equal to
\begin{equation} \label{91}
\left[
\begin{array}{c}
{\bf r}_1 \\
{\bf r}_2
\end{array}
\right]=
\left[
\begin{array}{@{}c@{}}
\langle (\Psi^{H^1_0})',\tilde{u}'-\tilde{\theta}\rangle_{L_2(0,1)}\\
-\langle \Psi^{L_2}, (\Psi^{\hat{H}^1_0})'\rangle_{L_2(0,1)} \langle \Psi^{\hat{H}^1_0},\tilde{\theta}'+f\rangle_{L_2(0,1)}+\langle \Psi^{L_2},\tilde{\theta}-\tilde{u}'\rangle_{L_2(0,1)}
\end{array}
\right]\!\!.
\end{equation}
This last step is essential because it allows us, after the replacement of $f$ by a piecewise polynomial $f_\eps$, to express
$\tilde{\theta}'+f_\eps$ in a common dictionary.
The additional condition is satisfied by piecewise polynomial, globally {\em continuous} wavelets, which are available on general domains in multiple dimensions.

In view of the previous discussions, to describe the approximate residual evaluation, it suffices to consider the term $\langle \Psi^{L_2}, {\bf z}^\top (\Psi^{\hat{H}^1_0})'\rangle_{L_2(0,1)}$ with ${\bf z}:= \langle \Psi^{\hat{H}^1_0},f+\tilde{\theta}'\rangle_{L_2(0,1)}$. The by now familiar approach is applied twice:
The function $\tilde{\theta}'$ is piecewise polynomial w.r.t.\ a partition $\tria_1$ into ${\mathcal O}(\# \supp \bm{\tilde{\theta}})$ subintervals. 
If $({\bf u},\bm{\theta}) \in \cA^s$, then there exists a piecewise polynomial $f_\eps$ w.r.t.\ a partition $\tria_2$ of (0,1) into ${\mathcal O}(\eps^{-1/s})$ subintervals such that $\|f-f_\eps\|_{H^{-1}(0,1)} \leq \eps$. 
Consequently, there exists a tree $\Lambda^{\hat{H}^1_0}$ with $\# \Lambda^{\hat{H}^1_0} \lesssim \# \supp \bm{\tilde{\theta}}+\eps^{-1/s}$ (dependent on $\varsigma$) such that $\|{\bf z}-{\bf z}|_{\Lambda^{\hat{H}^1_0}}\| \lesssim \varsigma \|f+\tilde{\theta}'\|_{H^{-1}(0,1)}+\eps$.
The function $\tilde{z}:={\bf z}|_{\Lambda^{\hat{H}^1_0}}^\top \Psi^{\hat{H}^1_0}$ is piecewise polynomial  w.r.t.\ a partition $\tria_2$ of (0,1) into ${\mathcal O}(\# \Lambda^{\hat{H}^1_0})$ subintervals.
Consequently, exists a tree $\Lambda_1^{L_2}$ with $\# \Lambda_1^{L_2} \lesssim \# \Lambda^{\hat{H}^1_0}$ (dependent on $\varsigma$) such that
 $\|\langle \Psi^{L_2}, \tilde{z}'\rangle_{L_2(0,1)}-\langle \Psi^{L_2}, \tilde{z}'\rangle_{L_2(0,1)}\big|_{\Lambda_1^{L_2}}\| \leq \varsigma \|\tilde{z}\|_{H^1(0,1)} \lesssim \varsigma(\|{\bf z}\|+\|{\bf z}-{\bf z}|_{\Lambda^{\hat{H}^1_0}}\|) \lesssim \varsigma(\|f+\tilde{\theta}'\|_{H^{-1}(0,1)}+\eps)$.
 Combining this with the approximations for the other two terms that constitute the residual, we infer that there exist trees $\Lambda^{H^1_0}$ and $\Lambda^{L_2}$ with $\# \Lambda^{H^1_0}+\# \Lambda^{L_2} \lesssim \#\supp \# {\bf \tilde{u}}+\# \supp \bm{\tilde{\theta}} +\eps^{-1/s}$ (dependent on $\varsigma$), such that 
 $$
\Big\|{\bf r}-
\left[
\begin{array}{c}
{\bf r}_1|_{\Lambda^{H^1_0}} \\ \tilde{\bf r}_2|_{\Lambda^{L_2}}
\end{array}
\right]
\Big\| \lesssim \varsigma (\|\tilde{u}'-\tilde{\theta}\|_{L_2(0,1)}+\|\tilde{\theta}'+f_\eps\|_{H^{-1}(0,1)}) +\eps
\lesssim \varsigma \|{\bf r}\|+\eps,
$$
where $\tilde{\bf r}_2$ is constructed from ${\bf r}_2$ by replacing ${\bf z}$ by ${\bf z}|_{\Lambda^{\hat{H}^1_0}}$.
This approximate residual evaluation satisfies the cost condition \eqref{condition} for optimality.

As we will see in Sect.~\ref{S1}, the advantage of the current construction of a first order system least squares problem is that it applies to {\em any} well-posed (semi-linear) second order PDE.
The two instances of the spaces $H^1_0(0,1)$ represent the trial and test spaces in the standard variational formulation, and well-posedness of the latter implies well-posedness of the least squares formulation. The additional space $L_2(0,1)$ reads in general as an $L_2$-type space.
So, in particular, with this formulation, $H(\divv)$-spaces do not enter.
The price to be paid is that \eqref{91} is somewhat more complicated than \eqref{90}, and that therefore  its approximation is somewhat more costly to compute.

\begin{remark} The more popular  `dual' mixed formulation of our model problem reads as finding $(u,\theta) \in L_2(0,1) \times H^1(0,1)$ such that
$\langle -\theta',v\rangle_{L_2(0,1)}+\langle \theta,\eta\rangle_{L_2(0,1)}+\langle u,\eta'\rangle_{L_2(0,1)}=\langle f,v\rangle_{L_2(0,1)}$ ($(v,\eta) \in L_2(0,1) \times H^1(0,1)$).
The resulting least squares formulation has the combined disadvantages of both other formulations that we considered.
It requires $f \in L_2(0,1)$, possibly nonlinear terms should map into $L_2(0,1)$, in more than one dimension the space $H^1(0,1)$ reads as an $H(\divv)$-space, and one of the norms involved in the least squares minimalisation is a dual norm.
\end{remark}

\begin{remark} With the aim to avoid both a dual norm in the least squares minimalisation, and $H(\divv)$ or other vectorial Sobolev spaces as trial spaces, in our first investigations of this least squares approach in \cite{249.94}, we considered the `extended div-grad' first order system least squares formulation studied in \cite{35.9301}. A sufficient and necessary  (\cite{249.96}), but restrictive condition for its well-posedness is $H^2$-regularity of the homogeneous boundary value problem.
\end{remark}


\subsection{Layout of the paper} In Sect.~\ref{S1}, a general procedure is given to reformulate {\em any} well-posed semi-linear 2nd order PDE  as a well-posed first order least squares problem.
As we will see, this procedure gives an effortless derivation of well-posed first order least squares formulations of elliptic 2nd order PDEs, and that of the stationary Navier-Stokes equations.
The arising dual norm can be replaced by the equivalent $\ell_2$-norm of a functional in wavelet coordinates.

In Sect.~\ref{Sawgm}, we recall properties of the adaptive wavelet Galerkin method ({\bf awgm}). Operator equations of the form $F(z)=0$, where, for some Hilbert space $\HH$,  $F:\HH \rightarrow \HH'$ and $DF(z)$ is symmetric and positive definite, are solved by the {\bf awgm} at the best possible rate from a Riesz basis for $\HH$.
Furthermore, under a condition on the cost of the  approximate residual evaluations, the method has optimal computational complexity.

In the short Sect.~\ref{Snormal}, it is shown that the {\bf awgm} applies to the normal equations that result from the first order least squares problems as derived in Sect.~\ref{S1}.

In Sect.~\ref{Selliptic}, we apply the {\bf awgm} to a first order least squares formulation of a semi-linear 2nd order elliptic PDE with general inhomogeneous boundary conditions. 
Under a mild condition on the wavelet basis for the trial space, the efficient approximate residual evaluation that was outlined in Sect.~\ref{Sseemingly} applies, and it satisfies the cost condition, so that the {\bf awgm} is optimal. Wavelet bases that satisfy the assumptions are available on general polygonal domains.
Some technical results needed for this section are given in Appendix~\ref{Sdecay}.

In Sect.~\ref{Snumerics} the findings from Sect.~\ref{Selliptic} are illustrated by numerical results.

In Sect.~\ref{SNSE}, we consider the so-called velocity--pressure--velocity gradient and the velocity--pressure--vorticity first order system formulations of the stationary Navier-Stokes equations.
Results analogously to those demonstrated for the elliptic problem will be shown to be valid here as well.



\section{Reformulation of a semi-linear second order PDE as a first order system least squares problem}  \label{S1}

In an abstract framework, we give a procedure to write semi-linear second order PDEs, that have well-posed standard variational formulations, as a well-posed first order system least squares problems. A particular instance of this approach has been discussed in Sect.~\ref{Sseemingly}.

For some separable Hilbert spaces $\UU$ and $\VV$, for convenience over $\R$, consider a differentiable mapping
$$
G:\UU \supset\dom(G) \rightarrow \VV'.
$$
\begin{remark}  In applications $G$ is the operator associated to a variational formulation of a PDO with trial space $\UU$ and test space $\VV$.
\end{remark}

For $\PP$ being another separable Hilbert space, let
\begin{equation} \label{9}
G=G_0+G_1 G_2,\text{ where }
G_1 \in \cL(\PP,\VV'),\,G_2 \in \cL(\UU,\PP),
\end{equation}
i.e., $G_1$ and $G_2$ are bounded linear operators.

\begin{remark} In applications, as those discussed in Sect.~\ref{Selliptic} and ~\ref{SNSE}, $G_1 G_2$ will be a factorization of
the leading second order part of the PDO (possibly modulo terms that vanish at the solution, cf. Sect.~\ref{SVPV}) into a product of first order PDOs.
\end{remark}


Obviously, $u$ solves \framebox{$G(u)=0$} if and only if it is the first component of the solution $(u,\theta)$ of
$$
\framebox{$\vec{H}(u,\theta):=(G_0(u)+G_1 \theta,\theta-G_2 u)=\vec{0},$}
$$
where $\vec{H}:\UU \times \PP \supset \dom(G)\times \PP=\dom(\vec{H}) \rightarrow \VV' \times \PP$.

The following lemma shows that well-posedness of the original formulation implies that of the reformulation as a system.

\begin{lemma} \label{lem1} Let $DG(u) \in \cL(\UU,\VV')$ be a homeomorphism with its range, i.e., $\|DG(u) v\|_{\VV'} \eqsim \|v\|_\UU$ $(v \in \UU)$.
Then
$$
D\vec{H}(u,\theta)=\left[\begin{array}{@{}cc@{}} DG_0(u) & G_1 \\ -G_2 & I \end{array} \right] \in \cL(\UU\times \PP,\VV'\times \PP)
$$
is a homeomorphism with its range, being $\{(f,g) \in \VV'\times \PP\colon f-G_1g \in \ran DG(u)\}$. In particular, with $r, R>0$ such that
$r \|v\|_\UU \leq \|DG(u)(v)\|_{\VV'} \leq R \|v\|_\UU$, it holds that
\begin{align} \label{een}
\|(v,\eta)\|_{\UU\times \PP}^2 \leq \Big[({\textstyle \frac{1+\|G_2\|}{r}})^2+(1+{\textstyle\frac{1+\|G_1\|(1+\|G_2\|)}{r}})^2\Big]
\Big\|D\vec{H}(u,\theta) \left[\begin{array}{@{}c@{}} v \\ \eta \end{array} \right]\Big\|_{\VV' \times \PP}^2,
 \\ \label{twee}
\Big\|D\vec{H}(u,\theta) \left[\begin{array}{@{}c@{}} v \\ \eta \end{array} \right]\Big\|_{\VV' \times \PP}^2 \!\!\leq 
\big((R\!+\!(1\!+\!\|G_1\|)\|G_2\|)^2\!+\!(1\!+\!\|G_1\|)^2\big)\|(v,\eta)\|_{\UU\times \PP}^2.
\end{align}
\end{lemma}
\begin{remark}
Since $\ran DG(u)=\VV'$ iff $\ran D\vec{H}(u,\theta)=\VV'\times \PP$,  in particular we have that $DG(u) \in \Lis(\UU,\VV')$ implies that $D\vec{H}(u,\theta) \in \Lis(\UU \times \PP,\VV' \times \PP)$.
\end{remark}

\begin{proof} 
We have
$D\vec{H}(u,\theta)\left[\begin{array}{@{}c@{}} v \\ \eta \end{array} \right] =\left[\begin{array}{@{}c@{}} DG_0(u)v+G_1\eta \\ \eta-G_2 v \end{array} \right]
=\left[\begin{array}{@{}c@{}} DG(u) v+G_1(\eta-G_2v) \\ \eta-G_2 v \end{array} \right]$ by $DG_0(\cdot)=DG(\cdot)-G_1 G_2$.
So
\begin{equation} \label{35}
\ran D\vec{H}(u,\theta) \subseteq \{(f,g) \in \VV' \times \PP \colon f-G_1g \in \ran DG(u)\}.
\end{equation}
By estimating $\|D\vec{H}(u,\theta) \left[\begin{array}{@{}c@{}} v \\ \eta \end{array} \right]\|_{\VV' \times \PP} \leq \|DG(u) v+G_1(\eta-G_2v)\|_{\VV'}+\|\eta-G_2 v\|_\PP$, one easily arrives at \eqref{twee}.

For $(f,g)$ being in the set at the right-hand side in \eqref{35}, consider the system
$$
D\vec{H}(u,\theta) \left[\begin{array}{@{}c@{}} v \\ \eta \end{array} \right]=\left[\begin{array}{@{}c@{}} f \\ g \end{array} \right] \Longleftrightarrow
\left\{ \begin{array}{@{}r@{}c@{}c@{}}  DG_0(u) v+G_1 \eta &\,=\,& f \\
\eta-G_2 v &\,=\,& g
\end{array} \right.
\Longleftrightarrow
\left\{ \begin{array}{@{}r@{}c@{}c@{}}  DG(u) v &\,=\,& f-G_1 g \\
\eta &\,=\,& g+G_2 v
\end{array} \right..
$$
This system has a unique solution, so that the $\subseteq$-symbol in \eqref{35} reads as an equality sign, 
and  $r\|v\|_\UU \leq \|f\|_{\VV'}+\|G_1\|\|g\|_{\PP}$ and $\|\eta\|_\PP \leq \|g\|_\PP+\|G_2\|\|v\|_\UU$.
By estimating $\|(v,\eta)\|_{\UU\times \PP} \leq \|v\|_\UU+\|\eta\|_\PP$ one easily arrives at \eqref{een}.
\end{proof}

In the following, we will always assume that
\renewcommand{\theenumi}{\roman{enumi}}
\begin{enumerate} \item \label{r1} there exists a solution $u$ of $G(u)=0$;
\item \label{r2}$G$ is two times continuously Fr\'{e}chet differentiable in a neighborhood of $u$;
\item \label{r3}$DG(u) \in \cL(\UU,\VV')$ is a homeomorphism with its range.
\end{enumerate}
Then
\renewcommand{\theenumi}{\alph{enumi}}
\begin{enumerate} 
\item \label{r4} $(u,\theta)=(u,G_2 u)$ solves $\vec{H}(u,\theta)=\vec{0}$;
\item \label{r5} $\vec{H}$ is two times continuously Fr\'{e}chet differentiable in a neighborhood of $(u,\theta)$;
\item \label{r6} $D\vec{H}(u,\theta) \in \cL(\UU\times \PP,\VV' \times \PP)$ is a homeomorphism with its range,
\end{enumerate}
the latter by Lemma~\ref{lem1}.
In summary, when the equation $G(u)=0$ is {\em well-posed} (\eqref{r1}-\eqref{r3} are valid), then so is $\vec{H}(u,\theta)=\vec{0}$ (\eqref{r4}-\eqref{r6} are valid), and solving one equation is equivalent to solving the other. 

\begin{remark}Actually, one might dispute whether these equations should be called well-posed when $\ran DG(u) \subsetneq \VV'$ and so 
$\ran D\vec{H}(u,\theta) \subsetneq \VV' \times \PP$. 
In any case, under conditions \eqref{r1}--\eqref{r3}, and so \eqref{r4}--\eqref{r6}, the corresponding least-squares problems and resulting (nonlinear) normal equations {\em are} well-posed, as we will see next.
\end{remark}

A solution $(u,\theta)$ of $\vec{H}(u,\theta)=\vec{0}$ is a minimizer of the least squares functional
$$
\framebox{$Q(u,\theta):={\textstyle \frac{1}{2}} \|\vec{H}(u,\theta)\|_{\VV' \times \PP}^2$.}
$$
In particular, it holds that
\begin{lemma} \label{lem2} For $\vec{H}:\UU \times \PP \supset \dom(\vec{H}) \rightarrow \VV' \times \PP$, and $\vec{H}$  being Fr\'{e}chet differentiable at a root $(u,\theta)$, property \eqref{r6}
is equivalent to the property that for $(\tilde{u},\tilde{\theta}) \in \UU \times \VV$ in a neighbourhood of $(u,\theta)$, 
$$
Q(\tilde{u},\tilde{\theta}) \eqsim  \|u-\tilde{u}\|^2_{\UU}+\|\theta-\tilde{\theta}\|^2_{\PP}.
$$
\end{lemma}

\begin{proof} This is a consequence of  $\vec{H}(\tilde{u},\tilde{\theta})=D\vec{H}(u,\theta)(\tilde{u}-u,\tilde{\theta}-\theta)+o(\|\tilde{u}-u\|_\UU+\|\tilde{\theta}-\theta\|_\PP)$.
\end{proof}

A minimizer $(u,\theta)$ of $Q$ is a solution of the Euler-Lagrange equations
\begin{equation} \label{1}
\framebox{$DQ(u,\theta)(v,\eta)= \big\langle D\vec{H}(u,\theta) \left[\begin{array}{@{}c@{}} v \\ \eta \end{array} \right], \vec{H}(u,\theta)\big\rangle_{\VV' \times \PP}=0 \quad ((v,\eta) \in  (\UU \times \PP)),$}
\end{equation}
that, in this setting, are usually called (nonlinear) {\em normal equations}.
Using \eqref{r4}-\eqref{r5},
one computes that
$$
D^2 Q(u,\theta)(v_1,\eta_1)(v_2,\eta_2)=\big\langle D\vec{H}(u,\theta) \left[\begin{array}{@{}c@{}} v_1 \\ \eta_1 \end{array} \right], D\vec{H}(u,\theta) \left[\begin{array}{@{}c@{}} v_2 \\ \eta_2 \end{array} \right] \big\rangle_{\VV' \times \PP}.
$$

We conclude the following: Under the assumptions \eqref{r4}--\eqref{r6}, it holds that 
\renewcommand{\theenumi}{\arabic{enumi}}
\begin{enumerate}
\item \label{38} $D Q$ is a mapping from a subset of a separable Hilbert space, viz. $\UU\times \PP$, to its dual;
\item there exists a solution of $DQ(u,\theta)=0$ (viz. any solution of $\vec{H}(u,\theta)=0$);
\item $DQ$ is continuously Fr\'{e}chet differentiable in a neighborhood of $(u,\theta)$;
\item \label{41} $0<D^2 Q(u,\theta)= D^2 Q(u,\theta)' \in \Lis(\UU\times \PP,(\UU\times \PP)')$.
\end{enumerate}
As a consequence of the last property, one infers that in a neighborhood of $(u,\theta)$, $DQ(u,\theta)=0$ has exactly one solution.
\medskip

In view of the above findings, in order to solve $G(u)=0$, for a $G$ that satisfies \eqref{r1}--\eqref{r3}, we are going to solve the (nonlinear) normal equations $DQ(u,\theta)=0$.
A major advantage of $DQ$ over $G$ is that its derivative is symmetric and coercive.

A concern, however, is whether, for given $(u,\theta), (v,\eta) \in \UU \times \PP$, $DQ(u,\theta)(v,\eta)$ as given by \eqref{1} is evaluable. We will think of
the inner product on $\PP$ as being evaluable. In our applications, $\PP$ will be of the form $L_2(\Omega)^N$.
To deal with the dual norm on $\VV'$, we equip $\VV$ with a {\em Riesz basis}
$$
\Psi^\VV=\{\psi_\lambda^{\VV}:\lambda \in \vee_\VV\},
$$
meaning that the {\em analysis operator}
$$
\cF_\VV:g \mapsto [g(\psi^\VV_\lambda)]_{\lambda \in \vee_\VV} \in \Lis(\HH,\ell_2(\vee_\VV)),
$$
and so its adjoint, known as the {\em synthesis operator},
$$
\cF_\VV':{\bf v} \mapsto {\bf v}^\top \Psi^\VV:=\sum_{\lambda \in \vee_\VV} v_\lambda \psi^\VV_\lambda \in \Lis(\ell_2(\vee_\VV),\VV').
$$
In the definition of the least squares functional $Q$, and consequently in that of $DQ$, we now {\em replace} the standard dual norm on $\VV'$  by the {\em equivalent}  norm $\|\cF_\VV \cdot\|_{\ell_2(\vee_\VV)}$.
Then in view of the definition of $\vec{H}$ and the expression for $D\vec{H}$, we obtain that
\begin{equation} \label{7}
\framebox{$\begin{split}
&DQ(u,\theta)(v,\eta)=\\
&
(DG_0(u)v+G_1 \eta)(\Psi^\VV)^\top (G_0(u)+G_1 \theta)(\Psi^\VV)+\langle \eta-G_2 v,\theta-G_2 u\rangle_{\PP},
\end{split}$}
\end{equation}
where \eqref{38}--\eqref{41} are still valid.

\begin{remark} 
We refer to \cite{35.83} for an alternative approach to solve least square problems that involves dual norms.
\end{remark}

To solve the obtained (nonlinear) normal equations $DQ(u,\theta)=0$ we are going to apply the adaptive wavelet Galerkin method ({\bf awgm}).
Note that the definition of $DQ(u,\theta)(v,\eta)$ still involves an infinite sum over $\vee_\VV$ that later, inside the solution process, is going to be replaced by a finite one.

\section{The adaptive wavelet Galerkin method ({\bf awgm})} \label{Sawgm}
In this section, we summarize findings about the {\bf awgm} from \cite{249.94,38.42}. Let
\begin{enumerate} \renewcommand{\theenumi}{\Roman{enumi}}
\item \label{12} $F:\HH \supset \dom(F) \rightarrow \HH'$, with $\HH$ being a separable Hilbert space;
\item \label{13} $F(z)=0$;
\item \label{14} $F$ be continuously differentiable in a neighborhood of $z$;
\item \label{15} $0<DF(z)=DF(z)' \in \Lis(\HH,\HH')$.
\end{enumerate}
In our applications, the triple $(F,\HH,z)$ will read as $(DQ,\UU \times \PP,(u,\theta))$,  so that \eqref{12}-\eqref{15} are guaranteed by \eqref{38}--\eqref{41}.

Let $\Psi=\{\psi_\lambda:\lambda \in \vee\}$ be a {\em Riesz basis} for $\HH$, with analysis operator $\cF:g \mapsto [g(\psi_\lambda)]_{\lambda \in \vee} \in \Lis(\HH,\ell_2(\vee))$,
and so synthesis operator $\cF':{\bf v} \mapsto {\bf v}^\top \Psi:=\sum_{\lambda \in \vee} v_\lambda \psi_\lambda \in \Lis(\ell_2(\vee),\HH')$. For  any $\Lambda \subset \vee$, we set
$$
\ell_2(\Lambda):=\{{\bf v}\in \ell_2(\vee)\colon\supp {\bf v} \subset \Lambda\}.
$$
For satisfying the forthcoming Condition~\ref{nonlineareval} that concerns the computational cost, it will be relevant that $\Psi$ is a basis of {\em wavelet} type.

Writing $z=\cF' {\bf z}$, and with
$$
{\bf F}:=\cF F \cF'\colon\ell_2(\vee) \rightarrow \ell_2(\vee),
$$
an {\em equivalent} formulation of $F(z)=0$ is given by
$$
{\bf F}({\bf z})=0.
$$

We are going to approximate ${\bf z}$, and so $z$, by a sequence of Galerkin approximations from the spans of increasingly larger sets of wavelets, which sets are created by an adaptive process.
Given $\Lambda \subset \vee$, the Galerkin approximation ${\bf z}_\Lambda$, or equivalently, $z_\Lambda:={\bf z}_\Lambda^\top \Psi$, are the solutions
of $\langle {\bf F}({\bf z}_\Lambda),{\bf v}_\Lambda \rangle_{\ell_2(\vee)}=0$ (${\bf v}_\Lambda \in \ell_2(\Lambda)$), i.e., ${\bf F}({\bf z}_\Lambda)|_{\Lambda}=0$, and $F(z_\Lambda)(v_\Lambda)=0$ ($v_\Lambda \in \Span\{\psi_\lambda:\lambda \in \Lambda\}$), respectively.
These solutions exist uniquely when $\inf_{{\bf \tilde{z}}_\Lambda \in \ell_2(\Lambda)} \|{\bf z}-{\bf \tilde{z}}_\Lambda\|$ is sufficiently small (\cite{243.4, 249.94}).

In order to be able to construct efficient algorithms, in particular when $F$ is non-affine, it will be needed to consider only sets $\Lambda$ from a certain subset of all finite subsets of $\vee$.
In our applications, this collection of so-called {\em admissible} $\Lambda$ will consist of (Cartesian products of) finite {\em trees}.
For the moment,  it suffices when the collection of admissible sets is such that
the union of any two admissible sets is again admissible.

To provide a benchmark to evaluate our adaptive algorithm, for $s>0$, we define the nonlinear {\em approximation class}
\begin{equation} \label{class}
\begin{split}
\cA^s:=\Big\{{\bf z} \in &\ell_2(\vee)\colon \|{\bf z}\|_{\cA^s}:=\\
&\sup_{\eps>0} \eps \times \min\big\{(\# \Lambda)^s\colon \Lambda \text{ is admissible, }\inf_{{\bf \tilde{z}} \in \ell_2(\Lambda)} \|{\bf z}-{\bf \tilde{z}}\|\leq \eps\big\}<\infty\Big\}.
\end{split} 
\end{equation}
A vector ${\bf z}$ is in $\cA^s$ if and only if  there exists a sequence of admissible $(\Lambda_i)_i$, with $\lim_{i \rightarrow \infty} \#\Lambda_i=\infty$, such that $\sup_i \inf_{{\bf z}_i \in \ell_2(\Lambda_i)} (\#\Lambda_i)^s \|{\bf z}-{\bf z}_i\|<\infty$. This means that ${\bf z}$ {\em can} be approximated in $\ell_2(\vee)$ at rate $s$ by vectors supported on admissible sets, or, equivalently, $z$ can be approximated in $\HH$ at rate $s$ from spaces of type $\Span\{\psi_\lambda\colon\lambda \in \Lambda,\,\Lambda \text{ is admissible}\}$.

The adaptive wavelet Galerkin method ({\bf awgm}) defined below produces a sequence of increasingly more accurate Galerkin approximations ${\bf z}_\Lambda$ to ${\bf z}$. The, generally, infinite residual ${\bf F}({\bf z}_\Lambda)$ is used as an a posteriori error estimator.
A motivation for the latter is given by the following result.
\begin{lemma} \label{lem12}
For $\|{\bf z}-{\bf \tilde{z}}\|$ sufficiently small, it holds that $\|{\bf F}({\bf \tilde{z}})\| \eqsim \|{\bf z}-{\bf \tilde{z}}\|$.
\end{lemma}
\begin{proof}
With $\tilde{z}={\bf \tilde{z}}^\top \Psi$, it holds that $\|{\bf F}({\bf \tilde{z}})\|\eqsim \|F(\tilde{z})\|_{\HH'}$. From \eqref{13}-\eqref{14}, we have
$F(\tilde{z})=DF(z)(\tilde{z}-z)+o(\|\tilde{z}-z\|_\HH)$.
The proof is completed by $\|DF(z)(\tilde{z}-z)\|_{\HH'} \eqsim \|\tilde{z}-z\|_\HH$ by \eqref{15}.
\end{proof}

This a posteriori error estimator guides  an appropriate enlargement of the current set $\Lambda$ using a bulk chasing strategy, so that the sequence of approximations converge with the best possible rate to ${\bf z}$. To arrive at an implementable method, that is even of optimal computational complexity, both the Galerkin solution and its residual are 
allowed to be computed inexactly within sufficiently small relative tolerances.

\begin{algorithm}[{\bf awgm}] \label{coordinates} \mbox{}
\begin{tabbing}
\quad\=\% Let $0 < \mu_0 \leq \mu_1< 1$, $\delta, \gamma>0$ be constants, $\Lambda_0 \subset \vee$ be admissible,\\
\> \%  and ${\bf z}_{\Lambda_0} \in \ell_2(\Lambda_0)$. Let ${\bf Z}$ be a neighborhood of ${\bf z} \in \ell_2(\vee)$.\\
\rule{0pt}{5mm} \>{\em \texttt{fo}}\={\em \texttt{r}} $i=0,1,\ldots$ {\em \texttt{do}} \\
\rule{0pt}{5mm}\>\> {\rm (R)} $\zeta:=\frac{2\delta}{1+\delta} \|{\bf r}_{i-1}\|$.\quad \% {\small {\rm(}Read $\|{\bf r}_{-1}\|$ as some scalar $\eqsim \|{\bf z}\|$.{\rm)}}
\\
\>\> {\em \texttt{do}}  $\zeta:=\zeta/2$; Compute ${\bf r}_i\in \ell_2(\vee)$ such that $\|{\bf r}_i-{\bf F}({\bf z}_{\Lambda_i})\| \leq \zeta$.\\ 
\>\> {\em \texttt{until}} $\zeta \leq \frac{\delta}{1+\delta} \|{\bf r}_i\|$.\\
\rule{0pt}{5mm}\> \> {\rm (B)} Determine  an admissible $\Lambda_{i+1} \supset \Lambda_i$ with $\|{\bf r}_i|_{\Lambda_{i+1}}\| \geq \mu_0 \|{\bf r}_i\|$ such that \\
\>\> $\#(\Lambda_{i+1} \setminus \Lambda_i) \lesssim \#(\tilde{\Lambda}\setminus \Lambda_i)$ for any admissible $\tilde{\Lambda} \supset \Lambda_i$ with
$\|{\bf r}_i|_{\tilde{\Lambda}}\| \geq \mu_1 \|{\bf r}_i\|$.\\
\rule{0pt}{5mm}\>\> {\rm (G)} Compute ${\bf z}_{\Lambda_{i+1}} \in \ell_2(\Lambda_{i+1}) \cap {\bf Z}$ with $\|{\bf F}({\bf z}_{\Lambda_{i+1}})|_{\Lambda_{i+1}}\| \leq \gamma \|{\bf r}_i\|$.\\
\rule{0pt}{5mm}\> {\em \texttt{endfor}}
 \end{tabbing}
\end{algorithm}

In step (R), by means of a loop in which an absolute tolerance is decreased, the true {\em residual} ${\bf F}({\bf z}_{\Lambda_i})$ is approximated within a relative tolerance $\delta$.
In step (B), {\em bulk chasing} is performed on the approximate residual. The idea is to find a smallest admissible $\Lambda_{i+1} \supset \Lambda_i$ with $\|{\bf r}_i|_{\Lambda_{i+1}}\| \geq \mu_0 \|{\bf r}_i\|$. In order to be able to find an implementation that is of linear complexity, the condition of having a truly smallest $\Lambda_{i+1}$ has been relaxed.
Finally, in step (G), a sufficiently accurate approximation of the {\em Galerkin} solution w.r.t. the new set $\Lambda_{i+1}$ is determined.

{\em Convergence} of the adaptive wavelet Galerkin method, with the {\em best possible rate}, is stated in the following theorem.

\begin{theorem}[{\cite[Thm.~3.9]{249.94}}]   \label{th2} Let $\mu_1, \gamma, \delta$, 
$\inf_{{\bf v}_{\Lambda_0} \in \ell_2(\Lambda_0)}\|{\bf z}-{\bf v}_{\Lambda_0}\|$, $\|{\bf F}({\bf z}_{\Lambda_0})|_{\Lambda_0}\|$, and the neighborhood ${\bf Z}$ of the solution ${\bf z}$ all be sufficiently small. Then, for some $\alpha=\alpha[\mu_0]<1$, the sequence $({\bf z}_{\Lambda_i})_i$ produced by {\bf awgm} satisfies
$$
\|{\bf z}-{\bf z}_{\Lambda_{i}}\| \lesssim \alpha^{i} \|{\bf z}-{\bf z}_{\Lambda_0}\|.
$$
If, for whatever $s>0$,  ${\bf z} \in \cA^s$, then $\# (\Lambda_{i+1} \setminus \Lambda_0) \lesssim  \|{\bf z}-{\bf z}_{\Lambda_{i}}\|^{-1/s}$.
\end{theorem}

The \emph{computation} of the approximate Galerkin solution ${\bf z}_{\Lambda_{i+1}}$ can be implemented by performing the simple fixed point iteration
\begin{equation} \label{fixedpoint}
{\bf z}_{\Lambda_{i+1}}^{(j+1)}={\bf z}_{\Lambda_{i+1}}^{(j)}-\omega {\bf F}({\bf z}_{\Lambda_{i+1}}^{(j)})|_{\Lambda_{i+1}}.
\end{equation}
Taking $\omega>0$ to be a sufficiently small constant and starting with ${\bf z}_{\Lambda_{i+1}}^{(0)}={\bf z}_{\Lambda_{i}}$, a fixed number of iterations suffices to meet the condition $\|{\bf F}({\bf z}^{(j+1)}_{\Lambda_{i+1}})|_{\Lambda_{i+1}}\| \leq \gamma \|{\bf r}_i\|$. This holds also true when each of the ${\bf F}()|_{\Lambda_{i+1}}$ evaluations is performed within an absolute tolerance that is a sufficiently small fixed multiple of $\|{\bf r}_i\|$.

Optimal {\em computational} complexity of the {\bf awgm} --meaning that the work to obtain an approximation within a given tolerance $\eps>0$ can be bounded on some constant multiple of the bound on its support length from Thm.~\ref{th2},--  is guaranteed under the following two conditions concerning the cost of  the ``bulk chasing'' process, and that of the approximate residual evaluation, respectively.
Indeed, apart from some obvious computations, these are the only two tasks that have to be performed in {\bf awgm}.

\begin{condition}\label{bulk}
The determination of $\Lambda_{i+1}$ in Algorithm~\ref{coordinates} is performed in ${\mathcal O}(\# \supp {\bf r}_i+\# \Lambda_i)$ operations.
\end{condition}

In case of unconstrained approximation, i.e., any finite $\Lambda \subset \vee$ is admissible, this condition is satisfied by collecting the largest entries in modulus of ${\bf r}_i$,    where, to avoid a suboptimal complexity, an exact sorting should be replaced by an approximate sorting based on binning.
With tree approximation, the condition is satisfied by the application of the so-called {\em Thresholding Second Algorithm} from \cite{20.5}. We refer to \cite[\S3.4]{249.94} for a discussion.

To understand the second condition, that in the introduction was referred to as the cost condition \eqref{condition},
note that inside the {\bf awgm} it is never needed to approximate a residual more accurately than within a sufficiently small, but fixed relative tolerance.

\begin{condition}\label{nonlineareval}
For a sufficiently small, fixed $\varsigma>0$, there exists a neighborhood ${\bf Z}$ of the solution ${\bf z}$ of ${\bf F}({\bf z})=0$, such that for all 
admissible $\Lambda \subset \vee$, ${\bf \tilde{z}}\in \ell_2(\Lambda) \cap {\bf Z}$, and any $\eps>0$, 
there exists an ${\bf r} \in \ell_2(\vee)$ with 
$$\|{\bf F}({\bf \tilde{z}})-{\bf r}\|\leq \varsigma \|{\bf F}({\bf \tilde{z}})\|+ \eps,
$$
that one can compute in ${\mathcal O}(\eps^{-1/s}+\#\Lambda)$ operations.
Here $s>0$ is such that ${\bf z} \in \cA^s$.
\end{condition}

Under both conditions, the {\bf awgm} has optimal computational complexity:

\begin{theorem} \label{optcomput} In the setting of Theorem~\ref{th2}, and under Conditions~\ref{bulk} and \ref{nonlineareval}, not only $\# {\bf z}_{\Lambda_i}$, but also the number of arithmetic operations required by {\bf awgm} for the computation of  ${\bf z}_{\Lambda_i}$
is ${\mathcal O}(\|{\bf z}-{\bf z}_{\Lambda_i}\|^{-1/s})$.
\end{theorem}



 

\section{Application to normal equations} \label{Snormal}
As discussed in Sect.~\ref{S1}, we will apply the {\bf awgm} to the (nonlinear) normal equations $DQ(u,\theta)=0$, with $DQ$ from \eqref{7}. That is, we apply the findings collected in the previous section for
the general triple $(F,\HH,z)$ now reading as $(DQ,\UU \times \PP,(u,\theta))$.

For $\Psi^\UU=\{\psi_\lambda^\UU\colon \lambda \in \vee_\UU\}$ and $\Psi^\PP=\{\psi_\lambda^\PP \colon \lambda \in \vee_\PP\}$ being Riesz bases
for $\UU$ and $\PP$, respectively, we equip $\UU \times \PP$ with Riesz basis
\begin{equation} \label{defpsi}
\Psi=\{\psi_\lambda\colon \lambda \in \vee:=\vee_\UU \cup \vee_\PP\}:=(\Psi^\UU,0_\PP) \cup (0_\UU,\Psi^\PP)
\end{equation}
(w.l.o.g. we assume that $\vee_\UU \cap \vee_\PP=\emptyset$).
With $\cF \in \Lis(\UU \times \PP,\ell_2(\vee))$ being the corresponding analysis operator, and $D{\bf Q}:=\cF DQ \cF'$,
for  $[{\bf \tilde{u}}^\top\!,\bm{\tilde{\theta}}^\top]^\top \in  \ell_2(\vee)$, and with $(\tilde{u},\tilde{\theta}):=[{\bf \tilde{u}}^\top\!,\bm{\tilde{\theta}}^\top]\Psi$, we have
\begin{equation} \label{3}
\begin{split}
D{\bf Q}([{\bf \tilde{u}}^\top\!,\bm{\tilde{\theta}}^\top]^\top) 
=&\left[\begin{array}{@{}c@{}}  
DG_0(\tilde{u}) (\Psi^\UU)(\Psi^\VV)^\top \\
G_1 (\Psi^\PP)(\Psi^\VV)^\top
\end{array} \right] 
\big(
G_0(\tilde{u})+ G_1 \tilde{\theta}
\big)
(\Psi^\VV)
\\
&+ 
\left\langle 
\left[\begin{array}{@{}c@{}}  
- G_2(\Psi^\UU) \\
\Psi^\PP
\end{array} \right],\tilde{\theta}-G_2 \tilde{u}\right\rangle_\PP.
\end{split}
\end{equation}

In this setting, using Lemma~\ref{lem12},  Condition~\ref{nonlineareval} can be reformulated as follows:
\begin{modcondition} \label{nonlineareval*}
For a sufficiently small, fixed $\varsigma>0$, there exists a neighborhood of the solution $(u,\theta)$ of $DQ(u,\theta)=0$, such that for all 
admissible $\Lambda \subset \vee$, all $[{\bf \tilde{u}}^\top\!,\bm{\tilde{\theta}}^\top]^\top \in \ell_2(\Lambda)$, with $(\tilde{u},\tilde{\theta}):=[{\bf \tilde{u}}^\top\!,\bm{\tilde{\theta}}^\top] \Psi$ being in this neighborhood, and any $\eps>0$,
there exists an ${\bf r} \in \ell_2(\vee)$ with 
$$\|D{\bf Q}([{\bf \tilde{u}}^\top\!,\bm{\tilde{\theta}}^\top]^\top)-{\bf r}\|\leq \varsigma (\|u-\tilde{u}\|_\UU+\|\theta-\tilde{\theta}\|_\PP) +
\eps,
$$
that one can compute in ${\mathcal O}(\eps^{-1/s}+\#\Lambda)$ operations, where 
$s>0$ is such that $[{\bf u}^\top\!, {\bf \theta}^\top]^\top \in \cA^s$.
\end{modcondition}

To verify this condition, we will use the additional property, i.e. on top of \eqref{38}--\eqref{41}, that
$\|u-\tilde{u}\|_\UU+\|\theta-\tilde{\theta}\|_\PP \eqsim \|G_0(\tilde{u})-G_1\tilde{\theta}\|_{\VV'}+\|\tilde{\theta}-G_2\tilde{u}\|_{\PP}$, which is provided by Lemma~\ref{lem2}.

\section{Semi-linear 2nd order elliptic PDE} \label{Selliptic}
We apply the solution method outlined in Sect.~\ref{S1}-\ref{Snormal} to the example of a semi-linear 2nd order elliptic PDE with general (inhomogeneous) boundary conditions.
The main task will be to verify Condition~\ref{nonlineareval*}.
\subsection{Reformulation as a first order system least squares problem} \label{Sreformulation}
Let $\Omega \subset \R^n$ be a bounded domain, $\Gamma_N \cup \Gamma_D=\partial\Omega$ with $\meas(\Gamma_N \cap \Gamma_D)=0$, 
$\meas(\Gamma_D)>0$ when $\Gamma_D \neq \emptyset$,
and $A:\Omega \rightarrow \R^{n\times n}_{\rm symm}$ with $\xi^\top A(\cdot) \xi \eqsim \|\xi\|^2$ ($\xi \in \R^n$, a.e.).
We set
$$
\framebox{$\UU:=H^1(\Omega)$} \,\,\,\text{or, in case } \meas(\Gamma_D)=0, \text{ possibly } \framebox{$\UU:=H^1(\Omega)/\R$},
$$
and
$$\framebox{$\VV=\VV_1\times \VV_2 :=\{u \in \UU\colon u|_{\Gamma_D}=0\} \times H^{-\frac{1}{2}}(\Gamma_D)$}.
$$
For $N:\UU \supset \dom(N) \rightarrow \VV_1'$, $f \in  \VV_1'$, $g \in H^{\frac{1}{2}}(\Gamma_D)$, and $h \in H^{-\frac{1}{2}}(\Gamma_N)$, we consider 
the semi-linear boundary value problem
\begin{equation} \label{bvp}
 \left\{
\begin{array}{r@{}c@{}ll}
-\divv A \grad u+ N(u) &\,\,=\,\,& f &\text{ on } \Omega,\\
u &\,\,=\,\,& g &\text{ on } \Gamma_D,\\
A \grad u \cdot {\bf n} &\,\,=\,\,& h &\text{ on } \Gamma_N,
\end{array}
\right.
\end{equation}
that in standard variational form reads as finding $u \in \UU$ such that
$$
(Gu)(v):=\int_\Omega A \grad u \cdot \grad v_1 + (N(u)-f) v_1 dx-\int_{\Gamma_N} h v_1 \,ds +\int_{\Gamma_D} (u-g)v_2 \,ds=0
$$
$(v=(v_1,v_2) \in \VV)$.

We assume that this variational problem has a solution $u$, and that $G$, i.e., $N$, is two times continuously Fr\'{e}chet differentiable in a neighborhood of $u$, and $DG(u) \in \cL(\UU,\VV')$ is a homeomorphism with its range, i.e., we {\em assume} that
$$
G:\UU \supset \dom(G) \rightarrow \VV' \text{ satisfies } \eqref{r1}-\eqref{r3} 
$$
formulated in Sect.~\ref{S1}.
\begin{remarks}
Because $\UU \rightarrow H^{\frac{1}{2}}(\Gamma_D)\colon u \mapsto u|_{\partial\Omega}$ is surjective, 
from \cite[Thm.~2.1]{249.96} it follows that condition \eqref{r3} is satisfied when
$$
L:= w \mapsto \Big(v_1 \mapsto \int_\Omega A \grad w \cdot \grad v_1+ DN(u)(w) v_1\,dx\Big) \in \Lis(\VV_1,\VV_1') 
$$
(actually, $L$ being a homeomorphism with its range is already sufficient).

By writing $g=u_0|_{\Gamma_D}$ for some $u_0 \in \UU$, one infers that 
for linear $N$, existence of a (unique) solution $u$, i.e. \eqref{r1}, follows from $L \in \Lis(\VV_1,\VV_1')$.
For $g=0$, the conditions of $N$ being monotone and locally Lipschitz are sufficient for having a (unique) solution $u$.
Relaxed conditions on $N$ suffice to have a (locally unique) solution. We refer to \cite{18.68}.
\end{remarks}

Using the framework outlined in Sect.~\ref{S1}, we write this second order elliptic PDE as a first order system least squares problem.
Putting \framebox{$\PP=L_2(\Omega)^n$}, we define
$$
G_1 \in \cL(\PP,\VV'),\qquad G_2 \in \cL(\UU,\PP),
$$
by
$$
G_2 u=A \grad u,\qquad (G_1 \vec{\theta})(v_1,v_2)=\int_\Omega \vec{\theta} \cdot \grad v_1 \,d x.
$$

The results from Sect.~\ref{S1} show that the solution $u$ can be found as the first component of the minimizer $(u,\vec{\theta}) \in  \UU \times \PP$ of
\begin{equation}  \label{def_Q}
\begin{split}
Q(u,\vec{\theta}):={\textstyle \frac{1}{2}}\Big(&\big\|v_1 \mapsto 
\int_\Omega \vec{\theta} \cdot \grad v_1  +(N(u)-f)v_1\,dx  -\int_{\Gamma_N} v_1 h \,ds
\big\|_{\VV_1'}^2+\\
&\|\vec{\theta}-A \grad u\|_{L_2(\Omega)^n}^2+\|u-g\|^2_{\VV_2'}\Big),
\end{split}
\end{equation}
being the solution of the normal equations $DQ(u,\vec{\theta})=0$, and furthermore, 
that these normal equations are well-posed in the sense that they satisfy \eqref{38}--\eqref{41}.

To deal with the `unpractical'  norm on $\VV'$, as in Sect.~\ref{Sseemingly}, at the end of Sect.~\ref{S1}, and in Sect.~\ref{Snormal}, we equip $\VV_1$ and $\VV_2$
 with wavelet {\em Riesz bases}
$$
\Psi^{\VV_1}=\{\psi_\lambda^{\VV_1}\colon\lambda \in \vee_{\VV_1}\}, \quad \Psi^{\VV_2}=\{\psi_\lambda^{\VV_2}\colon\lambda \in \vee_{\VV_2}\},
$$
and replace, in the definition of $Q$,  the norms on their duals by the equivalent norms defined by
$\|g(\Psi^{\VV_1})\|$ or $\|g(\Psi^{\VV_2})\|$, for $g \in \VV_1'$ or $g \in \VV_2'$, respectively.

Next, after equiping $\UU$ and $\PP$ with {\em Riesz bases} 
$$
\Psi^{\UU}=\{\psi_\lambda^{\UU}\colon\lambda \in \vee_{\UU}\}, \quad \Psi^{\PP}=\{\psi_\lambda^{\PP}\colon\lambda \in \vee_{\PP}\},
$$
and so $\UU \times \PP$ with $\Psi=(\Psi^{\UU},0_{\PP}) \cup (0_{\UU},\Psi^{\PP})$, we apply the {\bf awgm} to the resulting system 
\[
\begin{split}
 &D{\bf Q}([{\bf {u}}^\top\!,\bm{{\theta}}^\top]^\top)\!=\!
\Big\langle \!\!
\left[\begin{array}{@{}c@{}}  
A \grad \Psi^{\UU}
 \\
-\Psi^{\PP}
\end{array} \right] \!\!
,\!A \grad {u} \!-\!\vec{{\theta}}\Big\rangle_{L_2(\Omega)^n} 
\!\!\!\!+\!\!
\left[\begin{array}{@{}c@{}}  
\langle \Psi^{\UU}\!, \! \Psi^{\VV_2} \rangle_{L_2(\Gamma_D)}
 \\
0_{\vee_\PP}
\end{array} \right] \!\!
\langle  \Psi^{\VV_2}, {u}\!-\!g\rangle_{L_2(\Gamma_D)}+\\
&\left[\!\begin{array}{@{}c@{}}  
\langle DN({u}) \Psi^{\UU}\!, \! \Psi^{\VV_1} \rangle_{L_2(\Omega)}
 \\
\langle \Psi^{\PP},\grad \Psi^{\VV_1} \rangle_{L_2(\Omega)^n}
\end{array}\! \right] \!\!
\Big\{\!
\big\langle \Psi^{\VV_1}\!,\! N({u})\!-\!f\big\rangle_{L_2(\Omega)}\!\!-\!\big\langle \Psi^{\VV_1}\!,\!h \big\rangle_{L_2(\Gamma_N)}
\!\!+\!\!\langle \nabla \Psi^{\VV_1}\!,\! \vec{{\theta}} \rangle_{L_2(\Omega)^n}
\!\!\Big\}=0,
\end{split}
\]
where $(u,\vec{{\theta}}):=[{\bf {u}}^\top\!,\bm{{\theta}}^\top] \Psi$.

To express the three terms in $v\mapsto \langle v, N({u})-f\rangle_{L_2(\Omega)}-\langle v,h \rangle_{L_2(\Gamma_N)}
+\langle \nabla v, \vec{{\theta}} \rangle_{L_2(\Omega)^n} \in \VV_1'$ w.r.t. one dictionary of functions on $\Omega$ and one dictionary of functions on $\Gamma_N$, similarly to Sect.~\ref{Sseemingly} we impose the additional, but in applications easily realisable condition that
\begin{equation} \label{8}
\Psi^{\PP} \subset H(\divv;\Omega).
\end{equation}
Then for finitely supported approximations $[{\bf \tilde{u}}^\top\!,\bm{\tilde{\theta}}^\top]^\top$ to $[{\bf {u}}^\top\!,\bm{{\theta}}^\top]^\top$, for $(\tilde{u},\vec{\tilde{\theta}}):=[{\bf \tilde{u}}^\top\!,\bm{\tilde{\theta}}^\top] \Psi \in \UU \times H(\divv;\Omega)$, 
we have
\begin{equation} \label{29}
\framebox{$
\begin{split}
 &\!\!D{\bf Q}([{\bf {\tilde{u}}}^\top\!,\!\bm{\tilde{\theta}}^\top]^\top)\!=\!
\Big\langle \!
\left[\!\begin{array}{@{}c@{}}  
A \grad \Psi^{\UU}
 \\
-\Psi^{\PP}
\end{array} \!\right] \!
,\!A \grad {\tilde{u}} \!-\!\vec{\tilde{\theta}}\Big\rangle_{L_2(\Omega)^n} 
\!\!\!\!+\!\!
\left[\begin{array}{@{}c@{}}  
\!\langle \Psi^{\UU}\!,\!  \Psi^{\VV_2} \rangle_{L_2(\Gamma_D)}
 \\
0_{\vee_\PP}
\end{array} \!\right] \!\!
\langle  \Psi^{\VV_2}\!,\! {\tilde{u}}\!-\!g\rangle_{L_2(\Gamma_D)}+\\
&\!\!\left[\!\begin{array}{@{}c@{}}  
\langle DN({\tilde{u}}) \Psi^{\UU},  \Psi^{\VV_1} \rangle_{L_2(\Omega)}
 \\
\langle \Psi^{\PP},\grad \Psi^{\VV_1} \rangle_{L_2(\Omega)^n}
\end{array}\! \right] \!\!
\Big\{\!
\big\langle \Psi^{\VV_1}, N({\tilde{u}})\!-\!f\!-\!\divv \vec{\tilde{\theta}} \big\rangle_{L_2(\Omega)}\!\!+\!\big\langle \Psi^{\VV_1},\vec{\tilde{\theta}}\cdot{\bf n}-h \big\rangle_{L_2(\Gamma_N)}
\!\!\Big\},
\end{split}
$}\hspace*{-1em}
\end{equation}
where we used the vanishing traces of  $v \in \VV_1 $ at $\Gamma_D$, to write 
$\langle \grad v, \vec{\tilde{\theta}}\rangle_{L_2(\Omega)^n}$
as $\langle v,- \divv \vec{\tilde{\theta}}\rangle_{L_2(\Omega)}+\langle v, \vec{\tilde{\theta}}\cdot {\bf n}\rangle_{L_2(\Gamma_N)}$.

Each of the terms  $A \grad {\tilde{u}} -\vec{\tilde{\theta}}$, ${\tilde{u}}-g$, $N({\tilde{u}})-f-\divv \vec{\tilde{\theta}}$, and $\vec{\tilde{\theta}}\cdot{\bf n}-h$
correspond, in strong form, to a term of the least squares functional, and therefore their norms can be bounded by a multiple of the norm of the residual, which is the basis of our approximate residual evaluation.
In order to verify Condition~\ref{nonlineareval*}, we have to collect some assumptions on the wavelets, which will be done in the next subsection.
\begin{remark} \label{homDir} If $\Gamma_D=\emptyset$, then obviously \eqref{29} should be read without the second term involving $\Psi^{\VV_2}$.
If  $\Gamma_D\neq \emptyset$ and {\em homogeneous} Dirichlet boundary conditions are prescribed on $\Gamma_D$, i.e., $g=0$, it is simpler to select $\UU=\VV_1=\{u \in H^1(\Omega)\colon u|_{\Gamma_D}=0\}$, and to omit integral over $\Gamma_D$ in the definition of $G$, so that again \eqref{29} should be read without the second term involving $\Psi^{\VV_2}$.
\end{remark}


\subsection{Wavelet assumptions and definitions} \label{Swavelets}
We formulate conditions on $\Psi^{\VV_1}$, $\Psi^{\VV_2}$, $\Psi^{\UU}$, and $\Psi^{\PP}$, in addition to being Riesz bases for $\VV_1$, $\VV_2$, $\UU$, and $\PP$, respectively.

Recalling that $\PP=\PP_1\times\cdots\times\PP_n$, we select $\Psi^\PP$ of canonical form
$$
(\Psi^{\PP_1},0_{\PP_2},\ldots,0_{\PP_n}) \cup \cdots \cup (0_{\PP_1},\ldots,0_{\PP_{n-1}},\Psi^{\PP_n}),
$$
where $\Psi^{\PP_q}=\{\psi^{\PP_q}_\lambda\colon \lambda \in \vee_{\PP_q}\}$ is a Riesz basis for $\PP_q$ (with $\vee_{\PP_{q'}} \cap \vee_{\PP_{q"}} =\emptyset$ when $q'\neq q''$).

For $\ast \in \{\UU,\PP_1,\ldots,\PP_n,\VV_1\}$, we collect a number of (standard) assumptions, \eqref{w1}--\eqref{w8}, on the scalar-valued wavelet collections $\Psi^\ast=\{\psi_\lambda^\ast\colon \lambda \in \vee_\ast\}$ on $\Omega$.
Corresponding assumptions on the wavelets $\Psi^{\VV_2}$ on $\Gamma_D$ will be formulated at the end of this subsection.
To each $\lambda \in \vee_\ast$, we associate a value $|\lambda| \in \N_0$, which is called the {\em level} of $\lambda$.
We will assume that the elements of $\Psi^\ast$ have {\em one vanishing moment}, and are {\em locally supported, piecewise polynomial} of some degree $m$, w.r.t.\ {\em dyadically nested partitions} in the following sense:

\begin{enumerate} \renewcommand{\theenumi}{$w_{\arabic{enumi}}$}
\item \label{w1} There exists a collection $\cO_\Omega:=\{\omega\colon \omega \in \cO_\Omega\}$ of closed polytopes,  such that, with $|\omega|\in \N_0$ being the {\em level} of $\omega$, $\meas(\omega \cap \omega')=0$ when $|\omega|=|\omega'|$ and $\omega \neq \omega'$;
for any $\ell \in \N_0$, $\bar{\Omega}=\cup_{|\omega|=\ell} \omega$; ${\rm diam}\,\omega \eqsim 2^{-|\omega|}$; and $\omega$ is the union of  $\omega'$ for some $\omega'$ with $|\omega'|=|\omega|+1$. We call $\omega$ the {\em parent} of its {\em children} $\omega'$.
Moreover, we assume that the $\omega \in \cO_\Omega$ are uniformly {\em shape regular}, in the sense that they satisfy a uniform Lipschitz condition, and  $\meas(F_\omega)\eqsim \meas(\omega)^{\frac{n-1}{n}}$ for $F_\omega$ being any facet of $\omega$.
\item\label{w2} ${\rm supp}\,\psi^\ast_\lambda$ is contained in a connected union of a uniformly bounded number of $\omega$'s with $|\omega|=|\lambda|$, and restricted to each of these $\omega$'s is $\psi^\ast_\lambda$ a polynomial of degree $m$.
\item\label{w3} Each $\omega$ is intersected by the supports of a uniformly bounded number of $\psi^\ast_\lambda$'s with  $|\lambda|=|\omega|$.
\item\label{w8} $\int_\Omega \psi^\ast_\lambda \,dx =0$, possibly with the exception of those $\lambda$ with ${\rm dist}({\rm supp}\,\psi^\ast_\lambda,\Gamma_D) \lesssim 2^{-|\lambda|}$, or with $|\lambda|=0$.
\end{enumerate}
Generally, the polynomial degree $m$ will be different for the different bases, but otherwise fixed.
The collection $\cO_\Omega$ is shared among all bases.
Note that the conditions of $\Psi^\UU$ being a basis for $\UU$, and to consist of piecewise polynomials, implies that $\UU \subset C(\bar{\Omega})$.
Wavelets of in principle arbitrary order that satisfy these assumptions can be found in e.g. \cite{56,239.17}.

\begin{definition} \label{def_tiling} A collection ${\mathcal T} \subset \cO_\Omega$ such that $\overline{\Omega}=\cup_{\omega \in {\mathcal T}} \omega$, and for $\omega_1 \neq \omega_2 \in {\mathcal T}$, $\meas(\omega_1 \cap \omega_2)=0$ will be called a {\em tiling}.
With $\cP_m(\tria)$, we denote the space of piecewise polynomials of degree $m$ w.r.t.\ $\tria$.
The smallest common refinement of tilings ${\mathcal T}_1$ and ${\mathcal T}_2$ is denoted as ${\mathcal T}_1 \oplus {\mathcal T}_2$.
\end{definition}

To be able to find, in linear complexity, a representation of a function, given as linear combination of wavelets, as a piecewise polynomial w.r.t.\ a tiling --mandatory for an efficient evaluation of nonlinear terms--, 
we will impose a tree constraint on the underlying 
set of wavelet indices. A similar approach was followed earlier in \cite{56.2, 45.22, 316.5, 22.6, 310.5}.

\begin{definition} \label{def1} 
To each $\lambda \in \vee_*$ with $|\lambda|>0$, we associate one $\mu \in \vee_*$ with $|\mu|=|\lambda|-1$ and $\meas(\supp \psi^*_\lambda \cap \supp \psi^*_\mu)>0$.
We call $\mu$ the {\em parent} of $\lambda$, and so $\lambda$ a {\em child} of $\mu$.

To each $\lambda \in \vee_*$, we associate some neighbourhood  $\cS(\psi^*_\lambda)$ of $\supp \psi^*_\lambda$, with diameter $\lesssim2^{-|\lambda|}$, such that
$\cS(\psi^*_\lambda) \subset \cS(\psi^*_\mu)$ when $\lambda$ is a child of $\mu$.

We call a finite $\Lambda \subset \vee_\ast$ a {\em tree}, if it contains all $\lambda \in \vee_*$ with $|\lambda|=0$, as well as the parent of any $\lambda \in \Lambda$ with $|\lambda|>0$.
\end{definition}

Note that we now have tree structures on the set $\cO_\Omega$ of polytopes, and as well as on the wavelet index sets $\vee_*$. We trust that no confusion will arise when we speak about parents or children.

For some collections of wavelets, as the Haar or more generally, Alpert wavelets (\cite{10}), it suffices to take
$\cS(\psi^*_\lambda):=\supp \psi^*_\lambda$.
The next result shows that, thanks to \eqref{w1}-\eqref{w2}, a suitable neighbourhood $\cS(\psi^*_\lambda)$ as meant in Definition~\ref{def1} always exists.
\begin{lemma} With $C:=\sup_{\lambda \in \vee_*} 2^{|\lambda|} \diam \supp \psi^*_\lambda$, a valid choice of 
$\cS(\psi^*_\lambda)$ is given by 
$\{x \in \Omega\colon \dist(x,\supp \psi_\lambda^*) \leq C 2^{-|\lambda|}\}$.
\end{lemma}

\begin{proof} For $\mu,\lambda \in \vee_*$ with $|\mu|=|\lambda|-1$ and $\meas(\supp \psi^*_\lambda \cap \supp \psi^*_\mu)>0$, and $x \in \Omega$ with $\dist(x,\supp \psi_\lambda^*) \leq C 2^{-|\lambda|}$, it holds that
$\dist(x,\supp \psi_\mu^*) \leq \dist(x,\supp \psi_\lambda^*)+\diam(\supp \psi^*_\lambda) \leq C2^{-|\mu|}$.
\end{proof}


A proof of the following proposition, as well as an algorithm to apply the multi-to-single-scale transformation that is mentioned, is given in \cite[\S4.3]{249.94}.

\begin{proposition} \label{prop4}
Given a tree $\Lambda \subset \vee_\ast$, there exists a tiling ${\mathcal T}({\Lambda}) \subset \cO_\Omega$ with $\# {\mathcal T}({\Lambda}) \lesssim \# \Lambda$ such that 
$\Span\{ \psi^\ast_\lambda \colon \lambda \in \Lambda\}\subset \cP_{m}({\mathcal T}({\Lambda}))$.
Moreover, equipping $\cP_{m}(\tria(\Lambda))$  with a basis of functions, each of which supported in $\omega$ for one  $\omega \in \tria(\Lambda)$, the  representation of this embedding, known as the multi- to single-scale transform, can be applied in ${\mathcal O}(\#\Lambda)$ operations.
\end{proposition}


The benefit of the definition of $\cS(\psi_\lambda^*)$ appears from the following lemma.

\begin{lemma} \label{lem-is-tree} Let $\overline{\Omega}=\Sigma_0 \supseteq \Sigma_1 \supseteq \cdots$.
Then
$$
\Lambda^*:=\big\{\lambda \in \vee_* \colon \meas(\cS(\psi_\lambda^*) \cap \Sigma_{|\lambda|})>0\big\}
$$
is a tree.
\end{lemma}

\begin{proof}
The set $\Lambda^*$ contains all $\lambda \in \vee_*$ with $|\lambda|=0$, as well as, by definition of $\cS(\cdot)$, the parent of any $\lambda \in \Lambda^*$ with $|\lambda|>0$.
\end{proof}

In Proposition~\ref{prop4} we saw that for each tree $\Lambda$ there exists a tiling ${\mathcal T}(\Lambda)$, with $\# {\mathcal T}(\Lambda) \lesssim \#\Lambda$, such that $\Span\{ \psi^\ast_\lambda \colon \lambda \in \Lambda\}\subset \cP_{m}({\mathcal T}({\Lambda}))$.
Conversely, in the following, given a tiling ${\mathcal T}$, and a constant $k \in \N_0$, we construct a tree $\Lambda^\ast({{\mathcal T},k})$ with $\# \Lambda^\ast({{\mathcal T},k}) \lesssim \# {\mathcal T}$ (dependent on $k$) such that a kind of reversed statements hold:
In Appendix~\ref{Sdecay}, statements of type $\lim_{k \rightarrow \infty} \sup_{0  \neq g \in \cP_{m}({\mathcal T})} \frac{\|\langle \Psi^*,g\rangle_{L_2(\Omega)}|_{\vee_*\setminus \Lambda^\ast({{\mathcal T},k})}\|}{\|g\|_{*'}} = 0$ will be shown, meaning that 
for any tolerance there exist a $k$ such that for any $g  \in \cP_{m}({\mathcal T})$ the relative error in dual norm in the best approximation from the span of the corresponding dual wavelets with indices in $\Lambda^\ast({{\mathcal T},k})$ is less than this tolerance.

\begin{definition} \label{partition-to-lambda}
Given a tiling ${\mathcal T} \subset \cO_\Omega$, let $t({\mathcal T}) \subset \cO_\Omega$ be its enlargement by adding all ancestors of all $\omega \in {\mathcal T}$.
Given a $k  \in \N_0$, we set 
$$
\Lambda^\ast({{\mathcal T},k}):=\big\{\lambda \in \vee_* \colon \meas\big(\cS(\psi_\lambda^*) \cap  
\bigcup_{\{\omega \in t({\mathcal T})\colon |\omega|=\max(|\lambda|-k,0)\}} \omega
\big)>0\big\}.
$$
\end{definition}

\begin{proposition} 
The set $\Lambda^\ast({{\mathcal T},k})$ is a tree, and 
 $\# \Lambda^\ast({{\mathcal T},k}) \lesssim \# {\mathcal T}$ {\rm(}dependent on $k \in \N_0${\rm)}.
\end{proposition}

\begin{proof}
The first statement follows from Lemma~\ref{lem-is-tree}.
Since the number of children of any  $\omega \in \cO_\Omega$ is uniformly bounded, it holds that $\# t({\mathcal T})\lesssim \# {\mathcal T}$,
and so $\# \Lambda^\ast({{\mathcal T},k}) \lesssim \# {\mathcal T}$ as a consequence of the wavelets being locally supported.
\end{proof}

\begin{example} Let $\Psi=\{\psi_\lambda\colon \lambda \in \vee\}$ be the collection of Haar wavelets on $\Omega=(0,1)$, i.e., the union of the function $\psi_{0,0} \equiv 1$, and, for $\ell \in \N$ and $k=0,\ldots,2^{\ell-1}-1$, the functions $\psi_{\ell,k}:=2^{\frac{\ell-1}{2}}\psi(2^{\ell-1}(\cdot-k))$, where $\psi\equiv 1$ on $[0,\frac{1}{2}]$ and $\psi\equiv -1$ on $(\frac{1}{2},1]$.
Writing $\lambda=(\ell,k)$, we set $|\lambda|=\ell$.
The parent of $\lambda$ with $|\lambda|>0$ is $\mu$ with $|\mu|=|\lambda|-1$ and  $\supp \psi_\lambda \subset \supp \psi_\mu$, and $\cS(\psi_\lambda):=\supp \psi_\lambda$.

Let $\cO_{\Omega}$ be the union, for $\ell \in \N_0$ and $k=0,\ldots,2^\ell-1$, of the intervals $2^{\ell}[k,k+1]$ to which we assign the level $\ell$.

Now as an example let $\Lambda \subset \vee$ be the set $\{(0,0),(1,0),(2,0),(3,0)\}$, which is a tree in the sense of Definition~\ref{def1}.
It corresponds to the solid parts in the left picture in Figure~\ref{fig1}.
\begin{figure}[h]
\begin{center}
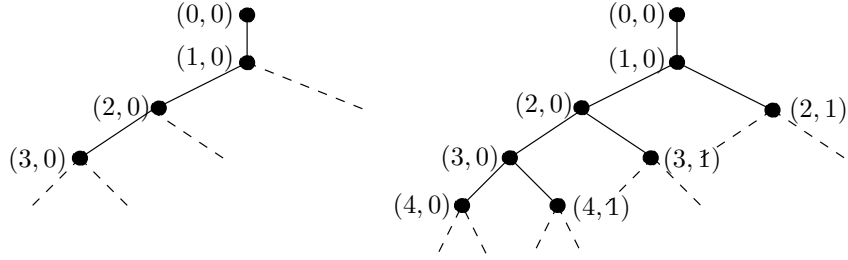
\end{center}
\caption{The index set of the Haar basis given as an infinite binary tree, and its subsets $\Lambda=\{(0,0),(1,0),(2,0),(3,0)\}$ (left) and 
$\Lambda(\tria(\Lambda),1)$ (right).}
\label{fig1}
\end{figure}
The (minimal) tiling ${\mathcal T}(\Lambda)$ as defined in Proposition~\ref{prop4} is given by $\{[0,\frac{1}{8}],[\frac{1}{8},\frac{1}{4}], [\frac{1}{4},\frac{1}{2}], [\frac{1}{2},1]\}$.

Conversely, taking ${\mathcal T}:={\mathcal T}(\Lambda)$, the set $\Lambda(\tria,1) \subset \vee$ as defined in Definition~\ref{partition-to-lambda} is given by $\{(0,0),(1,0),(2,0),(2,1),(3,0),(3,1), (4,0), (4,1)\}$ and is illustrated in the right picture in Figure~\ref{fig1}.
\end{example}

\begin{definition} \label{def_admissible}
Recalling from \eqref{defpsi} and the first lines of this subsection
that the Riesz basis $\Psi$ for $\UU \times \PP$
is of canonical form
$$
\Psi=(\Psi^{\UU},0_{\PP_1},\ldots,0_{\PP_n}) \cup (0_\UU,\Psi^{\PP_1},0_{\PP_2},\ldots,0_{\PP_n}) \cup \cdots\cup
(0_\UU,0_{\PP_1},\ldots,0_{\PP_{n-1}},\Psi^{\PP_n}),
$$
with index set $\vee:=\vee_{\UU} \cup \vee_{\PP_1} \cup \cdots \cup \vee_{\PP_n}$,
we call $\Lambda \subset \vee$ {\em admissible} when
each of  $\Lambda \cap \vee_{\UU},\Lambda \cap \vee_{\PP_1},\,\ldots,\Lambda \cap \vee_{\PP_n}$ are trees. 
The tiling ${\mathcal T}(\Lambda)$ is defined as the smallest common refinement ${\mathcal T}(\Lambda \cap \vee_{\UU}) \oplus {\mathcal T}(\Lambda \cap \vee_{\PP_1})\oplus \cdots \oplus  {\mathcal T}(\Lambda \cap \vee_{\PP_n})$.
Conversely, given a tiling ${\mathcal T} \subset \cO_\Omega$ and a $k \in \N_0$, we define the admissible set $\Lambda({\mathcal T},k) \subset \vee$ by
$\Lambda^\UU({\mathcal T},k)\cup \Lambda^{\PP_1}({\mathcal T},k) \cup \cdots \cup \Lambda^{\PP_n}({\mathcal T},k)$.
\end{definition}

\begin{remark} \label{rem_order} Let $\Psi^{\UU}$ be a wavelet basis for $\UU$ of order $d_\UU>1$ (i.e., all wavelets $\psi_\lambda^{\UU}$ up to level $\ell$ span all piecewise polynomials in $\UU$ of degree $d_\UU-1$ w.r.t.\ $\{\omega\colon \omega \in \cO_\Omega,\,|\omega|=\ell\}$), and similarly, for $1 \leq q \leq n$, let $\Psi^{\PP_q}$ be a wavelet basis for $\PP_q$ of order $d_\PP>0$. 
Recalling the definition of an approximation class given in \eqref{class}, 
 a sufficiently smooth solution $(u,\vec{\theta})$ is in $\cA^s$ for $s=s_{\max}:=\min(\frac{d_\UU-1}{n},\frac{d_\PP}{n})$, whereas on the other hand membership of $\cA^s$ for $s>s_{\max}$ cannot be expected under whatever smoothness condition.
 
 For $s \leq s_{\max}$, a sufficient and `nearly' necessary condition for $(u,\vec{\theta})\in \cA^s$ is that $(u,\vec{\theta}) \in B^{s n+1}_{p,\tau}(\Omega) \times B^{s n}_{p,\tau}(\Omega)^n$ for $\frac{1}{p}<s+\frac{1}{2}$ and arbitrary $\tau>0$, see \cite{45.21}. This mild smoothness condition in the `Besov scale' has to be compared to the condition  $(u,\vec{\theta}) \in H^{s n+1}(\Omega) \times H^{s n}(\Omega)^n$ that is necessary to obtain a rate $s$ with approximation from the spaces of type $\Span \{\psi^\UU_\lambda\colon|\lambda|\leq L\} \times \prod_{q=1}^n \Span \{\psi^{\PP_q}_\lambda\colon|\lambda|\leq L\}$.
\end{remark}

\medskip

We pause to add {\em three more assumptions} on our PDE:
We assume that w.r.t.\ the coarsest possible tiling $\{\omega \colon \omega \in \cO_\Omega,\,|\omega|=0\}$ of $\bar{\Omega}$,
\begin{align}\label{16}
&A \text{ is piecewise polynomial,}\\ \label{17}
& \parbox{11 cm}{$N(u)$  is a partial differential operator of at most first order, with coefficients that are piecewise polynomials in $u$  and $x$, and}\\
& \text{$\overline{\Gamma}_N$, and so $\overline{\Gamma}_D$, is the union of facets of $\omega \in \cO_\Omega$ with $|\omega|=0$}.
\end{align}

\begin{remark}
The subsequent analysis can easily be generalized to $A$ being piecewise smooth.
With some more efforts other nonlinear terms $N$ can be handled as well.
For example, for $N(u)$ of the form $n_1(u) u$, it will be needed that for some 
$m \in \N$, for each $\omega \in \cO_\Omega$, there exists a
subspace $V_\omega \subset H^1_0(\omega)$ with $\|\cdot\|_{H^1(\omega)} \lesssim 2^{|\omega|} \|\cdot\|_{L_2(\omega)}$ on $V_\omega$, and 
$$
\|N(p_1) +p_2\|_{L_2(\omega)}  \lesssim \sup_{0 \neq v \in V_\omega} \frac{\langle N(p_1) +p_2,v\rangle_{L_2(\omega)}}{\|v\|_{L_2(\omega)}} \quad (p_1, p_2 \in \cP_m(\omega))
$$
(cf. proof of Lemma~\ref{lem3}).
\end{remark}

Finally in this subsection we formulate our assumptions on the wavelet Riesz basis $\Psi^{\VV_2}=\{\psi^{\VV_2}_\lambda \colon \lambda \in \vee_{\VV_2}\}$ for $H^{-\frac{1}{2}}(\Gamma_D)$.
We assume that it satisfies the assumptions \eqref{w2} and \eqref{w3} with $\cO_\Omega$ reading as
$$
\cO_{\Gamma_D}=\{\omega \cap \Gamma_D\colon \omega \in \cO_\Omega,\,\meas_{n-1}(\omega \cap \Gamma_D)>0\}.
$$
Furthermore, we impose that
\begin{equation} \label{substitute}
\|\psi^{\VV_2}_\lambda\|_{H^{-1}(\Gamma_D)} \lesssim 2^{-|\lambda|/2}.
\end{equation}
which, for biorthogonal wavelets, is essentially is \eqref{w8} (cf. \cite[lines following (A.2)]{56}).
(To relax the smoothness conditions on $\Gamma_D$ needed for the definition of $H^{-1}(\Gamma_D)$, one can replace \eqref{substitute} by
$\|\psi^{\VV_2}_\lambda\|_{H^s(\Gamma_D)} \lesssim 2^{|\lambda|(s+\frac{1}{2})}$ for some $s \in [-1,-\frac{1}{2})$).

The definition of a boundary tiling ${\mathcal T}_{\Gamma_D} \subset \cO_{\Gamma_D}$ is similar to Definition~\ref{def_tiling}.
Also similar to the corresponding preceding definitions are that of a tree $\Lambda \subset \vee_{\VV_2}$, and of the boundary tiling ${\mathcal T}_{\Gamma_D}(\Lambda) \subset \cO_{\Gamma_D}$ for 
$\Lambda \subset \vee_{\VV_2}$ being a tree.
Conversely, for a boundary tiling ${\mathcal T}_{\Gamma_D}\subset \cO_{\Gamma_D}$ and $k \in \N_0$, for $\ast \in \{\UU,\VV_2\}$ we define the tree
$$
\Lambda^\ast({{\mathcal T}_{\Gamma_D},k}):=
\big\{\lambda \in \vee_\ast \colon \meas_{n-1}(\cS(\psi_\lambda^\ast) \cap
\bigcup_{\{\omega \in \bar{\tria}_{\Gamma_D}\colon |\omega|= \max(|\lambda|-k,0)\}} \omega)>0\big\}.
$$

\subsection{An appropriate approximate residual evaluation} \label{Sapproxeval}
Given an admissible $\Lambda \subset \vee$, $[{\bf \tilde{u}}^\top\!,\bm{\tilde{\theta}}^\top]^\top \in \ell_2(\Lambda)$ with $(\tilde{u},\vec{\tilde{\theta}}):=[{\bf \tilde{u}}^\top\!,\bm{\tilde{\theta}}^\top] \Psi$ sufficiently close to $(u,\vec{\theta})$, and an $\eps>0$, our approximate evaluation of $D{\bf Q}([{\bf \tilde{u}}^\top\!,\bm{\tilde{\theta}}^\top]^\top)$, given in \eqref{29}, is built in the following steps, where $k \in \N_0$ is a sufficiently large constant:
 \renewcommand{\theenumi}{s\arabic{enumi}}
\begin{enumerate}
\item \label{step1}  Find a tiling ${\mathcal T}(\eps) \subset \cO_\Omega$, such that 
\begin{align*}
&\inf_{(g_\eps,f_\eps,\vec{h}_\eps) \in \cP_m({\mathcal T}(\eps)\cap \Gamma_D) \cap C(\Gamma_D) \times  \cP_m({\mathcal T}(\eps)) \times \cP_m({\mathcal T}(\eps))^n}
\Big(\|g-g_\eps\|_{H^{\frac{1}{2}}(\Gamma_D)} +\\
& \hspace*{8em}\|v_1 \mapsto \int_\Omega (f-f_\eps) v_1\,dx+\int_{\Gamma_N}(h-\vec{h}_\eps\cdot {\bf n}) v_1 \,ds\|_{\VV_1'}\Big)\leq \eps.
\end{align*}
Set ${\mathcal T}(\Lambda,\eps):={\mathcal T}(\Lambda) \oplus {\mathcal T}(\eps)$.
\item \label{step3} 
\begin{enumerate}
\item Approximate
$$
{\bf r}^{(\frac{1}{2})}_1:=\langle \Psi^{\VV_1}, N(\tilde{u})-f- \divv \vec{\tilde{\theta}}\rangle_{L_2(\Omega)}+\langle \Psi^{\VV_1}, \vec{\tilde{\theta}}\cdot {\bf n}-h\rangle_{L_2(\Gamma_N)}
$$
by
$$
{\bf \tilde{r}}^{(\frac{1}{2})}_1:={\bf r}^{(\frac{1}{2})}_1|_{\Lambda^{\VV_1}({\mathcal T}(\Lambda,\eps),k)}.
$$
\item With $\tilde{r}_1^{(\frac{1}{2})}:=({\bf \tilde{r}}^{(\frac{1}{2})}_1)^\top \Psi^{\VV_1}$, approximate
$$
{\bf r}_1=\left[\begin{array}{@{}c@{}} {\bf r}_{11} \\ {\bf r}_{12}\end{array} \right]:=\left[\begin{array}{@{}c@{}}  
\langle DN(\tilde{u}) \Psi^{\UU},  \tilde{r}_1^{(\frac{1}{2})} \rangle_{L_2(\Omega)}
 \\
\langle \Psi^{\PP},\grad \tilde{r}_1^{(\frac{1}{2})} \rangle_{L_2(\Omega)^n}
\end{array} \right] 
$$
by 
${\bf \tilde{r}}_1=
\left[\begin{array}{@{}c@{}} {\bf \tilde{r}}_{11} \\ {\bf \tilde{r}}_{12}\end{array} \right]:=
{\bf r}_1|_{\Lambda({\mathcal T}(\Lambda^{\VV_1}({\mathcal T}(\Lambda,\eps),k)),k)}$.
\end{enumerate}
\item \label{step4}
Approximate
$$
{\bf r}_2:=\left\langle
\left[\begin{array}{@{}c@{}}  
-A \grad \Psi^{\UU}
 \\
 \Psi^{\PP}
\end{array} \right] ,\vec{\tilde{\theta}}-A \grad \tilde{u}
\right\rangle_{L_2(\Omega)^n}
$$
by 
${\bf \tilde{r}}_2:={\bf r}_2|_{\Lambda({\mathcal T}(\Lambda),k)}$.
\item \label{step2} 
\begin{enumerate}
\item
Approximate ${\bf r}^{(\frac{1}{2})}_3:=\langle \Psi^{\VV_2}, \tilde{u}-g\rangle_{L_2(\Gamma_D)}$ by 
$$
{\bf \tilde{r}}^{(\frac{1}{2})}_3:={\bf r}^{(\frac{1}{2})}_3|_{\Lambda^{\VV_2}({\mathcal T}(\Lambda,\eps)\cap \Gamma_D,k)}.
$$
\item With $\tilde{r}_3^{(\frac{1}{2})}:=({\bf \tilde{r}}^{(\frac{1}{2})}_3)^\top \Psi^{\VV_2}$, 
approximate
${\bf r}_3:=\left[\begin{array}{@{}c@{}} \langle \Psi^\UU,\tilde{r}_3^{(\frac{1}{2})}\rangle_{L_2(\Gamma_D)}\\ 0_{\vee_\PP} \end{array}\right]$ by
$$
{\bf \tilde{r}}_3:={\bf r}_3|_{\Lambda({\mathcal T}_{\Gamma_D}(\Lambda^{\VV_2}({\mathcal T}(\Lambda,\eps)\cap \Gamma_D,k)),k)}.
$$
\end{enumerate}
\end{enumerate}

Although \eqref{step3}--\eqref{step2} may look involved at first glance, the same kind of approximation is used at all instances. Each term in \eqref{29} consists essentially of a wavelet basis that is integrated against a piecewise polynomial, or more precisely, a function that can be sufficiently accurately approximated by a piecewise polynomial thanks to the control of the data oscillation by the refinement of the partition performed in \eqref{step1}.
In each of these terms, all wavelets are neglected whose levels exceed locally the level of the partition plus a constant $k$.

In the next theorem it is shown that this approximate residual evaluation satisfies the condition for optimality of the adaptive wavelet Galerkin method.

\begin{theorem} \label{thm2}
For an admissible  $\Lambda \subset \vee$, $[{\bf \tilde{u}}^\top\!,\bm{\tilde{\theta}}^\top]^\top \in \ell_2(\Lambda)$ with $(\tilde{u},\vec{\tilde{\theta}}):=[{\bf \tilde{u}}^\top\!,\bm{\tilde{\theta}}^\top] \Psi$ sufficiently close to $(u,\vec{\theta})$, and an $\eps>0$, consider the steps \eqref{step1}-\eqref{step2}.
With $s>0$ such that $[{\bf u}^\top\!,\bm{\theta}^\top]^\top  \in \cA^s$, let ${\mathcal T}(\eps)$ from \eqref{step1} satisfy $\# {\mathcal T}(\eps) \lesssim \eps^{-1/s}$.
Then
$$
\|D{\bf Q}([{\bf \tilde{u}}^\top\!,\bm{\tilde{\theta}}^\top]^\top)-({\bf \tilde{r}}_1+{\bf \tilde{r}}_2+{\bf \tilde{r}}_3)\| \lesssim 
2^{-k/2} (\|u-\tilde{u}\|_\UU+\|\vec{\theta}-\vec{\tilde{\theta}}\|_\PP)+\eps,
$$
where the computation of ${\bf \tilde{r}}_1+{\bf \tilde{r}}_2+{\bf \tilde{r}}_3$ requires ${\mathcal O}(\# \Lambda +\eps^{-1/s})$ operations.
So by taking $k$ sufficiently large, Condition~\ref{nonlineareval*} is satisfied.
\end{theorem}

Before giving the proof of this theorem, let us discuss matters related to step \eqref{step1}.
First of all,  {\em existence} of tilings ${\mathcal T}(\eps)$ as mentioned in the theorem is guaranteed. Indeed, because of $[{\bf u}^\top\!,\bm{\theta}^\top]^\top  \in \cA^s$, given $C>0$, for any $\eps>0$ there exists a $[{\bf \tilde{u}}_\eps^\top\!,\bm{\tilde{\theta}}_\eps^\top]^\top \in \ell_2(\Lambda_\eps)$, where $\Lambda_\eps$ is admissible and $\#\Lambda_\eps \lesssim \eps^{-1/s}$, such that $\|[{\bf u}^\top\!,\bm{\theta}^\top]^\top -[{\bf \tilde{u}}_\eps^\top\!,\bm{\tilde{\theta}}_\eps^\top]^\top\|\leq \eps/C$. By taking a suitable $C$, 
from Lemma~\ref{lem2}, \eqref{def_Q} and \eqref{8} we infer that with 
$(\tilde{u}_\eps,\vec{\tilde{\theta}}_\eps):=[{\bf \tilde{u}}_\eps^\top\!,\bm{\tilde{\theta}}_\eps^\top] \Psi$,
$$
\|\tilde{u}_\eps-g\|_{H^{\frac{1}{2}}(\Gamma_D)}+\|v_1\mapsto \int_\Omega (N(\tilde{u}_\eps)-\divv \vec{\tilde{\theta}}_\eps-f)v_1\,dx+\int_{\Gamma_N} (\vec{\tilde{\theta}}_\eps\cdot {\bf n}-h) v_1 \,ds\|_{\VV_1'} \leq \eps.
$$
Since $\tilde{u}_\eps \in \cP_m(\tria(\Lambda_\eps)) \cap C(\Omega)$, $N(\tilde{u}_\eps)-\divv \vec{\tilde{\theta}}_\eps \in \cP_m(\tria(\Lambda_\eps))$, $\vec{\tilde{\theta}}_\eps \in  \cP_m(\tria(\Lambda_\eps))^n$, and $\# \tria(\Lambda_\eps) \lesssim \# \Lambda_\eps$, the tiling ${\mathcal T}(\eps):= \tria(\Lambda_\eps)$ satisfies the assumptions.

Loosely speaking, this result can be rephrased by saying that if the solution of $D{\bf Q}([{\bf u}^\top\!,\bm{\theta}^\top]^\top)=0$ is in $\cA^s$, then so is the forcing function $(f,g,h)$. 
This is not automatically true, cf. \cite{45.47} for a discussion in the adaptive finite element context, but in the current setting it is a consequence of the fact that, thanks to assumption \eqref{8},
the first order partial differential operators apply to the wavelet bases $\Psi^\ast$ for $\ast \in \{\UU,\PP_1,\ldots,\PP_n,\VV_1,\VV_2\}$ in `mild' sense (the result of the application of each of these operators lands in $L_2$-space).

Knowing that a suitable $\tria(\eps)$ exists is different from knowing how to construct it.
For our convenience thinking of $g=h=0$, and so $\UU=\VV_1=H^1_0(\Omega)$, assuming that $f \in L_2(\Omega)$
one has $\inf_{f_\eps \in \cP_m(\tria)} \|f-f_\eps\|^2_{H^{-1}(\Omega)} \lesssim \mathrm{osc}(f,\tria)^2:=\sum_{\omega \in \tria} \diam(\omega)^2 \inf_{f_\omega \in \cP_m(\omega)} \|f-f_\omega\|^2_{L_2(\omega)}$.
Ignoring quadrature issues, for any partition $\tria$, $\mathrm{osc}(f,\tria)$ is computable. A quasi-minimal partition $\tria(\eps)$ such that $\mathrm{osc}(f,\tria(\eps)) \lesssim \eps$ can be computed using 
the Thresholding Second Algorithm from \cite{20.5}. Now the assumption to be added to Theorem~\ref{thm2} is that for such a partition, $\# \tria(\eps) \lesssim \eps^{-1/s}$.


Note that it is nowhere needed to explicitly approximate the forcing functions by approximating their wavelet expansions.


\begin{proof}[Proof of Theorem~\ref{thm2}] By construction we have
\begin{align*}
D{\bf Q}([{\bf \tilde{u}}^\top\!,\bm{\tilde{\theta}}^\top]^\top)-({\bf \tilde{r}}_1+{\bf \tilde{r}}_2+{\bf \tilde{r}}_3)
=\left[\begin{array}{@{}c@{}}  
\langle DN(\tilde{u}) \Psi^{\UU},  \Psi^{\VV_1} \rangle_{L_2(\Omega)}
 \\
\langle \Psi^{\PP},\grad \Psi^{\VV_1} \rangle_{L_2(\Omega)^n}
\end{array} \right] ({\bf r}^{(\frac{1}{2})}_1-{\bf \tilde{r}}^{(\frac{1}{2})}_1)
\\
+ {\bf r}_1-{\bf \tilde{r}}_1
+ {\bf r}_2-{\bf \tilde{r}}_2+
\left[\begin{array}{@{}c@{}} \langle \Psi^\UU,\Psi^{\VV_2}\rangle_{L_2(\Gamma_D)}\\ 0_{\vee_\PP} \end{array}\right]
({\bf r}^{(\frac{1}{2})}_3-{\bf \tilde{r}}_3^{(\frac{1}{2}})+{\bf r}_3-{\bf \tilde{r}}_3.
\end{align*}
From $\Psi^{\UU}$, $\Psi^{\PP}$, and $\Psi^{\VV_1}$ being Riesz bases, and $N:\UU \supset \dom(N) \rightarrow \VV_1'$ being continuously differentiable at $u$, one infers that
$
 \left[\begin{array}{@{}c@{}}  
\langle DN(\tilde{u}) \Psi^{\UU},  \Psi^{\VV_1} \rangle_{L_2(\Omega)}
 \\
\langle \Psi^{\PP},\grad \Psi^{\VV_1} \rangle_{L_2(\Omega)^n}
\end{array} \right]\in \cL(\ell_2(\vee_{\VV_1}),\ell_2(\vee))$,
with a norm that is bounded uniformly in $\tilde{u}$ from a neighbourhood of $u$.
Similarly, from $u \mapsto u|_{\Gamma_D} \in \cL(\UU,H^{\frac{1}{2}}(\Gamma_D))$, and $\UU$ and $\VV_2$ being Riesz bases for $\UU$ and $H^{\frac{1}{2}}(\Gamma_D)'$, respectively, we infer that $\left[\begin{array}{@{}c@{}} \langle \Psi^\UU,\Psi^{\VV_2}\rangle_{L_2(\Gamma_D)}\\ 0 \end{array}\right] \in \cL(\ell_2(\vee_{\VV_2}),\ell_2(\vee))$.
We conclude that
\begin{equation} \label{202}
\begin{split}
\|D{\bf Q}([&{\bf \tilde{u}}^\top\!,\bm{\tilde{\theta}}^\top]^\top)-({\bf \tilde{r}}_1+{\bf \tilde{r}}_2+{\bf \tilde{r}}_3)\| \lesssim \\
&
\|{\bf r}^{(\frac{1}{2})}_1-{\bf \tilde{r}}^{(\frac{1}{2})}_1\|+ \|{\bf r}_1-{\bf \tilde{r}}_1\| + \|{\bf r}_2-{\bf \tilde{r}}_2\|+\|{\bf r}^{(\frac{1}{2})}_3-{\bf \tilde{r}}^{(\frac{1}{2})}_3\|+ \|{\bf r}_3-{\bf \tilde{r}}_3\|.
\end{split}
\end{equation}
We bound all terms at the right-hand side.

With $f_\eps$, $\vec{h}_\eps$ from \eqref{step1}, using that $\Psi^{\VV_1}$ is a Riesz basis, we have that
\begin{align} \label{200}
&\|{\bf r}^{(\frac{1}{2})}_1-{\bf \tilde{r}}^{(\frac{1}{2})}_1\| \lesssim \eps+\\ \nonumber
\Big\|\Big(\langle \Psi^{\VV_1}, N(\tilde{u})-&f_\eps- \divv \vec{\tilde{\theta}}\rangle_{L_2(\Omega)}+\langle \Psi^{\VV_1}, (\vec{\tilde{\theta}}-\vec{h}_\eps)\cdot {\bf n}\rangle_{L_2(\Gamma_N)}\Big)\big|_{\vee_{\VV_1}\setminus \Lambda^{\VV_1}({\mathcal T}(\Lambda,\eps),k)} \Big\|.
\end{align}
From $N(\tilde{u})-f_\eps- \divv \vec{\tilde{\theta}} \in \cP_m({\mathcal T}(\Lambda,\eps))$, and $\vec{\tilde{\theta}}-\vec{h}_\eps \in \cP_m({\mathcal T}(\Lambda,\eps))^n$, Proposition~\ref{prop1} shows that the norm at the right-hand side of \eqref{200} is $\lesssim 2^{-k} \big\|v_1 \mapsto 
\int_\Omega \vec{\tilde{\theta}} \cdot \grad v_1+(N(\tilde{u})-f_\eps)v_1 \,dx   -\int_{\Gamma_N} \vec{h}_\eps\cdot {\bf n} v_1 \,ds
\big\|_{\VV_1'}$.
Again by using \eqref{step1}, we infer that
\begin{equation} \label{b1}
\|{\bf r}^{(\frac{1}{2})}_1-{\bf \tilde{r}}^{(\frac{1}{2})}_1\|\lesssim \eps+ 2^{-k} \big\|v_1 \mapsto 
\int_\Omega \vec{\tilde{\theta}} \cdot \grad v_1+(N(\tilde{u})-f)v_1 \,dx   -\int_{\Gamma_N} h v_1 \,ds
\big\|_{\VV_1'}.
\end{equation}

Thanks to
$\grad \tilde{r}_1^{(\frac{1}{2})} \in \cP_m({\mathcal T}(\Lambda^{\VV_1}({\mathcal T}(\Lambda,\eps),k)))^n$, an application of Proposition~\ref{prop2} (first estimate) shows that
\begin{align*}
\|{\bf r}_{12}-{\bf \tilde{r}}_{12}\| &\lesssim 2^{-k/2}\|\tilde{r}_1^{(\frac{1}{2})}\|_{\VV_1} \eqsim 
 2^{-k/2} \|{\bf \tilde{r}}^{(\frac{1}{2})}_{1}\|.
\end{align*}
Our assumptions on $N$ show that $DN(\tilde{u})w$ is of the form $p_1(\tilde{u})w+\vec{p}_2(\tilde{u}) \cdot \nabla w$ for some piecewise polynomials $p_1$ and $p_2$ in $\tilde{u}$ and $x$ w.r.t. $\{\omega\colon \omega \in \cO_\Omega,|\omega|=0\}$, where moreover $w \mapsto p_1(\tilde{u})w \in \cL(\UU,\VV_1')$, $\vec{w} \mapsto \vec{p}_2(\tilde{u}) \cdot\vec{w} \in \cL(L_2(\Omega)^n,\VV_1')$, uniformly in $\tilde{u}$ in a neighborhood of $u \in \UU$.
Consequently, applications of  Propositions~\ref{prop3}-\ref{prop2} (second estimate) show that
\begin{align*}
\|{\bf r}_{11}-{\bf \tilde{r}}_{11}\| &\lesssim 2^{-k} \|p_1(\tilde{u}) \tilde{r}_{1}^{(\frac{1}{2})}\|_{\UU'}+2^{-k/2}\|\vec{p}_2(\tilde{u}) \tilde{r}_{1}^{(\frac{1}{2})}\|_{L_2(\Omega)^n} \\
&\lesssim 
2^{-k/2} \|\tilde{r}_{1}^{(\frac{1}{2})}\|_{\VV_1} \eqsim 2^{-k/2} \|{\bf \tilde{r}}^{(\frac{1}{2})}_1\|.
\end{align*}
Now use that $\|{\bf \tilde{r}}^{(\frac{1}{2})}_1\| \leq \|{\bf r}^{(\frac{1}{2})}_1-{\bf \tilde{r}}^{(\frac{1}{2})}_1\|+\|{\bf r}^{(\frac{1}{2})}_1\|$, and $\|{\bf r}^{(\frac{1}{2})}_1\| \lesssim 
\big\|v_1 \mapsto 
\int_\Omega \vec{\tilde{\theta}} \cdot \grad v_1+(N(\tilde{u})-f)v_1 \,dx   -\int_{\Gamma_N} h v_1 \,ds
\big\|_{\VV_1'}$, to conclude that
\begin{equation} \label{b2}
\|{\bf r}_1-{\bf \tilde{r}}_1\|\lesssim2^{-k/2}\Big(\eps+ \big\|v_1 \mapsto 
\int_\Omega \vec{\tilde{\theta}} \cdot \grad v_1+(N(\tilde{u})-f)v_1 \,dx   -\int_{\Gamma_N} h v_1 \,ds
\big\|_{\VV_1'}\Big).
\end{equation}

Thanks to $\vec{\tilde{\theta}}-A \grad \tilde{u}\in \cP_m({\mathcal T}(\Lambda))^n$ by \eqref{16}, and $A^\top$ being piecewise polynomial, an application of Proposition~\ref{prop2} shows that
\begin{equation} \label{b3}
\|{\bf r}_2-{\bf \tilde{r}}_2\|  \lesssim 2^{-k/2} \|\vec{\tilde{\theta}}-A \grad \tilde{u}\|_{L_2(\Omega)^n}.
\end{equation}

With $g_\eps$ from \eqref{step1}, using that $\Psi^{\VV_1}$ is a Riesz basis for $H^{-\frac{1}{2}}(\Gamma_D)$, we have that
\begin{equation} \label{201}
\|{\bf r}^{(\frac{1}{2})}_3-{\bf \tilde{r}}^{(\frac{1}{2})}_3\| \lesssim \eps+
\| \langle \Psi^{\VV_2},\tilde{u} -g_\eps\rangle_{L_2(\Gamma_D)}|_{\vee_{\VV_2} \setminus \Lambda^{\VV_2}({\mathcal T}(\Lambda,\eps)\cap \Gamma_D,k)}\|.
\end{equation}
From $\tilde{u}|_{\Gamma_D}-g_\eps \in \cP_m({\mathcal T}(\Lambda,\eps)\cap \Gamma_D) \cap C(\Gamma_D)$, Proposition~\ref{prop5} shows that the norm at the right-hand side of \eqref{201} is $\lesssim 2^{-k/2} \|\tilde{u}-g_\eps\|_{H^{\frac{1}{2}}(\Gamma_D)}$.
Again by using \eqref{step1}, we infer that
\begin{equation} \label{b4}
\|{\bf r}^{(\frac{1}{2})}_3-{\bf \tilde{r}}^{(\frac{1}{2})}_3\|\lesssim \eps+2^{-k/2} \|\tilde{u}-g\|_{H^{\frac{1}{2}}(\Gamma_D)}.
\end{equation}

Thanks to $\tilde{r}_3^{(\frac{1}{2})} \in \cP_m(
{\mathcal T}_{\Gamma_D}(\Lambda^{\VV_2}({\mathcal T}(\Lambda,\eps)\cap \Gamma_D,k)))$, Proposition~\ref{prop6} shows that
$\|{\bf r}_3-{\bf \tilde{r}}_3\| \lesssim 2^{-k/2}\|\tilde{r}_3^{(\frac{1}{2})}\|_{H^{-\frac{1}{2}}(\Gamma_D)} \eqsim 
2^{-k/2}\|{\bf \tilde{r}}^{(\frac{1}{2})}_3\|$.
Now use that $\|{\bf \tilde{r}}^{(\frac{1}{2})}_3\| \leq \|{\bf r}^{(\frac{1}{2})}_3-{\bf \tilde{r}}^{(\frac{1}{2})}_3\|+\|{\bf r}^{(\frac{1}{2})}_3\|$, and $\|{\bf r}^{(\frac{1}{2})}_3\| \lesssim \|\tilde{u}-g\|_{H^{\frac{1}{2}}(\Gamma_D)}$ to conclude that 
\begin{equation} \label{b5}
\|{\bf r}_3-{\bf \tilde{r}}_3\|\lesssim 2^{-k/2}\big(\eps+ \|\tilde{u}-g\|_{H^{\frac{1}{2}}(\Gamma_D)}\big).
\end{equation}

By collecting the  upper bounds \eqref{b1}--\eqref{b5} derived for all five terms at the right-hand side of \eqref{202},
and by using Lemma~\ref{lem2} in combination with the least squares functional given in \eqref{def_Q}, the proof of the first statement is completed.

To bound the cost of the computations, we consider the computation of ${\bf \tilde{r}}^{(\frac{1}{2})}_1$.
First, find a representation of $N(\tilde{u})-\divv \vec{\tilde{\theta}}$ as an element of  $\cP_m({\mathcal T}(\Lambda,\eps))$ by applying multi- to single-scale transforms. 
For each tile $\omega \in {\mathcal T}(\Lambda^{\VV_1}({\mathcal T}(\Lambda,\eps),k))$, and for $\phi$ running over some basis of $\cP_m(\omega)$, compute
$\langle \phi, N(\tilde{u})-f-\divv \vec{\tilde{\theta}}\rangle_{L_2(\omega)}$.
From this, compute $[\langle \psi_\lambda^{\VV_1}, N(\tilde{u})-f-\divv \vec{\tilde{\theta}}\rangle_{L_2(\Omega)}]_{\lambda \in \Lambda^{\VV_1}({\mathcal T}(\Lambda,\eps),k)}$ by applying a transpose of a multi- to single-scale transform. Similar steps yield 
$[\langle \psi_\lambda^{\VV_1},  \vec{\tilde{\theta}}\cdot {\bf n}-h\rangle_{L_2(\Gamma_N)}]_{\lambda \in \Lambda^{\VV_1}({\mathcal T}(\Lambda,\eps),k)}$.
The total cost involved in computing ${\bf \tilde{r}}^{(\frac{1}{2})}_1$ is
bounded by a multiple of $\# {\mathcal T}(\Lambda,\eps) \lesssim \#\Lambda +\eps^{-1/s}$ operations.

Since fully analogous considerations apply to bounding the cost of the computations of ${\bf \tilde{r}}_1$, ${\bf \tilde{r}}_2$, ${\bf \tilde{r}}^{(\frac{1}{2})}_3$, and ${\bf \tilde{r}}_3$, the proof is completed.
\end{proof}


\section{Numerical results} \label{Snumerics}
For $\Omega \subset \R^2$ being the L-shaped domain $(0,1)^2\setminus [\frac{1}{2},1)^2$, we consider the semi-linear boundary value problem
\begin{equation} \label{bvp_simple}
 \left\{
\begin{array}{r@{}c@{}ll}
-\Delta u+ u^3 &\,\,=\,\,& f &\text{ on } \Omega,\\
u &\,\,=\,\,& 0 &\text{ on } \partial\Omega,
\end{array}
\right.
\end{equation}
where, for simplicity, $f=1$ (to test our code we also tried some right hand sides corresponding to some fabricated polynomial solutions $u$). With $\UU=H^1_0(\Omega)=\VV$, $\PP=L_2(\Omega)^2$, we applied the {\bf awgm} (Algorithm~\ref{coordinates}), with ${\bf F}$ reading as  $D{\bf Q}$, for the adaptive solution of $[{\bf u}^\top\!,\!\bm{{\theta}}^\top]^\top$ from
\[
\begin{split}
 \!\!D{\bf Q}([{\bf u}^\top\!,\!\bm{{\theta}}^\top]^\top)\!=&\!
\left[\begin{array}{@{}l@{}}
\langle \partial_1 \Psi^{\UU}, \partial_1  {u}-{\theta}_1 \rangle_{L_2(\Omega)}+\langle \partial_2 \Psi^{\UU}, \partial_2  {u}-{\theta}_2 \rangle_{L_2(\Omega)}\\
\langle \Psi^{\PP_1},  {\theta}_1- \partial_1 {u} \rangle_{L_2(\Omega)}\\
\langle \Psi^{\PP_2}, {\theta}_2 -\partial_2  {u} \rangle_{L_2(\Omega)}\\
\end{array}
\right]\\
&+\!\!\left[\!\begin{array}{@{}c@{}}  
\langle  \Psi^{\UU},  3{{u}}^2 \Psi^{\VV} \rangle_{L_2(\Omega)}
 \\
\langle \Psi^{\PP_1},\partial_1 \Psi^{\VV} \rangle_{L_2(\Omega)}\\
\langle \Psi^{\PP_2},\partial_2 \Psi^{\VV} \rangle_{L_2(\Omega)}
\end{array}\! \right] \!\!
\big\langle \Psi^{\VV}, {{u}}^3\!-\!f\!-\!\divv \vec{{\theta}} \big\rangle_{L_2(\Omega)}={\bf 0},
\end{split}
\]
where $u:={\bf u}^\top \Psi^\UU$, $\theta_i:=\bm{\theta}_i^\top \Psi^{\PP_i}$.

Here we equipped $\PP_i$ ($i=1,2$) with the continuous piecewise linear three-point wavelet basis from \cite{249.70}, the space $\VV$ with the same basis (obviously scaled differently, and with homogeneous boundary conditions incorporated), and $\UU$ with a newly developed continuous piecewise quadratic wavelet basis. 
These bases can be applied on any polygon, and they
 satisfy all assumptions \eqref{w1}-\eqref{w8}. In particular all wavelets except possibly those `near' the Dirichlet boundary have one vanishing moment.
For each basis, to each wavelet that is not on the coarsest level we associate one parent on the next coarsest level according to Definition~\ref{def1}.
For any $\ell \in \N_0$ the subsets of the bases consisting of all wavelets up to some level span exactly the space of continuous piecewise linears, continuous piecewise linears zero at $\partial\Omega$, or continuous piecewise quadratics zero at $\partial\Omega$, respectively, w.r.t. the subdivision of $\Omega$ as indicated in Figure~\ref{fig2}.
\begin{figure}[h]
\begin{center}
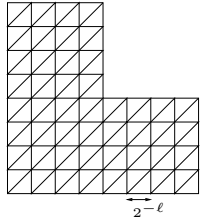
\end{center}
\caption{Meshes w.r.t. which the wavelets are piecewise polynomial.}
\label{fig2}
\end{figure}

On a bounded domain, the three-point basis has actually not be proven to be stable in $L_2(\Omega)$. Although alternative bases are available whose Riesz basis property has been proven, we opted for the three-point basis, because of its efficient implementation and because numerical results indicated that it is stable. 
In Figure~\ref{fig3}, numerically computed condition numbers are given of sets of all wavelets up to some level.
\begin{figure}[h]
\begin{center}
\hspace*{-1.5em}
\includegraphics[scale=0.22]{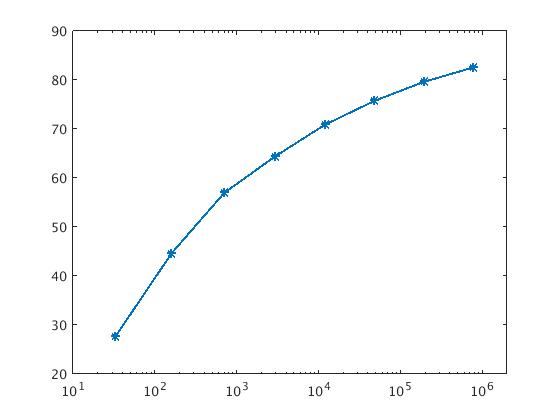} \hspace*{-1.5em}
\includegraphics[scale=0.22]{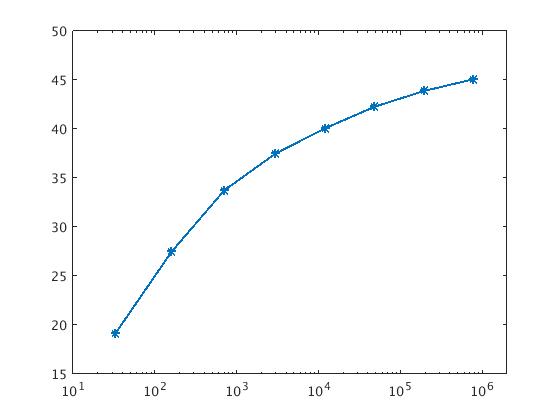}\hspace*{-1em}
\includegraphics[scale=0.22]{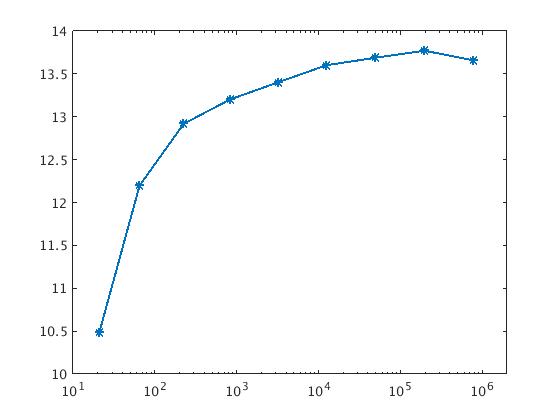}
\hspace*{-1.5em}
\end{center}
\caption{Condition numbers of $\langle \nabla \Psi^\UU_N,\nabla \Psi^\UU_N\rangle_{L_2(\Omega)^2}$, $\langle \nabla \Psi^\VV_N,\nabla \Psi^\VV_N\rangle_{L_2(\Omega)^2}$, and $\langle \Psi^{\PP_i}_N, \Psi^{\PP_i}_N\rangle_{L_2(\Omega)}$, where $\Psi^*_N$ is the subset 
of all wavelets from $\Psi^*$ up to some level, where $N$ denotes its cardinality.}
\label{fig3}
\end{figure}

The continuous piecewise quadratic wavelets are biorthogonal ones with the `dual multiresolution analysis' being the sequence of continuous piecewise linears, zero at the $\partial\Omega$,
w.r.t. one additional level of refinement. Details of this basis construction will be reported elsewhere.

We performed the approximate evaluation of $D{\bf Q}(\cdot)$ according to \eqref{step1}-\eqref{step2} and Thm.~\ref{thm2} in Sect.~\ref{Sapproxeval}
with some simplifications because of the current homogeneous Dirichlet boundary conditions and sufficiently smooth right-hand side
(\eqref{step1} and \eqref{step2} are void, and in \eqref{step3} the boundary term is void).
Taking the parameter $k=1$, it turns out that the approximate evaluation is sufficiently accurate to be used in Step (R) of {\bf awgm} (so we do not perform a loop),
as well in the simple fixed point iteration \eqref{fixedpoint} with damping $\omega=\frac{1}{4}$ that we use for Step (G).
We took the parameter $\gamma$ in Step (G) equal to $0.15$ (more precisely, for stopping the iteration we checked whether the norm of the approximate residual, restricted to $\Lambda_{i+1}$, is less or equal to $0.15 \|{\bf r}_i\|$).

For the bulk chasing, i.e. Step (B), we simply collected the indices of the largest entries of the approximate residual ${\bf r}_i$ until the norm of the residual restricted to those indices is not less than $0.4 \|{\bf r}_i\|$ (i.e. $\mu_1=0.4$), and then, after adding the indices from the current $\bm{\Lambda}_i$
to this set,  we expand it to an admissible set (cf. Definition~\ref{def_admissible}).
Although this simple procedure is neither guaranteed to satisfy Condition~\ref{bulk} nor
(B) for some constant $0<\mu_0\leq \mu_1$, we observed that it works satisfactory in practice.

In view of the orders $3$ and $2$ of the bases for $\UU$ and $\PP$, and the fact the PDE is posed in $n=2$ space dimensions, the best possible convergence rate that can be expected is $\min(\frac{3-1}{2},\frac{2-0}{2})=1$. 
In Figure~\ref{fig4}, we show the norm of the approximate residual vs. the total number of wavelets underlying the approximation for $(u,\vec{\theta})$.
\begin{figure}[h]
\begin{center}
\includegraphics[scale=0.4]{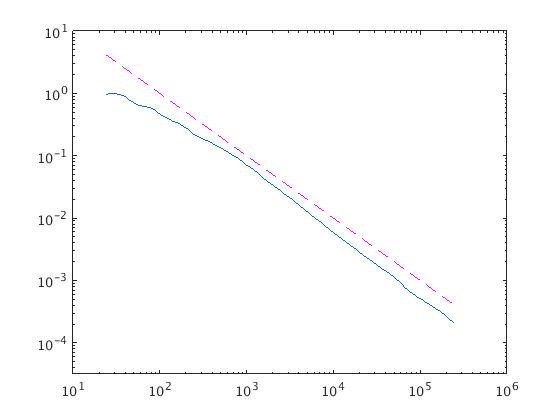}
\end{center}
\caption{Norm of the (approximate) residual, normalized by the norm of the initial residual,  generated by the {\bf awgm} vs. the total number of wavelets. The dotted line indicates the best possible slope $-1$.}
\label{fig4}
\end{figure}
The norm of the approximate residual is proportional to the $\UU \times \PP$-norm of the error in the approximation for $(u,\vec{\theta})$. We conclude that it decays with the best possible rate. Moreover, we observed that the computing times scale linearly with the number of unknowns.
Throughout the iteration, the number of wavelets for the approximation for $u$ is of the same order as the number of wavelets for the approximation for $\vec{\theta}$.
The maximum level that is reached at the end of the computations is $26$ for $u$ and $28$ for $\vec{\theta}$).
An approximate solution is illustrated in Figure~\ref{fig5}.
\begin{figure}[h]
\begin{center}
\includegraphics[scale=0.1]{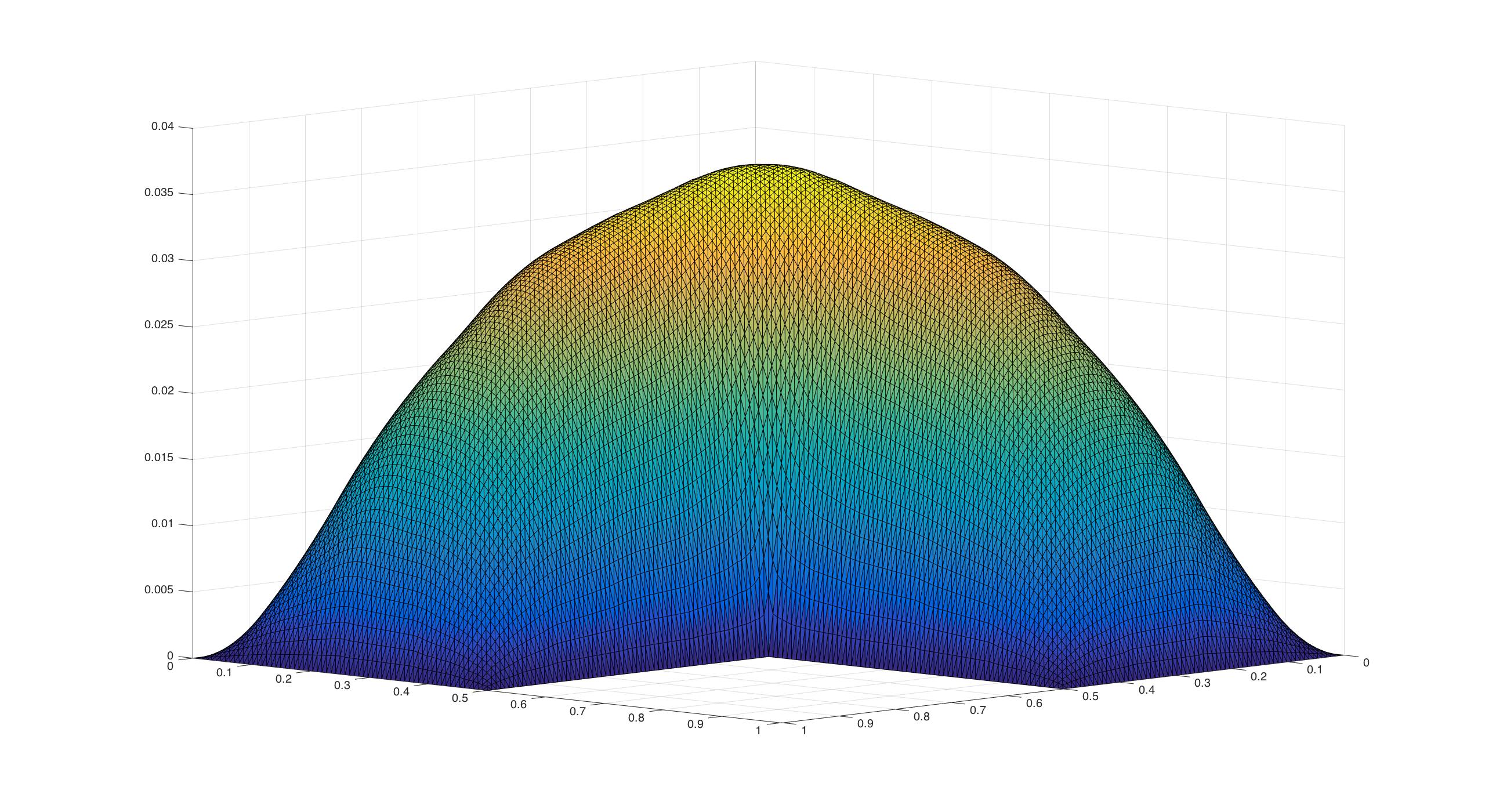}
\end{center}
\caption{The approximation for $u$ from the span of $202$ wavelets.}
\label{fig5}
\end{figure}
Centers of the supports of the wavelets that were selected for the approximation for $u$ are illustrated in Figure~\ref{fig6}.
\begin{figure}[h]
\begin{center}
\includegraphics[scale=0.25]{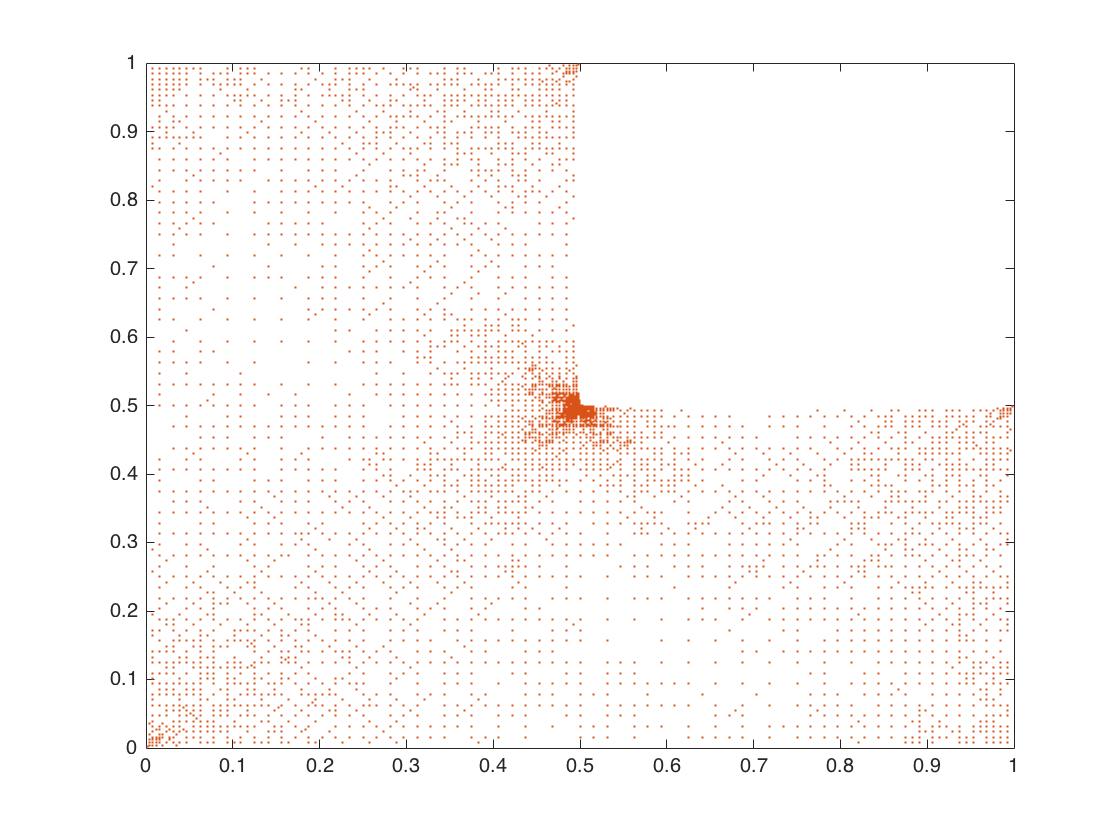}
\end{center}
\caption{Centers of the supports of the first 10366 wavelets for the approximation for $u$ that were selected by the {\bf awgm}.}
\label{fig6}
\end{figure}

Finally, in order to get an impression of the condition number of the bi-infinite linearized normal equations that eventually we are solving, we consider
the Poisson equation, i.e. \eqref{bvp_simple} without the $u^3$ term.
We are interested in the spectral condition number of the `system matrix' given by
\begin{align} \label{100}
\left[\begin{array}{ccc}
\langle \nabla \Psi^\UU, \nabla \Psi^\UU\rangle_{L_2(\Omega)^2} & -\langle \partial_1 \Psi^\UU,  \Psi^{\PP_1} \rangle_{L_2(\Omega)} &-\langle \partial_2 \Psi^\UU,  \Psi^{\PP_2} \rangle_{L_2(\Omega)} \\
-\langle  \Psi^{\PP_1}, \partial_1 \Psi^\UU  \rangle_{L_2(\Omega)} & \langle  \Psi^{\PP_1},  \Psi^{\PP_1}  \rangle_{L_2(\Omega)}  & 0\\
-\langle  \Psi^{\PP_2}, \partial_2 \Psi^\UU  \rangle_{L_2(\Omega)} & 0 &\langle  \Psi^{\PP_1},  \Psi^{\PP_1}  \rangle_{L_2(\Omega)}\\
\end{array}
\right]+ &\\ \nonumber
\left[\begin{array}{ccc} 0 & 0 & 0 \\ 0 &  
\langle  \Psi^{\PP_1}, \partial_1 \Psi^\VV  \rangle_{L_2(\Omega)} \langle \partial_1  \Psi^{\VV}, \Psi^{\PP_1}  \rangle_{L_2(\Omega)} &
\langle  \Psi^{\PP_1}, \partial_1 \Psi^\VV  \rangle_{L_2(\Omega)} \langle \partial_2  \Psi^{\VV}, \Psi^{\PP_2}  \rangle_{L_2(\Omega)}\\
 0 & 
\langle  \Psi^{\PP_2}, \partial_2 \Psi^\VV  \rangle_{L_2(\Omega)} \langle \partial_1  \Psi^{\VV}, \Psi^{\PP_1}  \rangle_{L_2(\Omega)} &
\langle  \Psi^{\PP_2}, \partial_2 \Psi^\VV  \rangle_{L_2(\Omega)} \langle \partial_2  \Psi^{\VV}, \Psi^{\PP_2}  \rangle_{L_2(\Omega)}\\
\end{array}
\right]&
\end{align}
To that end we numerically approximated the condition numbers of the finite square blocks of rows and columns with indices in $\Lambda$, with $\Lambda$ running over the wavelet index sets that were created by the {\bf awgm}. Even such a finite  block cannot be evaluated exactly, because it still involves the infinite collection $\Psi^\VV$. Given a $\Lambda$, we restricted this collection to the wavelets with indices in $\Lambda^\VV(\tria(\Lambda),k)$ as defined by Proposition~\ref{prop4} and Definition~\ref{partition-to-lambda}, where, as always, we take $k=1$.
The resulting matrix is exactly the one that we approximately invert in Step (G) by the fixed point iteration.
The computed condition numbers are given in Figure~\ref{fig7}.
\begin{figure}[h]
\begin{center}
\includegraphics[scale=0.3]{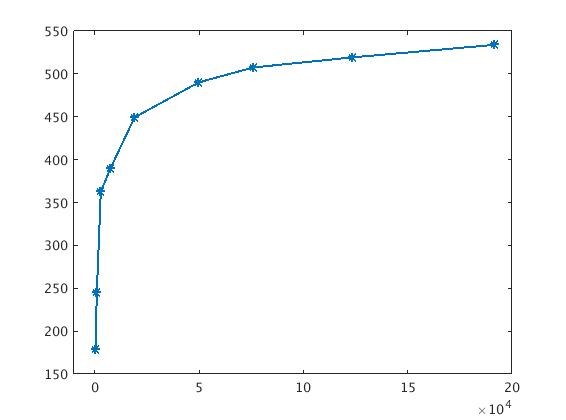}
\end{center}
\caption{Condition numbers of the (approximate) Galerkin system matrices vs. $\# \Lambda$.}
\label{fig7}
\end{figure}
We performed the same computation also for $k=2$, so with an enlarged set of wavelet indices from the basis $\Psi^\VV$, and found nearly indistinguishable results. 
We may conclude that for $\Lambda \rightarrow \infty$, the given numbers give accurate approximations for the 
condition number of the matrix in \eqref{100}.

\section{Stationary Navier-Stokes equations} \label{SNSE}
For $n \in \{2,3,4\}$, let $\Omega \subset \R^n$ be a bounded Lipschitz domain. The stationary Navier-Stokes equations in velocity--pressure formulation and with no-slip boundary conditions are given by
$$
\left\{
\begin{array}{rcll}
-\nu \triangle \vec{u}+(\vec{u}\cdot \grad)\vec{u}+\grad p & \!\!=\! \!& \vec{f} & \text{on } \Omega\\
\divv \vec{u} &\!\! = \!\!& g & \text{on } \Omega\\
\vec{u} & \!\!=\! \!& 0 & \text{on } \partial \Omega.
\end{array}
\right.
$$
In order to obtain, in any case in the linear Stokes case, results that hold uniformly in $\nu>0$, one may equip the spaces for velocities and pressure with $\nu$-dependent norms. The equivalent, but notationally more convenient approach that we will follow is to keep the standard norms, but to make the substitutions $\vec{\breve{u}}=\sqrt{\nu}\, \vec{u}$, $\breve{p}=\frac{1}{\sqrt{\nu}} p$, $\vec{\breve{f}}=\frac{1}{\sqrt{\nu}} \vec{f}$, and $\breve{g}=\sqrt{\nu} \, g$.
For convenience dropping the \raisebox{-0.3cm}{\huge \,$\breve{}\,$}-accents, the equations for the new unknowns read as
$$
\left\{
\begin{array}{rcll}
- \triangle \vec{u}+\nu^{-3/2} (\vec{u}\cdot \grad)\vec{u}+\grad p & \!\!=\! \!& \vec{f} & \text{on } \Omega\\
\divv \vec{u} &\!\! = \!\!& g & \text{on } \Omega\\
\vec{u} & \!\!=\! \!& 0 & \text{on } \partial \Omega.
\end{array}
\right.
$$

 In variational form they read as finding $(\vec{u},p) \in \framebox{$\UU :=H^1_0(\Omega)^n \times L_2(\Omega)/\R$}$ such that for some $(\vec{f},g) \in \UU'$,
$$
G(\vec{u},p)(\vec{v},q):=\int_\Omega \grad \vec{u} : \grad \vec{v}-p \divv\vec{v}+\nu^{-3/2}(\vec{u}\cdot \grad)\vec{u} \cdot \vec{v}+q (\divv  \vec{u} -g) -\vec{f}\cdot \vec{v}\,dx=0
$$
($(\vec{v},q) \in$ \framebox{$\VV:=\UU$}).

It is known that $G:\UU \rightarrow \UU'$, and that a solution $(\vec{u},p)$ exists (see e.g. \cite[Ch. IV]{75.4}).
Furthermore, $G$ is two times differentiable with its second derivative being constant.
We will {\em assume} that $DG(\vec{u},p) \in \cL(\UU,\UU')$ is a homeomorphism with its range, so that each of the conditions \eqref{r1}--\eqref{r3} from Sect.~\ref{S1} are satisfied.
The latter is known to hold true, with its range being equal to $\UU'$, when $\vec{f}$ is sufficiently small, in which case the solution $(\vec{u},p)$ is also unique (e.g. see \cite[Ch. IV]{75.4}).
For the linear case, so without the term $\nu^{-3/2} (\vec{u}\cdot \grad)\vec{u}$, thanks to our re-scaling, $DG(\vec{u},p)=G \in \Lis(\UU,\UU')$, and is independent of $\nu$.

Using the framework outlined in Sect.~\ref{S1}, we write this second order elliptic PDE as a first order system least squares problem.
There are different possibilities to do so.

\subsection{Velocity--pressure--velocity gradient formulation}
With \framebox{$\PP:=L_2(\Omega)^{n^2}$}, we define
$$
G_1 \in \cL(\PP,\UU'),\qquad G_2 \in \cL(\UU,\PP),
$$
by
$$
G_2 (\vec{u},p)=\grad \vec{u},\qquad (G_1 \underline{\theta})(\vec{v},q)=\int_\Omega \underline{\theta} : \grad \vec{v} \,d x
$$
The results from Sect.~\ref{S1} show that the solution $(\vec{u},p)$ can be found as the first components of the minimizer $(\vec{u},p, \underline{\theta}) \in  \UU \times \PP$ of
\begin{equation}  \label{def_Q2}
\begin{split}
Q(\vec{u},p, \underline{\theta}):={\textstyle \frac{1}{2}}\Big(&\big\|\vec{v} \mapsto 
\int_\Omega \underline{\theta} : \grad \vec{v}-p \divv\vec{v}+\nu^{-3/2} (\vec{u}\cdot \grad)\vec{u} \cdot \vec{v}-\vec{f}\cdot \vec{v}\|_{H^{-1}(\Omega)^n}^2+\\
&\|\divv \vec{u}-g\|_{L_2(\Omega)}^2+\|\underline{\theta}- \grad \vec{u}\|_{L_2(\Omega)^{n^2}}^2\Big),
\end{split}
\end{equation}
and so as the solution of the normal equations $DQ(\vec{u},p,\underline{\theta})=0$. 
Here we have used that on $H^1_0(\Omega)^n$, $\|\divv \cdot\|_{(L_2(\Omega)/\R)'}=\|\divv \cdot\|_{L_2(\Omega)}$.
Following \cite{23.5}, we call $\underline{\theta}=\nabla \vec{u}$ the velocity gradient.
As follows from Sect.~\ref{S1}, these normal equations are well-posed in the sense that they satisfy \eqref{38}--\eqref{41}.
This gives us an alternative, effortless proof for \cite[Thm.~3.1]{35.9302}.


To deal with the `unpractical'  norm on $H^{-1}(\Omega)^n$,  we equip $H^1_0(\Omega)^n$ 
 with some wavelet Riesz basis
$$
\Psi^{(\hat{H}^1_0)^n}=\{\psi_\lambda^{(\hat{H}^1_0)^n}\colon\lambda \in \vee_{(\hat{H}^1_0)^n}\},
$$
and replace, in the definition of $Q$,  the norm on its dual by the equivalent norm defined by
$\|\vec{h}(\Psi^{(\hat{H}^1_0)^n})\|$  for $\vec{h} \in H^{-1}(\Omega)^n$.

Next, after equipping $* \in \{H^1_0(\Omega)^n,L_2(\Omega)/\R,L_2(\Omega)^{n^2}\}$ with a Riesz basis
$\Psi^{*}=\{\psi_\lambda^{*}\colon\lambda \in \vee_{*}\}$, and so $H^1_0(\Omega)^n\times L_2(\Omega)/\R\times L_2(\Omega)^{n^2}$ with 
$$
\Psi:=(\Psi^{(H^1_0)^n},0_{L_2/\R},0_{L_2^{n^2}}) \cup (0_{(H^1_0)^n},\Psi^{L_2/\R},0_{L_2^{n^2}})\cup(0_{(H^1_0)^n},0_{L_2/\R},\Psi^{L_2^{n^2}}),
$$
with index set $\vee:=\vee_{(H^1_0)^n}\cup\vee_{L_2/\R}\cup \vee_{L_2^{n^2}}$, we apply the {\bf awgm} to the resulting system
\[
\begin{split}
& D{\bf Q}([{\bf {u}}^\top\!,{\bf {p}}^\top\!,\bm{{\theta}}^\top]^\top)=\left[\begin{array}{@{}c@{}} 
\langle \divv  \Psi^{(H^1_0)^n}, \divv \vec{{u}} -g  \rangle_{L_2(\Omega)}\\
0_{\vee_{L_2/\R}}\\
0_{\vee_{L_2^{n^2}}}
\end{array} \right] 
+
\left[\begin{array}{@{}c@{}}
\langle  \grad \Psi^{(H^1_0)^n},  \grad \vec{{u}}-\underline{{\theta}}\rangle_{L_2(\Omega)^{n^2}}\\
0_{\vee_{L_2/\R}}\\
\langle  \Psi^{L_2^{n^2}}, \underline{{\theta}} - \grad \vec{{u}}\rangle_{L_2(\Omega)^{n^2}}
\end{array} \right] +
\\
&\left[\begin{array}{@{}c@{}}  
\langle \frac{(\vec{{u}}\cdot\grad) \Psi^{(H^1_0)^n}+(\Psi^{(H^1_0)^n}\cdot\grad) \vec{{u}}}{\nu^{3/2}},\Psi^{(\hat{H}^1_0)^n}\rangle_{L_2(\Omega)^n}\\
-\langle \Psi^{L_2/\R},\divv \Psi^{(\hat{H}^1_0)^n} \rangle_{L_2(\Omega)} \\
\langle \Psi^{L_2^{n^2}}, \nabla \Psi^{(\hat{H}^1_0)^n} \rangle_{L_2(\Omega)^{n^2}} 
\end{array} \right] 
\Big\{\langle \Psi^{(\hat{H}^1_0)^n},{\textstyle \frac{(\vec{{u}}\cdot\grad) \vec{{u}}}{\nu^{3/2}}}-\vec{f}\rangle_{L_2(\Omega)^n}
\\& \hspace*{13em}
+\langle \nabla \Psi^{(\hat{H}^1_0)^n}, \underline{\theta}\rangle_{L_2(\Omega)^{n^2}}-
\langle \divv  \Psi^{(\hat{H}^1_0)^n}, p\rangle_{L_2(\Omega)}
\Big\}
=0.
\end{split}
\]

To express the three terms in $\vec{v} \mapsto \langle \vec{v},\nu^{-3/2}(\vec{{u}}\cdot\grad) \vec{{u}}-\vec{f}\rangle_{L_2(\Omega)^n}+
\langle \nabla \vec{v}, \underline{\theta}\rangle_{L_2(\Omega)^{n^2}}-
\langle \divv \vec{v}, p\rangle_{L_2(\Omega)}
\in H^{-1}(\Omega)^n$ w.r.t. one dictionary, similarly to Sect.~\ref{Sseemingly} we impose the additional, but in applications easily realizable conditions that  
\begin{equation} \label{88}
\Psi^{L_2/\R} \subset H^1(\Omega),\, \Psi^{L_2^{n^2}} \subset H(\divv;\Omega)^n.
\end{equation}
Then for finitely supported approximations $[{\bf \tilde{u}}^\top\!,{\bf \tilde{p}}^\top\!,\bm{\tilde{\theta}}^\top]^\top$ to $[{\bf u}^\top\!,{\bf p}^\top\!,\bm{\theta}^\top]^\top$, for $(\vec{\tilde{u}},\tilde{p},\underline{\tilde{\theta}}):=[{\bf \tilde{u}}^\top\!,{\bf \tilde{p}}^\top\!,\bm{\tilde{\theta}}^\top] \Psi \in H^1_0(\Omega)^n \times H^1(\Omega) \times H(\divv;\Omega)^n$, we have
\begin{equation} \label{290}
\framebox{$\begin{split}
& \!\!D{\bf Q}([{\bf \tilde{u}}^\top\!,{\bf \tilde{p}}^\top\!,\bm{\tilde{\theta}}^\top]^\top)\!=\!
\left[\begin{array}{@{}c@{}} 
\langle \divv  \Psi^{(H^1_0)^n}\!,\! \divv \vec{\tilde{u}}-g  \rangle_{L_2(\Omega)}\!\\
0_{\vee_{L_2/\R}}\\
0_{\vee_{L_2^{n^2}}}
\end{array} \!\right] 
\!+\!
\left[\!\begin{array}{@{}c@{}}
\langle  \grad \Psi^{(H^1_0)^n},  \grad \vec{\tilde{u}}-\underline{\tilde{\theta}} \rangle_{L_2(\Omega)^{n^2}}\\
0_{\vee_{L_2/\R}}\\
\langle  \Psi^{L_2^{n^2}}, \underline{\tilde{\theta}} - \grad \vec{\tilde{u}}\rangle_{L_2(\Omega)^{n^2}}
\end{array}\! \right] \!+\!\!\\
&\!\!\!\!\left[\begin{array}{@{}c@{}}  
\!\langle \frac{(\vec{\tilde{u}}\cdot\grad) \Psi^{(H^1_0)^n}+(\Psi^{(H^1_0)^n}\cdot\grad) \vec{\tilde{u}}}{\nu^{3/2}},\Psi^{(\hat{H}^1_0)^n}\!\rangle_{L_2(\Omega)^n}\\
-\langle \Psi^{L_2/\R},\divv \Psi^{(\hat{H}^1_0)^n} \rangle_{L_2(\Omega)} \\
-\langle \divv \Psi^{L_2^{n^2}}, \Psi^{(\hat{H}^1_0)^n} \rangle_{L_2(\Omega)^n} 
\end{array} \!\!\right] 
\!\!\langle \Psi^{(\hat{H}^1_0)^n}\!\!,{\textstyle \frac{(\vec{\tilde{u}}\cdot\grad) \vec{\tilde{u}}}{\nu^{3/2}}}\!-\!\vec{f}\!-\!\divv \underline{\tilde{\theta}}\!+\!\grad \tilde{p}
\rangle_{L_2(\Omega)^n}.\!
\end{split}$}\hspace*{-1em}
\end{equation}
Each of the terms  $\divv \vec{\tilde{u}}-g$, $\grad \vec{\tilde{u}}-\underline{\tilde{\theta}}$ , $\nu^{-3/2} (\vec{\tilde{u}}\cdot\grad) \vec{\tilde{u}}-\vec{f}-\divv \underline{\tilde{\theta}}+\grad \tilde{p}$
correspond, in strong form, to a term of the least squares functional, and therefore their norms can be bounded by a multiple of the norm of the residual, which is the basis of our approximate residual evaluation.

This approximate residual evaluation follows the same lines as with the elliptic problem from Sect.~\ref{Selliptic}.
Actually, things are easier here because we assume homogeneous boundary conditions.
Selecting the Riesz bases for the Cartesian products  $H^1_0(\Omega)^n$ and $L_2(\Omega)^{n^2}$ of canonical form,
we assume that all scalar-valued bases $\Psi^\ast$ for $\ast \in \{\hat{H}^1_0, H^1_0,L_2/\R,L_2\}$ satisfy the assumptions that were made in Sect.~\ref{Swavelets}, in particular \eqref{w1}--\eqref{w8}.
Let $\Lambda:=\supp [{\bf \tilde{u}}^\top\!,{\bf \tilde{p}}^\top\!,\bm{\tilde{\theta}}^\top]^\top$ be admissible, i.e., $\Lambda \cap \vee_*$ are trees.
\renewcommand{\theenumi}{s\arabic{enumi}}
\begin{enumerate}
\item \label{ss1} Find a tiling ${\mathcal T}(\eps) \subset \cO_\Omega$, such that
$$\inf_{\vec{f}_\eps \in \cP_m({\mathcal T}(\eps))^n,\,g_\eps \in \cP_m({\mathcal T}(\eps))/\R} \|\vec{f}-\vec{f}_\eps\|_{H^{-1}(\Omega)^n} +\|g-g_\eps\|_{L_2(\Omega)} \leq \eps.
$$
If $ [{\bf {u}}^\top\!,{\bf {p}}^\top\!,\bm{{\theta}}^\top]^\top \in \cA^s$, then such a tiling exists with $\# {\mathcal T}(\eps) \lesssim \eps^{-1/s}$.
Set ${\mathcal T}(\Lambda,\eps):={\mathcal T}(\Lambda) \oplus {\mathcal T}(\eps)$.
\item  
\begin{enumerate} \item
Approximate ${\bf r}^{(\frac{1}{2})}_1:=\langle \Psi^{(\hat{H}^1_0)^n},\nu^{-3/2} (\vec{\tilde{u}}\!\cdot\!\grad) \vec{\tilde{u}}-\vec{f}-\divv \underline{\tilde{\theta}}+\grad \tilde{p}
\rangle_{L_2(\Omega)^n}$ by ${\bf \tilde{r}}^{(\frac{1}{2})}_1:={\bf r}^{(\frac{1}{2})}_1|_{\Lambda^{(\hat{H}^1_0)^n}({\mathcal T}(\Lambda,\eps),k)}$. 
\item With $\tilde{r}_1^{(\frac{1}{2})}:=({\bf \tilde{r}}^{(\frac{1}{2})}_1)^\top \Psi^{(H^1_0)^n}$, approximate
$$
{\bf r}_1=\left[\begin{array}{@{}c@{}} {\bf r}_{11} \\ {\bf r}_{12}\\ {\bf r}_{13}\end{array} \right]:=
\left[\begin{array}{@{}c@{}}  
\langle \frac{(\vec{\tilde{u}}\cdot\grad) \Psi^{(H^1_0)^n}+(\Psi^{(H^1_0)^n}\cdot\grad) \vec{\tilde{u}}}{\nu^{3/2}},\tilde{r}_1^{(\frac{1}{2})}\rangle_{L_2(\Omega)^n}\\
-\langle \Psi^{L_2/\R},\divv \tilde{r}_1^{(\frac{1}{2})} \rangle_{L_2(\Omega)} \\
-\langle \divv \Psi^{L_2^{n^2}}, \tilde{r}_1^{(\frac{1}{2})} \rangle_{L_2(\Omega)^n} 
\end{array} \right] 
$$
by 
${\bf \tilde{r}}_1:={\bf r}_1|_{\Lambda({\mathcal T}(\Lambda^{(H^1_0)^n}({\mathcal T}(\Lambda,\eps),k)),k)}$.
\end{enumerate}
\item 
Approximate
$$
{\bf r}_2=\left[\begin{array}{@{}c@{}} {\bf r}_{21} \\ {\bf r}_{22}\\ {\bf r}_{23}\end{array} \right]:=\left[\begin{array}{@{}c@{}} 
\langle \divv  \Psi^{(H^1_0)^n}, \divv \vec{\tilde{u}}-g  \rangle_{L_2(\Omega)}\!\\
0_{\vee_{L_2/\R}}\\
0_{\vee_{L_2^{n^2}}}
\end{array} \!\right] \text{ by }{\bf \tilde{r}}_2:={\bf r}_2|_{\Lambda({\mathcal T}(\Lambda,\eps),k)}
$$
\item \label{ss4}
Approximate
$$
{\bf r}_3=\left[\begin{array}{@{}c@{}} {\bf r}_{31} \\ {\bf r}_{32}\\ {\bf r}_{33}\end{array} \right]:=
\left[\begin{array}{@{}c@{}}
\langle  \grad \Psi^{(H^1_0)^n},  \grad \vec{\tilde{u}}-\underline{\tilde{\theta}} \rangle_{L_2(\Omega)^{n^2}}\\
0_{\vee_{L_2/\R}}\\
\langle  \Psi^{L_2^{n^2}}, \underline{\tilde{\theta}} - \grad \vec{\tilde{u}}\rangle_{L_2(\Omega)^{n^2}}
\end{array}\right] \text{ by } {\bf \tilde{r}}_3:={\bf r}_3|_{\Lambda({\mathcal T}(\Lambda,\eps),k)}.
$$
\end{enumerate}

The same arguments (actually a subset) that led to Theorem~\ref{thm2} show the following theorem.
\begin{theorem} 
For an admissible  $\Lambda \subset \vee$, $[{\bf \tilde{u}}^\top\!,{\bf \tilde{p}}^\top\!,\bm{\tilde{\theta}}^\top]^\top \in \ell_2(\Lambda)$ with 
$(\vec{\tilde{u}},\tilde{p},\underline{\tilde{\theta}})$ sufficiently close to $(\vec{u},p,\underline{\theta})$, and an $\eps>0$, consider the steps \eqref{ss1}-\eqref{ss4}.
With $s>0$ such that $[{\bf \tilde{u}}^\top\!,{\bf \tilde{p}}^\top\!,\bm{\tilde{\theta}}^\top]^\top  \in \cA^s$, 
it holds that
\begin{align*}
\|D{\bf Q}([{\bf \tilde{u}}^\top\!,{\bf \tilde{p}}^\top\!,&\bm{\tilde{\theta}}^\top]^\top)-({\bf \tilde{r}}_1+{\bf \tilde{r}}_2+{\bf \tilde{r}}_3)\|\\
& \lesssim 
2^{-k/2} (\|\vec{u}-\vec{\tilde{u}}\|_{H^1_0(\Omega)^n}+\|p-\tilde{p}\|_{L_2(\Omega)} +\|\underline{\theta}-\underline{\tilde{\theta}}\|_{L_2(\Omega)^{n^2}})+\eps,
\end{align*}
where the computation of ${\bf \tilde{r}}_1+{\bf \tilde{r}}_2+{\bf \tilde{r}}_3$ requires ${\mathcal O}(\# \Lambda +\eps^{-1/s})$ operations.
So by taking $k$ sufficiently large, Condition~\ref{nonlineareval*} is satisfied.
\end{theorem}

We conclude that the {\bf awgm} is an optimal solver for the stationary Navier-Stokes equations in the form $D{\bf Q}([{\bf {u}}^\top\!,{\bf {p}}^\top\!,\bm{{\theta}}^\top]^\top)=0$ resulting from the velocity--pressure--velocity gradient formulation.
Obviously, we cannot claim or even expect that this holds true uniformly in a vanishing viscosity parameter $\nu$. This because in the limit already well-posedness of $DG(\vec{u},p)$ cannot be expected.

\subsection{Velocity--pressure--vorticity formulation} \label{SVPV}
Restricting to $n \in \{2,3\}$, we set
\framebox{$\PP:=L_2(\Omega)^{2n-3}$}, and define
$$
G_1 \in \cL(\PP,\UU'),\qquad G_2 \in \cL(\UU,\PP),
$$
by
$$
G_2 (\vec{u},p)= \curl \vec{u},\qquad (G_1 \vec{\omega} )(\vec{v},q)=\int_\Omega \vec{\omega} \cdot \curl \vec{v} \,d x
$$
where for $n=2$, $\curl$ should be read as the scalar-valued operator $\vec{v} \mapsto \partial_x v_2-\partial_y v_1$.
 (and so $\vec{\omega}\cdot \curl \vec{u}$ as $\omega \curl \vec{u}$).
 The (formal) adjoint $\curl'$ equals $\curl$ for $n=3$, whereas for $n=2$ it is $v \mapsto[\partial_y v,-\partial_x v]^\top$.

 Since a vector field in the current space $\PP$ has $2n-2$ components, instead of $n^2$ as in the previous subsection, the first order system formulation studied in this subsection is more attractive. As we will see, later in its derivation it will be needed that $g=0$, i.e., $\divv \vec{u}=0$.

Using that on $H_0^1(\Omega)^n \times H_0^1(\Omega)^n $, $\int_\Omega \grad \vec{u} :\grad\vec{v}-\divv \vec{u} \divv \vec{v}-\curl \vec{u} \cdot \curl \vec{v} \,dx=0$,
the results from Sect.~\ref{S1} show that the solution $(\vec{u},p)$ can be found as the first component of the solution in $\UU \times \PP$ of the system
\begin{align*}
\vec{H}_1(\vec{u},p,\vec{\omega}):=
\Big(
(\vec{v},q) \mapsto \int_\Omega \vec{\omega} \cdot \curl \vec{v}+ \divv \vec{u} &\divv \vec{v}-p \divv \vec{v} +\nu^{-3/2}(\vec{u}\cdot \nabla) \vec{u} \cdot \vec{v}\\
&+q (\divv \vec{u}-g) -\vec{f}\cdot \vec{v}
\,dx,\vec{\omega}-\curl \vec{u}\Big)=\vec{0}
\end{align*}
on $\UU' \times \PP$, being a minimizer of 
\[
\begin{split}
Q_1(\vec{u},p,\vec{\omega}):={\textstyle \frac{1}{2}}\Big(&\big\|
\vec{v} \mapsto 
\int_\Omega 
\vec{\omega} \cdot \curl \vec{v}+ \divv \vec{u} \divv \vec{v}-p \divv \vec{v} +{\textstyle \frac{(\vec{u}\cdot \nabla) \vec{u} \cdot \vec{v}}{\nu^{3/2}}} -\vec{f}\cdot \vec{v}
\,dx
\|_{H^{-1}(\Omega)^n}^2\\
&+\|\divv \vec{u}-g\|_{L_2(\Omega)}^2+\|\vec{\omega}- \curl \vec{u}\|_{L_2(\Omega)^{2n-3}}^2\Big).
\end{split}
\]
The function $\vec{\omega}=\curl \vec{u}$ is known as the vorticity.

Since $G$ satisfies \eqref{r1}--\eqref{r3}, $\vec{H}_1$ satisfies \eqref{r4}--\eqref{r6}, and so by Lemma~\ref{lem2},
\begin{equation} \label{similar}
Q_1(\vec{\tilde{u}},\tilde{p},\vec{\tilde{\omega}}) \eqsim \|(\vec{\tilde{u}},\tilde{p},\vec{\tilde{\omega}})-(\vec{u},{p},\vec{\omega})\|^2_{\UU \times \PP}
\end{equation}
 for $(\vec{\tilde{u}},\tilde{p},\vec{\tilde{\omega}})$ in a neighborhood of $(\vec{u},{p},\vec{{\omega}})$.

From here on, we assume that
$$
g=0,
$$
so that the velocities component of the exact solution is divergence-free.
This will allow us to get rid of  the second order term $\grad \divv \vec{u}$ in the definition of $\vec{H}_1$. We define
$$
\vec{H}_2(\vec{u},p,\vec{\omega}):=
\Big(
(\vec{v},q) \mapsto \int_\Omega \vec{\omega} \cdot \curl \vec{v} -p \divv \vec{v} +{\textstyle \frac{(\vec{u}\cdot \nabla) \vec{u} \cdot \vec{v}}{\nu^{3/2}}}
+q \divv \vec{u} -\vec{f}\cdot \vec{v}
\,dx,\vec{\omega}-\curl \vec{u}\Big)
$$
with corresponding quadratic functional
\[
\begin{split}
Q_2(\vec{u},p,\vec{\omega}):={\textstyle \frac{1}{2}}\Big(&\big\|
\vec{v} \mapsto 
\int_\Omega 
\vec{\omega} \cdot \curl \vec{v}-p \divv \vec{v} +{\textstyle \frac{(\vec{u}\cdot \nabla) \vec{u} \cdot \vec{v}}{\nu^{3/2}}} -\vec{f}\cdot \vec{v}
\,dx
\|_{H^{-1}(\Omega)^n}^2\\
&+\|\divv \vec{u}\|_{L_2(\Omega)}^2+\|\vec{\omega}- \curl \vec{u}\|_{L_2(\Omega)^{2n-3}}^2\Big).
\end{split}
\]

Clearly the solution of $\vec{H}_1(\vec{u},p,\vec{\omega})=0$ is a solution of $\vec{H}_2(\vec{u},p,\vec{\omega})=0$ (\eqref{r4}), and $\vec{H}_2$ is two times continuously differentiable (\eqref{r5}).
From $\|\vec{v} \mapsto 
\int_\Omega  \divv \vec{u} \divv \vec{v}\,dx
\|_{H^{-1}(\Omega)^n}$ $ \lesssim \|\divv \vec{u}\|_{L_2(\Omega)}$, one infers that $Q_1\lesssim Q_2$ by the triangle inequality, and analogously $Q_2\lesssim Q_1$.
Thanks to \eqref{similar}, an application of Lemma~\ref{lem2} shows that $\vec{H}_2$ satisfies also \eqref{r6}. We conclude that $\vec{H}_2(\vec{u},p,\vec{\omega})$ is a well-posed first order system formulation of $G(\vec{u},p)=0$, and consequently, that $(\vec{u},p,\vec{\omega})$ can be found by solving the normal equations 
$DQ_2(\vec{u},p,\vec{\omega})=0$, which are well-posed in the sense that they satisfy \eqref{38}--\eqref{41}.
This gives us an alternative, effortless proof of \cite[Thm.~2.1]{35.93006}.

As usual,  to deal with the `unpractical'  norm on $H^{-1}(\Omega)^n$,  we equip $H^1_0(\Omega)^n$ 
 with a wavelet Riesz basis
$$
\Psi^{(\hat{H}^1_0)^n}=\{\psi_\lambda^{(\hat{H}^1_0)^n}\colon\lambda \in \vee_{(\hat{H}^1_0)^n}\},
$$
and replace, in the definition of $Q_2$,  the norm on its dual by the equivalent norm $\|\vec{g}(\Psi^{(H^1_0)^n})\|$  for $\vec{g} \in H^{-1}(\Omega)^n$.

Next, after equipping $* \in \{H^1_0(\Omega)^n, L_2(\Omega)/\R, L_2(\Omega)^{2n-3}\}$ with Riesz basis 
$\Psi^{*}=\{\psi_\lambda^{*} \colon \lambda\in\vee_{*}\}$, and so $H^1_0(\Omega)^n\times L_2(\Omega)/\R\times L_2(\Omega)^{2n-3}$ with 
$$
\Psi:=(\Psi^{(H^1_0)^n},0_{L_2/\R},0_{L_2^{2n-3}}) \cup (0_{(H^1_0)^n},\Psi^{L_2/\R},0_{L_2^{2n-3}})\cup(0_{(H^1_0)^n},0_{L_2/\R},\Psi^{L_2^{2n-3}})
$$
with index set $\vee:=\vee_{(H^1_0)^n}\cup\vee_{L_2/\R}\cup \vee_{L_2^{2n-3}}$, we apply the {\bf awgm} to the resulting system
\[
\begin{split}
& D{\bf Q}_2([{\bf {u}}^\top\!,{\bf {p}}^\top\!,\bm{\omega}^\top]^\top)=\left[\begin{array}{@{}c@{}} 
\langle \divv  \Psi^{(H^1_0)^n}, \divv \vec{{u}}  \rangle_{L_2(\Omega)}\\
0_{\vee_{L_2/\R}}\\
0_{\vee_{L_2^{2n-3}}}
\end{array} \right] 
+
\left[\begin{array}{@{}c@{}}
\langle  \curl \Psi^{(H^1_0)^n},  \curl \vec{{u}}-\vec{\omega}\rangle_{L_2(\Omega)^{2n-3}}\\
0_{\vee_{L_2/\R}}\\
\langle  \Psi^{L_2^{2n-3}}, \vec{{\omega}} - \curl \vec{{u}}\rangle_{L_2(\Omega)^{2n-3}}
\end{array} \right] 
\\
&+ \left[\begin{array}{@{}c@{}}  
\langle \frac{(\vec{{u}}\cdot\grad) \Psi^{(H^1_0)^n}+(\Psi^{(H^1_0)^n}\cdot\grad) \vec{{u}}}{\nu^{3/2}},\Psi^{(\hat{H}^1_0)^n}\rangle_{L_2(\Omega)^n}\\
-\langle \Psi^{L_2/\R},\divv \Psi^{(\hat{H}^1_0)^n} \rangle_{L_2(\Omega)} \\
\langle  \Psi^{L_2^{2n-3}}, \curl \Psi^{(\hat{H}^1_0)^n} \rangle_{L_2(\Omega)^{2n-3}} 
\end{array} \right] 
\Big\{\langle \Psi^{(\hat{H}^1_0)^n},{\textstyle \frac{(\vec{{u}}\cdot\grad) \vec{{u}}}{\nu^{3/2}}}-\vec{f}\rangle_{L_2(\Omega)^n}
+\\& \hspace*{14em}
\langle \curl \Psi^{(\hat{H}^1_0)^n}, \vec{\omega}\rangle_{L_2(\Omega)^{2n-3}}-
\langle \divv  \Psi^{(\hat{H}^1_0)^n}, p\rangle_{L_2(\Omega)}
\Big\}
=0.
\end{split}
\]

To express the three terms in $\vec{v} \mapsto \langle \vec{v},\nu^{-3/2}(\vec{{u}}\cdot\grad) \vec{{u}}-\vec{f}\rangle_{L_2(\Omega)^n}
+
\langle \curl \vec{v}, \vec{\omega}\rangle_{L_2(\Omega)^{2n-3}}-
\langle \divv  \vec{v}, p\rangle_{L_2(\Omega)}
$ w.r.t. one dictionary, we impose the easily realizable conditions that
$$
\Psi^{L_2/\R} \subset H^1(\Omega),\, \Psi^{L_2^{2n-3}} \subset H(\curl;\Omega)
$$
Then for finitely supported approximations $[{\bf \tilde{u}}^\top\!,{\bf \tilde{p}}^\top\!,\bm{\tilde{\omega}}^\top]^\top$ to $[{\bf u}^\top\!,{\bf p}^\top\!,\bm{\omega}^\top]^\top$, for $(\vec{\tilde{u}},\tilde{p},\vec{\tilde{\omega}}):=[{\bf \tilde{u}}^\top\!,{\bf \tilde{p}}^\top\!,\bm{\tilde{\omega}}^\top] \Psi \in H^1_0(\Omega)^n \times H^1(\Omega) \times H(\curl';\Omega)$, we have
\[
\framebox{$\begin{split}
& \!\!D{\bf Q}_2([{\bf \tilde{u}}^\top\!,{\bf \tilde{p}}^\top\!,\bm{\tilde{\omega}}^\top]^\top)\!=\!
\left[\begin{array}{@{}c@{}} 
\langle \divv  \Psi^{(H^1_0)^n}\!,\! \divv \vec{\tilde{u}}  \rangle_{L_2(\Omega)}\!\\
0_{\vee_{L_2/\R}}\\
0_{\vee_{L_2^{2n-3}}}
\end{array} \!\right] 
\!+\!
\left[\!\begin{array}{@{}c@{}}
\langle  \curl \Psi^{(H^1_0)^n},  \curl \vec{\tilde{u}}-\vec{\tilde{\omega}} \rangle_{L_2(\Omega)^{2n-3}}\\
0_{\vee_{L_2/\R}}\\
\langle  \Psi^{L_2^{2n-3}}, \vec{\tilde{\omega}} - \grad \vec{\tilde{u}}\rangle_{L_2(\Omega)^{2n-3}}
\end{array}\! \right] \!+\!\!\\
&\!\!\!\!\left[\begin{array}{@{}c@{}}  
\!\langle \frac{(\vec{\tilde{u}}\cdot\grad) \Psi^{(H^1_0)^n}+(\Psi^{(H^1_0)^n}\cdot\grad) \vec{\tilde{u}}}{\nu^{3/2}},\Psi^{(\hat{H}^1_0)^n}\!\rangle_{L_2(\Omega)^n}\\
-\langle \Psi^{L_2/\R},\divv \Psi^{(\hat{H}^1_0)^n} \rangle_{L_2(\Omega)} \\
\langle  \Psi^{L_2^{2n-3}}, \curl \Psi^{(\hat{H}^1_0)^n} \rangle_{L_2(\Omega)^{2n-3}} 
\end{array} \!\!\right] 
\!\!\langle \Psi^{(\hat{H}^1_0)^n}\!\!,\!{\textstyle \frac{(\vec{\tilde{u}}\cdot\grad) \vec{\tilde{u}}}{\nu^{3/2}}}\!-\!\vec{f}\!+\!\curl'\vec{\tilde{\omega}}\!+\!\grad \tilde{p}
\rangle_{L_2(\Omega)^n}.\!
\end{split}$}\hspace*{-1em}
\]

The design of an approximate residual evaluation follows analogous steps as in the previous subsection.
Equipping Cartesian products with bases of canonical form, and assuming that the scalar-valued bases $\Psi^\ast$ for $\ast \in \{\hat{H}^1_0, H^1_0,L_2/\R,L_2\}$ satisfy \eqref{w1}--\eqref{w8}, and that $[{\bf \tilde{u}}^\top\!,{\bf \tilde{p}}^\top\!,\bm{\tilde{\omega}}^\top]^\top$ is supported on an admissible set,
four steps fully analogous to \eqref{ss1}--\eqref{ss4} in the previous subsection define an approximation scheme that satisfies Condition~\ref{nonlineareval*}.
We conclude that the {\bf awgm} is an optimal solver for the stationary Navier-Stokes equations in the form $D{\bf Q}([{\bf {u}}^\top\!,{\bf {p}}^\top\!,\bm{{\theta}}^\top]^\top)=0$ resulting from the velocity--pressure--vorticity formulation.
Again, also here we cannot claim or even expect that this holds true uniformly in a vanishing viscosity parameter $\nu$.

\section{Conclusion}
We have seen that a well-posed (system of) 2nd order PDE(s) can always be formulated as a well-posed 1st order least squares system.
The arising dual norm(s) can be replaced by the equivalent $\ell_2$-norm(s) of the wavelet coefficients of the functional.
The resulting Euler-Lagrange equations, also known as the (nonlinear) normal equations, can be solved at the best possible rate by the adaptive wavelet Galerkin method.
We developed a new approximate residual evaluation scheme that also for semi-linear problems
satisfies the condition for optimal computational complexity, and that is quantitatively much more efficient than the usual {\bf apply} scheme.
Moreover, regardless of the order of the wavelets, it applies already to wavelet bases that have only one vanishing moment.
As applications we discussed optimal solvers for
first order least squares reformulations of 2nd order elliptic PDEs with inhomogeneous boundary conditions, and that of the stationary Navier-Stokes equations.
In a forthcoming work, we will apply this approach to time-evolution problems.

\%bibliography{../ref}

\appendix
\section{Decay estimates} \label{Sdecay}
We collect a number of decay estimates that have been used in the proof of Theorem~\ref{thm2}.
Recall the definition of the spaces $\UU$ and $\VV$ given at the beginning of Sect.~\ref{Sreformulation}.

The following proposition and subsequent lemma have been used to bound $\|{\bf r}^{(\frac{1}{2})}_1-{\bf \tilde{r}}^{(\frac{1}{2})}_1\|$.
The presence of the boundary integral and the fact that the upper bound that is given depends on the norm of $g$ as a whole, and not on norms of $g_1$ and $g_2$ requires a non-standard treatment.

\begin{proposition} \label{prop1}
For a tiling ${\mathcal T} \subset \cO_\Omega$, let $g \in \VV_1'$
be of the form
$$
g(v)=\int_{\Omega} g_1 v \,dx+\int_{\Gamma_{N}} \vec{g}_2 \cdot {\bf n}v \,ds,
$$
where $g_1 \in \cP_m(\tria)$, $\vec{g}_2 \in \cP_m(\tria)^n$. Then
$$
\big\| g(\Psi^{\VV_1}) \big|_{\vee_{\VV_1} \setminus \Lambda^{\VV_1}(\tria,k)}\big\|\lesssim 2^{-k} \|g\|_{\VV_1'}
$$
{\rm(}uniform in ${\mathcal T}$ and $g${\rm)}.
\end{proposition}

\begin{proof} Since by assumption \eqref{w8}, for $\lambda \in \vee_{\VV_1}$ with $|\lambda|>0$ either $\int_\Omega \psi_\lambda^{\VV_1} \,dx =0$ or ${\rm dist}({\rm supp}\,\psi_\lambda^{\VV},\Gamma_D) \lesssim 2^{-|\lambda|}$, an application of Friedrich's or Poincar\'{e}'s inequality shows that
$\|\psi_\lambda^{\VV_1}\|_{L_2(\Omega)} \lesssim 2^{-|\lambda|} |\psi_\lambda^{\VV_1}|_{H^1(\Omega)}\eqsim 2^{-|\lambda|}$.

Since by  \eqref{w1}--\eqref{w2}, each descendant $\omega' \in \cO_\Omega$ of $\omega \in {\mathcal T}$ with $|\omega'|>|\omega|+k$ is intersected by the supports of a uniformly bounded number of $\lambda \in \vee_{\VV_1} \setminus \Lambda^{\VV_1}(\tria,k)$ with $|\lambda|=|\omega'|$, we have
\begin{equation} \label{10}
\begin{split}
\sum_{\lambda \in \vee_{\VV_1} \setminus \Lambda^{H^1_0}(\tria,k)} \Big| \int_\Omega g_1 \psi_\lambda^{\VV_1}  \, dx\Big|^2
&\leq 
\sum_{\lambda \in \vee_{\VV_1} \setminus \Lambda^{H^1_0}(\tria,k)} \sum_{\omega \in {\mathcal T}} 4^{-|\lambda|} \|g_1\|^2_{L_2(\omega \cap \supp \psi^{\VV_1}_\lambda)}\\
&\lesssim 4^{-k} \sum_{\omega \in {\mathcal T}} 4^{-|\omega|} \|g_1\|_{L_2(\omega)}^2,
\end{split}
\end{equation}

A standard homogeneity argument shows that for $\omega \in {\mathcal T}$ and $v \in H^1(\omega)$, $\|v\|_{L_2(\partial \omega)} \lesssim 2^{-|\omega|/2}(|v|_{H^1(\omega)}+2^{|\omega|}\|v\|_{L_2(\omega)})$, so that $\|\psi^{\VV_1}_\lambda\|_{L_2(\Gamma_N)}\lesssim 2^{-|\lambda|/2}$.
Writing $\partial \omega \cap \Gamma_N$ as $\partial \omega_N$, the arguments that led to \eqref{10} show that
\begin{equation} \label{11}
\begin{split}
\sum_{\lambda \in \vee_{\VV_1} \setminus \Lambda^{H^1_0}(\tria,k)} \Big| \int_{\Gamma_N} \vec{g}_2 \cdot {\bf n} \psi_\lambda^{\VV_1}  \, ds\Big|^2 & \lesssim
 \sum_{\lambda \in \vee_{\VV_1} \setminus \Lambda^{L_2}(\tria,k)} \sum_{\omega \in {\mathcal T}} 2^{-|\lambda|} \|\vec{g}_2 \cdot {\bf n}\|^2_{L_2(\omega_N \cap \supp \psi^{\VV_1}_\lambda)}
\\
&
\lesssim 2^{-k} \sum_{\omega \in {\mathcal T}} 2^{-|\omega|} \|\vec{g}_2 \cdot {\bf n}\|_{L_2(\omega_N)}^2.
\end{split}
\end{equation}

By combining \eqref{10}, \eqref{11} with Lemma~\ref{lem3}, the proof is completed unless $\UU=\VV_1=H^1(\Omega)/\R$.
In the latter case, define $\bar{g}_1:=g_1-\meas(\Omega)^{-1} g(\mathbb{1})$ and $\bar{g}(v):=\int_{\Omega} \bar{g}_1 v \,dx+\int_{\Gamma_{N}} \vec{g}_2 \cdot {\bf n}v \,ds$.
From $g(\Psi^{\VV_1}) \big|_{\vee_{\VV_1} \setminus \Lambda^{\VV_1}(\tria,k)}=\bar{g}(\Psi^{\VV_1}) \big|_{\vee_{\VV_1} \setminus \Lambda^{\VV_1}(\tria,k)}$,
and $\|g\|_{\VV_1'}=\|\bar{g}\|_{\VV_1'}$, applications of \eqref{10}, \eqref{11} and that of Lemma~\ref{lem3} to $\bar{g}$ complete the proof in this case.
\end{proof}

\begin{lemma} \label{lem3} In the situation of Proposition~\ref{prop1}, with additionally  $g(\mathbb{1})=0$ when $\UU=\VV_1=H^1(\Omega)/\R$, it holds that 
$$
\sum_{\omega \in {\mathcal T}} 4^{-|\omega|} \|g_1|_{\omega}\|_{L_2(\omega)}^2 +
2^{-|\omega|} \|\vec{g}_2\cdot{\bf n}|_{\omega}\|_{L_2(\partial \omega \cap \Gamma_{N})}^2 
\lesssim \|g\|_{\VV'}^2.
$$
\end{lemma}

\begin{proof}
Thanks to the uniform shape regularity condition, for any $\omega \in \cO_\Omega$, there exists a $V_\omega \subset H^1_0(\omega)$ such that
\begin{align*}
\|v\|_{H^1(\omega)} &\lesssim 2^{|\omega|} \|v\|_{L_2(\omega)} \quad (v \in V_\omega),\\
\|p\|_{L_2(\omega)} & \lesssim \sup_{0 \neq v \in V_\omega} \frac{\int_\omega p v\,dx}{\|v\|_{L_2(\omega)}} \quad (p \in \cP_m(\omega)).
\end{align*}
For each $\omega \in {\mathcal T}$, select $v_\omega \in V_\omega$ with $\|g_1|_{\omega}\|_{L_2(\omega)} 
\|v_\omega\|_{L_2(\omega)} \lesssim \int_\omega g_1|_{\omega} v_\omega\,dx$, and $\|v_\omega\|_{L_2(\omega)} = 4^{-|\omega|} \|g_1|_{\omega}\|_{L_2(\omega)}$. Then, with $v=\sum_{\omega \in {\mathcal T}} v_\omega \in H^1_0(\Omega)$, we have $
\sum_{\omega \in {\mathcal T}} 4^{-|\omega|} \|g_1|_{\omega}\|_{L_2(\omega)}^2 \lesssim \int_\Omega g_1 v \,dx$. By combining this with 
$$
\|v\|_{H^1(\Omega)}^2=\sum_{\omega \in {\mathcal T}} \|v_\omega\|_{H^1(\omega)}^2 \lesssim \sum_{\omega \in {\mathcal T}} 4^{|\omega|} \|v_\omega\|_{L_2(\omega)}^2=
\sum_{\omega \in {\mathcal T}}  4^{-|\omega|} \|g_1|_{\omega}\|_{L_2(\omega)}^2,
$$
and $g(v)=\int_{\Omega} g_1 v \,dx$,
we arrive at $\sqrt{\sum_{\omega \in {\mathcal T}}  4^{-|\omega|} \|g_1|_{\omega}\|_{L_2(\omega)}^2} \lesssim \frac{g(v)}{\|v\|_{H^1(\Omega)}} \leq \|g\|_{\VV'}$,
with the last inequality being valid when $H^1_0(\Omega) \subset \VV_1$.

Otherwise, when $\VV_1=H^1(\Omega)/\R$, we take $\bar{v}=v-\frac{\int_\Omega v \,dx}{\meas(\Omega)} \mathbb{1} \in \VV$.  Then $\|\bar{v}\|_{H^1(\Omega)} \lesssim \|v\|_{H^1(\Omega)}$, $g(v)=g(\bar{v})$ by assumption, and so $\frac{g(v)}{\|v\|_{H^1(\Omega)}}\lesssim \frac{g(\bar{v})}{\|\bar{v}\|_{H^1(\Omega)}}\leq \|g\|_{\VV'}$.


For bounding the second term in the statement of the lemma, for a tile $\omega \in \cO_\Omega$ we write $\partial \omega \cap \Gamma_N$ as $\partial \omega_N$.
Thanks to the uniform shape regularity condition, for any $\omega \in \cO_\Omega$ with $\meas(\partial\omega_N)>0$, there 
exists a $V_\omega \subset \{v \in H^1(\omega)\colon v|_{\partial\omega \setminus \partial\omega_N}=0\}$ such that
\begin{align*}
&\|v\|_{H^1(\omega)} \lesssim 2^{|\omega|/2} \|v\|_{L_2(\partial\omega_N)} \quad (v \in V_\omega),\\
&\|\vec{p}\cdot {\bf n}\|_{L_2(\partial\omega_N)}  \lesssim \sup_{0 \neq v \in V_\omega} \frac{\int_{\partial\omega_N} \vec{p}\cdot{\bf n} v\,ds}{\|v\|_{L_2(\partial\omega_N)}} \quad (\vec{p} \in \cP_m(\omega)^n),\\
&V_\omega  \perp_{L_2(\omega)} \cP_m(\omega).
\end{align*}

For each $\omega \in {\mathcal T}$ with $\meas(\partial \omega_{N})>0$, select $v_\omega \in V_\omega$ with $\|\vec{g}_2|_{\omega}\cdot {\bf n}\|_{L_2(\partial \omega_N)} 
\|v_\omega\|_{L_2(\partial \omega_N)} \lesssim \int_{\partial \omega_N} \vec{g}_2|_{\omega} \cdot {\bf n} v_\omega\,ds$, and $\|v_\omega\|_{L_2(\partial \omega_N)} = 2^{-|\omega|} \|\vec{g}_2 \cdot {\bf n}|_{\omega}\|_{L_2(\partial \omega_N)}$. For the other $\omega \in {\mathcal T}$, set $v_\omega=0$.
Then, for the function $v=\sum_{\omega \in {\mathcal T} } v_\omega \in \{w \in H^1(\Omega)\colon w|_{\Gamma_D}=0\}$, we have
$$
\sum_{\omega \in {\mathcal T} } 2^{-|\omega|} \|\vec{g}_2 \cdot {\bf n}|_{\omega}\|_{L_2(\partial \omega_N)}^2 \lesssim   \int_{\Gamma_N} \vec{g}_2\cdot {\bf n} v \, ds.
$$
By combining this with 
$$
\|v\|_{H^1(\Omega)}^2=\sum_{\omega \in {\mathcal T}} \|v_\omega\|_{H^1(\omega)}^2 \lesssim \sum_{\omega \in {\mathcal T}} 2^{|\omega|} \|v_\omega\|_{L_2(\partial \omega_N)}^2=\sum_{\omega \in {\mathcal T}}  2^{-|\omega|} \|\vec{g}_2 \cdot {\bf n}|_{\omega}\|_{L_2(\partial \omega_N)}^2,
$$
and $g(v)=\int_{\Gamma_N} \vec{g}_2 \cdot {\bf n} v \,dx$,
we arrive at 
$$
\sqrt{\sum_{\omega \in {\mathcal T}} 2^{-|\omega|} \|g_2|_{\omega}\cdot {\bf n}\|_{L_2(\partial \omega_N)}^2}
 \lesssim \frac{g(v)}{\|v\|_{H^1(\Omega)}} \leq \|g\|_{\VV'},
 $$
in the case that $\{w \in H^1(\Omega)\colon w|_{\Gamma_D}=0\} \subset \VV_1$.

Otherwise, when $\VV_1=H^1(\Omega)/\R$, we take $\bar{v}=v-\frac{\int_\Omega v \,dx}{\meas(\Omega)} \mathbb{1} \in \VV$.  Then $\|\bar{v}\|_{H^1(\Omega)} \lesssim \|v\|_{H^1(\Omega)}$, $g(v)=g(\bar{v})$ by assumption, and so $\frac{g(v)}{\|v\|_{H^1(\Omega)}}\lesssim \frac{g(\bar{v})}{\|\bar{v}\|_{H^1(\Omega)}}\leq \|g\|_{\VV'}$.
\end{proof}

An easy version of the proof of Proposition~\ref{prop1} shows the following result, which has been used to bound $\|{\bf r}_{11}-{\bf \tilde{r}}_{11}\|$ in the proof of Theorem~\ref{thm2}.

\begin{proposition} \label{prop3}
For a tiling ${\mathcal T} \subset \cO_\Omega$, and $g \in P_m(\tria)$, it holds that
$$
\big\| \langle \Psi^{\UU},g\rangle_{L_2(\Omega)}\big|_{\vee_{\UU} \setminus \Lambda^{\UU}(\tria,k)}\big\| \lesssim 2^{-k} \|g\|_{\UU'} 
$$
{\rm(}uniform in ${\mathcal T}$ and $g${\rm)}.
\end{proposition}

The statements from the following proposition have been used to bound the terms $\|{\bf r}_{11}-{\bf \tilde{r}}_{11}\|$ (first statement) and $\|{\bf r}_2-{\bf \tilde{r}}_2\|$ (both statements) in the proof of Theorem~\ref{thm2}.

\begin{proposition} \label{prop2}
For a tiling ${\mathcal T} \subset \cO_\Omega$, $g \in P_m(\tria)$, $1 \leq q \leq n$, and $\vec{g} \in P_m(\tria)^n$, it holds that
\begin{align*}
\big\| \langle \Psi^{\PP_q},g\rangle_{L_2(\Omega)}\big|_{\vee_{\PP_q} \setminus \Lambda^{\PP_q}(\tria,k)}\big\| &\lesssim 2^{-k/2} \|g\|_{L_2(\Omega)} \\
\big\| \langle \grad \Psi^{\UU},\vec{g}\rangle_{L_2(\Omega)^n}\big|_{\vee_{\UU} \setminus \Lambda^{\UU}(\tria,k)}\big\| & \lesssim 2^{-k/2} \|\vec{g}\|_{L_2(\Omega)^n}
\end{align*}
{\rm(}uniform in ${\mathcal T}$, $g$, and $\vec{g}${\rm)}.
\end{proposition}

\begin{proof}  
Since by assumption \eqref{w2}, for $\lambda \in \vee_{{\PP_q}} \setminus \Lambda^{{\PP_q}}(\tria,k)$, $\supp \psi^{{\PP_q}}_\lambda$ has non-empty intersection with a uniformly bounded number of $\omega \in {\mathcal T}$, we have
\begin{equation} \label{18}
\sum_{\lambda \in \vee_{{\PP_q}} \setminus \Lambda^{{\PP_q}}(\tria,k)} \big|\sum_{\omega \in {\mathcal T}} \langle \psi^{{\PP_q}}_\lambda, g\rangle_{L_2(\omega)}\big|^2 
\lesssim
\sum_{\omega \in {\mathcal T}} \sum_{\lambda \in \vee_{{\PP_q}} \setminus \Lambda^{{\PP_q}}(\tria,k)} |\langle \psi_\lambda^{{\PP_q}}, g\rangle_{L_2(\omega)}|^2.
\end{equation}
Given $\omega \in {\mathcal T}$ and $\ell \in \N_0$, we set
\begin{align*}
\Lambda_{\omega,\ell}^{(1)}&=\{\lambda \in \vee_{{\PP_q}}\colon |\lambda|=\ell,\,\supp \psi_\lambda^{{\PP_q}} \subset \omega,\,\psi_\lambda^{{\PP_q}} \text{ has a vanishing moment}\}\\
\Lambda_{\omega,\ell}^{(2)}&=\{\lambda \in \vee_{{\PP_q}}\setminus \Lambda_{\omega,\ell}^{(1)} \colon |\lambda|=\ell,\,\meas(\supp \psi_\lambda^{{\PP_q}} \cap \omega)>0 \}.
\end{align*}

For $\lambda \in \Lambda_{\omega,\ell}^{(2)}$, we estimate
\begin{equation} \label{21}
|\langle \psi^{{\PP_q}}_\lambda, g\rangle_{L_2(\omega)}| \leq \|\psi^{{\PP_q}}_\lambda\|_{L_1(\Omega)} \|g\|_{L_\infty(\omega)} \lesssim 2^{-\ell n /2} 2^{|\omega| n/2}\|g\|_{L_2(\omega)}.
\end{equation}
Using that $\# \Lambda_{\omega,\ell}^{(2)} \lesssim 2^{(\ell-|\omega|)(n-1)}$ (cf. \eqref{w8}) , we infer that
\begin{equation} \label{19}
\sum_{\ell >|\omega|+k} \sum_{\lambda \in \Lambda_{\omega,\ell}^{(2)}} |\langle \psi_\lambda^{{\PP_q}}, g\rangle_{L_2(\omega)}|^2 
\lesssim  2^{-k} \|g\|^2_{L_2(\omega)}.
\end{equation}

Using that for $\lambda \in \Lambda_{\omega,\ell}^{(1)}$, $\psi_\lambda^{{\PP_q}}$ has a vanishing moment, we find  that
\begin{equation} \label{22}
|\langle \psi^{{\PP_q}}_\lambda, g\rangle_{L_2(\omega)}|  \lesssim \|\psi^{{\PP_q}}_\lambda\|_{L_1(\Omega)} 2^{-\ell} \|g\|_{W_\infty^1(\omega)}
\lesssim 2^{-\ell n/2} 2^{-\ell} 2^{|\omega|} 2^{|\omega|n/2} \|g\|_{L_2(\omega)}.
\end{equation}
From $\# \Lambda_{\omega,\ell}^{(2)} \lesssim 2^{(\ell-|\omega|)n}$, we obtain
\begin{equation} \label{20}
\sum_{\ell >|\omega|+k} \sum_{\lambda \in \Lambda_{\omega,\ell}^{(2)}} |\langle \psi_\lambda^{{\PP_q}}, g\rangle_{L_2(\omega)}|^2 
\lesssim  4^{-k} \|g\|^2_{L_2(\omega)}.
\end{equation}
The proof of the first inequality follows from \eqref{18}, \eqref{19}, and \eqref{20}.

To prove the second inequality,  for any $\lambda \in \vee_{\UU}$, similar to \eqref{21} we have
$$
|\langle  \grad \psi^{\UU}_\lambda, \vec{g}\rangle_{L_2(\omega)^n}| \leq \|\psi^{\UU}_\lambda\|_{W^1_1(\Omega)} \|\vec{g}\|_{L_\infty(\omega)^n} \lesssim 2^{-|\lambda| n /2} 2^{|\omega| n/2}\|\vec{g}\|_{L_2(\omega)^n}.
$$
When $\psi^{\UU}_\lambda$ has a vanishing moment and $\supp \psi^{\UU}_\lambda \subset \omega$, we have
\begin{align*}
|\langle \grad \psi^{\UU}_\lambda, \vec{g}\rangle_{L_2(\omega)^n}|
& =
|\langle \psi^{\UU}_\lambda, \divv  \vec{g}\rangle_{L_2(\omega)}|
\lesssim \|\psi^{\UU}_\lambda\|_{L_1(\Omega)} 2^{-|\lambda|} \|\vec{g}\|_{W_\infty^2(\omega)^n}
\\ &\lesssim 2^{-|\lambda| n/2} 4^{-|\lambda|} 4^{|\omega|} 2^{|\omega|/2} \|\vec{g}\|_{L_2(\omega)^n}
\end{align*}
replacing \eqref{22}.
From these two estimates, the second inequality follows similarly to the first one.

\end{proof}

The following proposition has been used to bound the term $\|{\bf r}^{(\frac{1}{2})}_3-{\bf \tilde{r}}^{(\frac{1}{2})}_3\|$
 in the proof of Theorem~\ref{thm2}.
\begin{proposition} \label{prop5}
For a boundary tiling ${\mathcal T}_{\Gamma_D} \subset \cO_{\Gamma_D}$, and $g \in \cP_m({\mathcal T}_{\Gamma_D}) \cap C(\Gamma_D)$,
$$
\|\langle\Psi^{\VV_2},g\rangle_{L_2(\Gamma_D)}|_{\vee_{\VV_2}\setminus \Lambda^{\VV_2}({\mathcal T}_{\Gamma_D},k)}\| \lesssim 2^{-k/2}\|g\|_{H^{\frac{1}{2}}(\Gamma_D)}.
$$
\end{proposition}

\begin{proof} Using \eqref{substitute}, similar to \eqref{10}
\[
\begin{split}
\sum_{\lambda \in \vee_{\VV_2}\setminus \Lambda^{\VV_2}({\mathcal T}_{\Gamma_D},k)} \Big| \int_{\Gamma_D} g \psi_\lambda^{\VV_2}  \, ds\Big|^2
&\leq 
\sum_{\lambda \in \vee_{\VV_2}\setminus \Lambda^{\VV_2}({\mathcal T}_{\Gamma_D},k)} \sum_{\omega \in {\mathcal T}_{\Gamma_D}} 2^{-|\lambda|} \|g\|^2_{H^1(\omega \cap \supp \psi^{\VV_2}_\lambda)}\\
&\lesssim 2^{-k} \sum_{\omega \in {\mathcal T}_{\Gamma_D}} 2^{-|\omega|} \|g\|_{H^1(\omega)}^2 \lesssim 2^{-k} \|g\|^2_{H^{\frac{1}{2}}(\Gamma_D)},
\end{split}
\]
by an application of an inverse inequality.
\end{proof}

The following proposition has been used to bound the term $\|{\bf r}_3-{\bf \tilde{r}}_3\|$
 in the proof of Theorem~\ref{thm2}.
\begin{proposition} \label{prop6}
For a boundary tiling ${\mathcal T}_{\Gamma_D} \subset \cO_{\Gamma_D}$, and $g \in \cP_m({\mathcal T}_{\Gamma_D})$,
$$
\|\langle\Psi^{\UU},g\rangle_{L_2(\Gamma_D)}|_{\vee_{\UU}\setminus \Lambda^{\UU}({\mathcal T}_{\Gamma_D},k)}\| \lesssim 2^{-k/2}\|g\|_{H^{-\frac{1}{2}}(\Gamma_D)}.
$$
\end{proposition}

\begin{proof} In the proof of Proposition~\ref{prop1}, we saw that $\|\Psi_\lambda^{\VV_1}\|_{L_2(\Gamma_N)} \lesssim 2^{-|\lambda|/2}$. The same arguments show that $\|\Psi_\lambda^\UU\|_{L_2(\Gamma_D)} \lesssim 2^{-|\lambda|/2}$. 
Consequently, similar to \eqref{10},
\[
\begin{split}
\sum_{\lambda \in \vee_{\UU}\setminus \Lambda^{\UU}({\mathcal T}_{\Gamma_D},k)} \Big| \int_{\Gamma_D} g \psi_\lambda^{\UU}  \, ds\Big|^2
&\leq 
\sum_{\lambda \in \vee_{\UU}\setminus \Lambda^{\UU}({\mathcal T}_{\Gamma_D},k)} \sum_{\omega \in {\mathcal T}_{\Gamma_D}} 2^{-|\lambda|} \|g\|^2_{L_2(\omega \cap \supp \psi^{\UU}_\lambda)}\\
&\lesssim 2^{-k} \sum_{\omega \in {\mathcal T}_{\Gamma_D}} 2^{-|\omega|} \|g\|_{L_2(\omega)}^2 \lesssim 2^{-k}\|g\|^2_{H^{-\frac{1}{2}}(\Gamma_D)},
\end{split}
\]
where the last inequality follows from analogous arguments as were applied in the first paragraph of the proof of Lemma~\ref{lem3}.
\end{proof}


\begin{thebibliography}{CMM97b}

\bibitem[Alp93]{10}
B.K. Alpert.
\newblock A class of bases in \({L}^2\) for the sparse representation of
  integral operators.
\newblock {\em {SIAM} J. Math. Anal.}, 24:246--262, 1993.

\bibitem[BD04]{20.5}
P.~Binev and R.~{DeV}ore.
\newblock Fast computation in adaptive tree approximation.
\newblock {\em Numer. Math.}, 97(2):193 -- 217, 2004.

\bibitem[BG09]{23.5}
P.~B. Bochev and M.~D. Gunzburger.
\newblock {\em Least-squares finite element methods}, volume 166 of {\em
  Applied Mathematical Sciences}.
\newblock Springer, New York, 2009.

\bibitem[BLP97]{35.83}
J.~H. Bramble, R.~D. Lazarov, and J.~E. Pasciak.
\newblock A least-squares approach based on a discrete minus one inner product
  for first order systems.
\newblock {\em Math. Comp.}, 66(219):935--955, 1997.

\bibitem[BS11]{18.68}
M.~Badiale and E.~Serra.
\newblock {\em Semilinear elliptic equations for beginners}.
\newblock Universitext. Springer, London, 2011.
\newblock Existence results via the variational approach.

\bibitem[BU08]{22.6}
K.~Bittner and K.~Urban.
\newblock Adaptive wavelet methods using semiorthogonal spline wavelets: sparse
  evaluation of nonlinear functions.
\newblock {\em Appl. Comput. Harmon. Anal.}, 24(1):94--119, 2008.

\bibitem[CDD01]{45.2}
A.~Cohen, W.~Dahmen, and R.~{DeV}ore.
\newblock Adaptive wavelet methods for elliptic operator {equations --
  Convergence} rates.
\newblock {\em Math. Comp}, 70:27--75, 2001.

\bibitem[CDD02]{45.25}
A.~Cohen, W.~Dahmen, and R.~{DeV}ore.
\newblock Adaptive wavelet methods {II} - {B}eyond the elliptic case.
\newblock {\em Found. Comput. Math.}, 2(3):203--245, 2002.

\bibitem[CDD03a]{45.26}
A.~Cohen, W.~Dahmen, and R.~{DeV}ore.
\newblock Adaptive wavelet schemes for nonlinear variational problems.
\newblock {\em {SIAM} J. Numer. Anal.}, 41:1785--1823, 2003.

\bibitem[CDD03b]{45.22}
A.~Cohen, W.~Dahmen, and R.~De{V}ore.
\newblock Sparse evaluation of compositions of functions using multiscale
  expansions.
\newblock {\em SIAM J. Math. Anal.}, 35(2):279--303 (electronic), 2003.

\bibitem[CDDD01]{45.21}
A.~Cohen, W.~Dahmen, I.~Daubechies, and R.~DeVore.
\newblock Tree approximation and optimal encoding.
\newblock {\em Appl. Comput. Harmon. Anal.}, 11(2):192--226, 2001.

\bibitem[CDN12]{45.47}
A.~Cohen, R.~DeVore, and R.~H. Nochetto.
\newblock Convergence rates of {AFEM} with {$H^{-1}$} data.
\newblock {\em Found. Comput. Math.}, 12(5):671--718, 2012.

\bibitem[CMM95]{35.93006}
Z.~Cai, T.~A. Manteuffel, and S.~F. McCormick.
\newblock First-order system least squares for velocity-vorticity-pressure form
  of the {S}tokes equations, with application to linear elasticity.
\newblock {\em Electron. Trans. Numer. Anal.}, 3(Dec.):150--159 (electronic),
  1995.

\bibitem[CMM97a]{35.9301}
Z.~Cai, T.~A. Manteuffel, and S.~F. McCormick.
\newblock First-order system least squares for second-order partial
  differential equations. {II}.
\newblock {\em SIAM J. Numer. Anal.}, 34(2):425--454, 1997.

\bibitem[CMM97b]{35.9302}
Z.~Cai, T.~A. Manteuffel, and S.~F. McCormick.
\newblock First-order system least squares for the {S}tokes equations, with
  application to linear elasticity.
\newblock {\em SIAM J. Numer. Anal.}, 34(5):1727--1741, 1997.

\bibitem[CP15]{37.3}
C.~Carstensen and E.-J. Park.
\newblock Convergence and {O}ptimality of {A}daptive {L}east {S}quares {F}inite
  {E}lement {M}ethods.
\newblock {\em SIAM J. Numer. Anal.}, 53(1):43--62, 2015.

\bibitem[CS15]{38.42}
N.G. Chegini and R.P. Stevenson.
\newblock An adaptive wavelet method for semi-linear first order system least
  squares.
\newblock {\em {Comput. Math. Appl.}}, August 2015.
\newblock {DOI: 10.1515/cmam-2015-0023}.

\bibitem[DHS07]{58.1}
W.~Dahmen, H.~Harbrecht, and R.~Schneider.
\newblock Adaptive methods for boundary integral equations - complexity and
  convergence estimates.
\newblock {\em Math. Comp.}, 76:1243--1274, 2007.

\bibitem[DKS02]{56.3}
W.~Dahmen, A.~Kunoth, and R.~Schneider.
\newblock Wavelet least squares methods for boundary value problems.
\newblock {\em SIAM J. Numer. Anal.}, 39(6):1985--2013, 2002.

\bibitem[DS99]{56}
W.~Dahmen and R.P. Stevenson.
\newblock Element-by-element construction of wavelets satisfying stability and
  moment conditions.
\newblock {\em {SIAM} {J}. {N}umer. {A}nal.}, 37(1):319--352, 1999.

\bibitem[DSX00]{56.2}
W.~Dahmen, R.~Schneider, and Y.~Xu.
\newblock Nonlinear functionals of wavelet expansions---adaptive reconstruction
  and fast evaluation.
\newblock {\em Numer. Math.}, 86(1):49--101, 2000.

\bibitem[Gan08]{75.365}
T.~Gantumur.
\newblock An optimal adaptive wavelet method for nonsymmetric and indefinite
  elliptic problems.
\newblock {\em J. Comput. Appl. Math.}, 211(1):90--102, 2008.

\bibitem[GHS07]{75.36}
T.~Gantumur, H.~Harbrecht, and R.P. Stevenson.
\newblock An optimal adaptive wavelet method without coarsening of the
  iterands.
\newblock {\em Math. Comp.}, 76:615--629, 2007.

\bibitem[GR79]{75.4}
V.~Girault and P.A. Raviart.
\newblock An analysis of a mixed finite element method for the
  {N}avier-{S}tokes equations.
\newblock {\em Numer. Math.}, 33:235--271, 1979.

\bibitem[NS09]{239.17}
H.~Nguyen and R.P. Stevenson.
\newblock Finite element wavelets with improved quantitative properties.
\newblock {\em J. Comput. Appl. Math.}, 230(2):706--727, 2009.

\bibitem[PR94]{243.4}
J.~Pousin and J.~Rappaz.
\newblock Consistency, stability, a priori and a posteriori errors for
  {P}etrov-{G}alerkin methods applied to nonlinear problems.
\newblock {\em Numer. Math.}, 69(2):213--231, 1994.

\bibitem[Ste98]{249.70}
R.P. Stevenson.
\newblock Stable three-point wavelet bases on general meshes.
\newblock {\em Numer. Math.}, 80:131--158, 1998.

\bibitem[Ste04]{249.81}
R.P. Stevenson.
\newblock On the compressibility of operators in wavelet coordinates.
\newblock {\em {SIAM} J. Math. Anal.}, 35(5):1110--1132, 2004.

\bibitem[Ste09]{249.92}
R.P. Stevenson.
\newblock Adaptive wavelet methods for solving operator equations: An overview.
\newblock In R.A. DeVore and A.~Kunoth, editors, {\em Multiscale, Nonlinear and
  Adaptive Approximation: Dedicated to Wolfgang Dahmen on the Occasion of his
  60th Birthday}, pages 543--598. Springer, Berlin, 2009.

\bibitem[Ste13]{249.96}
R.P. Stevenson.
\newblock {First order system least squares with inhomogeneous boundary
  conditions}.
\newblock {\em IMA. J. Numer. Anal.}, 2013.

\bibitem[Ste14]{249.94}
R.P. Stevenson.
\newblock Adaptive wavelet methods for linear and nonlinear least-squares
  problems.
\newblock {\em Found. Comput. Math.}, 14(2):237--283, 2014.

\bibitem[Urb09]{298}
K.~Urban.
\newblock {\em Wavelet Methods for Elliptic Partial Differential Equations}.
\newblock Oxford University Press, 2009.

\bibitem[Vor09]{310.5}
J.~Vorloeper.
\newblock {\em Adaptive Wavelet Methoden f\"{u}r Operator Gleichungen,
  Quantitative Analyse und Softwarekonzepte}.
\newblock PhD thesis, {RTWH} Aachen, 2009.
\newblock {VDI Verlag GmbH, D\"{u}sseldorf, ISBN 978-3-18-342720-8}.

\bibitem[XZ03]{316}
Y.~Xu and Q.~Zou.
\newblock Adaptive wavelet methods for elliptic operator equations with
  nonlinear terms.
\newblock {\em Adv. Comput. Math.}, 19(1-3):99--146, 2003.
\newblock Challenges in computational mathematics (Pohang, 2001).

\bibitem[XZ05]{316.5}
Y.~Xu and Q.~Zou.
\newblock Tree wavelet approximations with applications.
\newblock {\em Sci. China Ser. A}, 48(5):680--702, 2005.

\end{thebibliography}
\end{document}